\documentclass{article}
\usepackage{graphicx} 

\usepackage{amsmath, mathrsfs, amssymb, amsthm, fullpage, graphicx, mathtools, bbm, hyperref}
\usepackage[ruled,vlined]{algorithm2e}

\newtheorem{theorem}{Theorem}
\newtheorem{lemma}{Lemma}

\newtheorem*{example*}{Example}

\makeatletter
\def\BState{\State\hskip-\ALG@thistlm}
\makeatother

\newtheorem{corollary}{Corollary}
  {
      \theoremstyle{plain}
      
  }

\newtheorem{conjecture}{Conjecture}

\numberwithin{equation}{section}

\usepackage{caption}
\usepackage{chngcntr}

\newcommand\independent{\protect\mathpalette{\protect\independenT}{\perp}}
\def\independenT#1#2{\mathrel{\rlap{$#1#2$}\mkern2mu{#1#2}}}

\newcommand{\mme}[0]{\mathbb{E}}
\newcommand{\mmp}[0]{\mathbb{P}}
\newcommand{\mmr}[0]{\mathbb{R}}
\newcommand{\mmn}[0]{\mathbb{N}}

\newcommand{\bone}[0]{\mathbbm{1}}

\DeclarePairedDelimiterX{\inp}[2]{\langle}{\rangle}{#1, #2}

\usepackage[nameinlink, noabbrev]{cleveref}

\usepackage{multirow,float}

\usepackage{relsize}

\usepackage{accents}

\title{Characterizing the Training-Conditional Coverage of Full Conformal Inference in High Dimensions}
\author{Isaac Gibbs\footnote{Code for reproducing the experiments in this article is available at \url{https://github.com/isgibbs/high-dim-fc}.}\ \footnote{Department of Statistics, University of California, Berkeley.  Email: \href{mailto:igibbs@berkeley.edu}{igibbs@berkeley.edu}} \and Emmanuel J. Cand\`{e}s\footnote{Departments of Mathematics and Statistics, Stanford University.}}
 \date{}

\usepackage{xcolor}
\usepackage{natbib}

\allowdisplaybreaks

\begin{document}

\maketitle

 \begin{abstract}
We study the coverage properties of full conformal regression in the proportional asymptotic regime where the ratio of the dimension and the sample size converges to a constant. In this setting, existing theory tells us only that full conformal inference is unbiased, in the sense that its average coverage lies at the desired level when marginalized over both the new test point and the training data. Considerably less is known about the behaviour of these methods conditional on the training set. As a result, the exact benefits of full conformal inference over much simpler alternative methods is unclear. This paper investigates the behaviour of full conformal inference and natural uncorrected alternatives for a broad class of $L_2$-regularized linear regression models. We show that in the proportional asymptotic regime the training-conditional coverage of full conformal inference concentrates at the target value. On the other hand, simple alternatives that directly compare test and training residuals realize constant undercoverage bias. While these results demonstrate the necessity of full conformal in correcting for high-dimensional overfitting, we also show that this same methodology is redundant for the related task of tuning the regularization level. In particular, we show that full conformal inference still yields asymptotically valid coverage when the regularization level is selected using only the training set, without consideration of the test point. Simulations show that our asymptotic approximations are accurate in finite samples and can be readily extended to other popular full conformal variants, such as full conformal quantile regression and the LASSO, that do not directly meet our assumptions.
 \end{abstract}

\section{Introduction}

We consider the generic problem of forming prediction sets around the outputs of a point-prediction model. Let $\{(X_i,Y_i)\}_{i=1}^{n+1}$ be an i.i.d. dataset of covariate-response pairs where $\{(X_i,Y_i)\}_{i=1}^{n}$ and $X_{n+1}$ denote the training data and test covariates (respectively) which we must use to predict the unobserved outcome $Y_{n+1}$. Let $\mathcal{M}$ denote a model-fitting procedure which we would like to use to forecast $Y_{n+1}$. Our goal is to quantify the accuracy of models fit using $\mathcal{M}$ and apply this information to construct a prediction set $\hat{C}$ such that $\mmp(Y_{n+1} \in \hat{C}) \geq 1-\alpha$ for some target level $\alpha \in (0,1)$.

The key difficulty in this problem comes from the fact that we have access to only a single dataset, which must be used to both fit the predictive model and measure its accuracy. More precisely, suppose we are given a scoring function $S(\mathcal{M}(\mathcal{D}),X,Y)$ that measures the accuracy of model $\mathcal{M}$ fit on dataset $\mathcal{D}$ at the point $(X,Y)$. For instance, if $\mathcal{M}$ is a regression model that outputs a point-predictor, $\hat{\mu}(X)$ of the conditional mean of $Y$ given $X$ we may set $S(\mathcal{M}(D),X,Y) = |Y - \hat{\mu}(X)|$ to be the absolute residuals. Perhaps the most straightforward solution to our problem would be to train the model on the observed data, $\mathcal{D}_{\text{train}} = \{(X_i,Y_i)\}_{i=1}^n$ and directly compare the test score $S(\mathcal{M}(\mathcal{D}_{\text{train}}),X_{n+1},Y_{n+1})$ to the training scores $\{S(\mathcal{M}(\mathcal{D}_{\text{train}}),X_{i},Y_i)\}_{i=1}^n$. If $\mathcal{M}$ is a simple model (e.g. linear regression on a small number of covariates) this may be a reasonable approach. However, for even moderately complex models it is well-known that overfitting yields incomparable training and test errors. 

To address this issue, full conformal inference proposes the clever idea of including the test point in the model fit. The idea is that by fitting with all $n+1$ data points the training and test data will be treated identically and thus their scores will be directly comparable. Of course, the true value of $Y_{n+1}$ is unknown and thus cannot itself be used to construct the model. So, rather than fitting with this value exactly, full conformal inference imputes guesses in its place. More precisely, let $\mathcal{D}_y := \{(X_i,Y_i)\}_{i=1}^n \cup \{(X_{n+1},y)\}$ denote an augmented dataset containing both the training data and an imputed guess for the test point. Let $S_{n+1}^y :=  S(\mathcal{M}(\mathcal{D}_y),X_{n+1},y)$ and $S_j^y :=  S(\mathcal{M}(\mathcal{D}_y),X_j,Y_j)$, for $j \in \{1,\dots,n\}$ denote the  test and training scores and $\text{Quantile}(1-\alpha,P)$ denote the $1-\alpha$ quantile of distribution $P$. Then, full conformal inference  outputs the prediction set,
\[
\hat{C}_{\text{full}} := \left\{y : S_{n+1}^y \leq \text{Quantile}\left(1-\alpha, \frac{1}{n+1} \sum_{i=1}^{n+1} \delta_{S_i^y}  \right) \right\}.
\]
This single idea of fitting with the imputed test point turns out to be quite powerful and in the pioneering work of \citet{VovkBook} full conformal inference was shown to satisfy the marginal coverage guarantee $\mmp(Y_{n+1} \in \hat{C}_{\text{full}}) \geq 1-\alpha$ (see also \cite{Gammerman1998}, \cite{Saunders1999}, and \cite{Vovk1999} for earlier proposals of fitting with the unknown test point in a different context).

Since its introduction, the study of full conformal inference has largely focused on computational issues. In all cases, the goal is to avoid the naive computation of the prediction set in which one must fit the model for each unique value of $y \in \mmr$, or more practically, for all $y$ in some dense grid. This includes work on efficient full conformal algorithms for linear and ridge regression (\cite{Nouretdinov2001}), $k$-nearest neighbours regression (\cite{VovkBook}), the LASSO (\cite{Hebiri1010, Chen2016, Lei2019}), and quantile regression (\cite{CondConf}), among other methods (\cite{Cherubin2021, Guha2023, Papadopoulos2023, Diptesh2024}), as well as a number procedures for computing  approximations to the full conformal prediction set (\cite{Chen2018, Ndiaye2019, Fong2021, Ndiaye2022a, Ndiaye2022b, Abad2023}).

On the other hand, comparatively little has been said about the coverage properties of full conformal inference beyond the initial work of \citet{VovkBook}. Most critically, the aforementioned coverage guarantee, $\mmp(Y_{n+1} \in \hat{C}_{\text{full}}) \geq 1-\alpha$, is a marginal result that averages over both the new test point and the training data. In essence, it says that full conformal inference is unbiased. This contrasts sharply with the requirements of many practical settings in which one is often given a single training dataset from which they must construct predictions for many test points. For example, a diagnostic model fit on historical patient records is typically deployed to predict the outcomes of many new individuals. Similarly, an insurance company may set the rates of many customers based on a single model of claim rates and risk factors built on outcomes from previous years. In these scenarios, the performance of conformal inference is dictated heavily by the variability of the training-conditional coverage, $\mmp(Y_{n+1} \in \hat{C}_{\text{full}} | \{(X_i,Y_i)\}_{i=1}^n)$ around the target level. 

Here, the status of full conformal inference is quite unclear. Perhaps the first answer that comes to mind is to argue that when $n$ is large the fitted model, $\mathcal{M}(\{(X_i,Y_i)\}_{i=1}^n \cup \{(X_{n+1},y)\})$ will be close close to some fixed limit, $\mathcal{M}^*$ and thus, the full conformal set will behave as if it were constructed directly with $\mathcal{M}^*$. For instance, in the regresssion setting where $S(M(\mathcal{D}),X,Y) = |Y - \hat{\mu}(X)|$ we may argue that $\hat{\mu}$ is converging to some limit $\mu^*$ and thus,
\[
\hat{C}_{\text{full}} \approx \{y : |y - \mu^*(X_{n+1})| \leq \text{Quantile}(1-\alpha,|Y_i - \mu^*(X_i)|)\},
\]
where $\text{Quantile}(1-\alpha,|Y_i - \mu^*(X_i)|)$ denotes the $1-\alpha$ quantile of the random variable $|Y_i - \mu^*(X_i)|$. However, if this argument holds it becomes quite unclear why full conformal inference should be preferred to the much simpler approach of fitting the model using only the training set. Indeed, any argument predicated on the fact that the fitted model converges will apply equally well to both full conformal inference and this simpler approach. Thus, in light of the fact that even the most efficient implementations of full conformal inference incur substantial computational costs, this argument is extremely unsatisfactory.

Results on more complex, non-converging models are sparse. In recent work, \citet{Bian2023} give negative results demonstrating the existence of a model fitting procedure for which the training-conditional coverage of full conformal inference can deviate arbitrarily from the target level. However, this model is quite unnatural and does not correspond to any common implementation of full conformal inference. Returning to the regression setting, \citet{Liang2023} demonstrate that if the algorithm for estimating $\hat{\mu}(\cdot)$ is stable (in a sense that we will formally define below), then the training-conditional coverage of a modified version of full conformal inference can be shown to satisfy a conservative coverage guarantee. However, it is not immediately clear what model fitting procedures satisfy their stability assumption, and, most critically, how their results pertain to the performance of full conformal inference relative to more simple methods.

In this article, we show that full conformal inference offers substantial benefits in high-dimensional regression. We study the proportional asymptotic regime, in which the ratio of the covariate dimension and the sample size converges to a constant. For a large class of $L_2$-regularized regressions, we demonstrate that the training-conditional coverage of full conformal inference converges in probability to the target level, while the uncorrected method that does not include the new test point in the fit exhibits constant order bias. Simulations show that our results are tight in finite samples and accurately predict model performance even in what might typically be considered low dimensional settings (e.g.~$n=200$, $d=20$). Additionally, through a combination of heuristic calculations and empirical results, we demonstrate that our results can be extended to other popular methods, such as the full conformal LASSO and full conformal quantile regression, that do not directly meet our assumptions. 

Interestingly, we find that while the full conformal correction is critical in fitting the regression, it is not universally necessary. In particular, we show that the regularization level in our model can be freely fit using only the training data without consideration of the test point. This contrasts with a variety of prior work in conformal inference in which a heavy emphasis is placed on ensuring complete symmetry of the model-fitting procedure. On a practical level, this result is particularly interesting due to the computational complexity of accounting for the test point in model selection. Indeed, while a plethora of methods are available for computing the full conformal prediction set at fixed regularization, existing methods for hyperparameter tuning are much more limited and involve additional data splitting (\cite{Yang2024, Liang2024}). While our results are far from universal, they show that perfect symmetry can be redundant and significant computational savings can be realized by implementing the full conformal correction in only the most critical aspects of the model-fitting pipeline. 

\section{Full conformal inference for stable regression models}

Before presenting our main results, it is useful to first give a more detailed account of the work of \citet{Liang2023}. As discussed above, they consider a regression setting with score $S(\mathcal{M}(D),X,Y) := |Y - \hat{\mu}(X)|$ defined to be the absolute residuals of the fitted model $\hat{\mu}(\cdot)$. Let $\hat{\mu}_n(\cdot)$ and $\hat{\mu}_{n+1}(\cdot)$ denote the models fit on i.i.d. datasets $\{(X_i,Y_i)\}_{i=1}^n$ and $\{(X_i,Y_i)\}_{i=1}^{n+1}$. We say that the model is in-sample stable if,
\begin{equation}\label{eq:stab_condition}
\mme[|\hat{\mu}_n(X_1) - \hat{\mu}_{n+1}(X_1)|] = o(1).
\end{equation}
Here the phrase ``in-sample" denotes that we are evaluating the stability of the model predictions at one of the training points. 

To construct a prediction set from this assumption, let $\hat{\mu}_{n+1}^y(\cdot)$ denote the model fit on the imputed data set $\{(X_i,Y_i)\}_{i=1}^n \cup \{(X_{n+1},y)\}$ and $\{\rho_n\}_{n=1}^{\infty}$ denote a positive non-increasing sequence such that $\mme[|\hat{\mu}_n(X_1) - \hat{\mu}_{n+1}(X_1)|]/\rho_n \to 0$. As above, let $S_{n+1}^y := |y - \hat{\mu}^y_{n+1}(X_{n+1})|$ and $S_{i}^y := |Y_i - \hat{\mu}^y_{n+1}(X_{i})|$, for $i \in \{1,\dots,n\}$ denote the test and training scores. Then, \citet{Liang2023} define the $\rho_n$-conservative full conformal prediction set as
\begin{equation}\label{eq:adjusted_prediction_set}
\hat{C}_{\text{full}, \text{cons.}}^{\rho_n} := \left\{y : S_{n+1}^y \leq \text{Quantile}\left(1-\alpha, \frac{1}{n+1} \sum_{i=1}^{n+1} \delta_{S_i^y}   \right) + \rho_n \right\},
\end{equation}
and they obtain the following coverage guarantee.

\begin{theorem}\label{thm:stab_cov}[Theorem 3.7 of \cite{Liang2023}]
Assume that $\{(X_i,Y_i)\}_{i=1}^n$ are i.i.d.~and $\hat{\mu}(\cdot)$ is in-sample stable. Let $\{\rho_n\}_{n=1}^{\infty}$ be defined as above. Then,
\[
\mmp(Y_{n+1} \in \hat{C}_{\textup{full}, \textup{cons.}}^{\rho_n} \mid \{(X_i,Y_i)\}_{i=1}^n) \geq 1-\alpha - o_{\mmp}(1).
\]
\end{theorem}

While appealing, this result has the downsides that it is a one-sided bound that requires the use of a hyperparameter, $\rho_n$ whose range of permissible values depends on the unknown stability of the model. As a result, it is not immediately clear how one should select $\rho_n$ to obtain good performance in practice. By going beyond stability, we will demonstrate that for the high-dimensional linear regression models we consider the standard full conformal prediction set with $\rho_n = 0$ asymptotically achieves exact training-conditional coverage. Moreover, by leveraging additional information about the model fitting procedure, we will derive a much more precise characterization of the behaviour of the full conformal residuals, estimated quantile function, and coverage properties of alternative methods.

\section{Setting and assumptions} \label{sec:high_dim_setting}

In what follows we will focus on the setting in which the model fit on the imputed dataset, $\{(X_i,Y_i)\}_{i=1}^n \cup \{(X_{n+1},y)\} \subseteq \mmr^d \times \mmr$ is a linear $L_2$-regularized regression of the form,
\[
\hat{\beta}^y := \text{argmin}_{\beta \in \mmr^d} \frac{1}{n+1} \sum_{i=1}^{n} \ell(Y_i - X_i^\top \beta) + \frac{1}{n+1} \ell(y - X_{n+1}^\top\beta) + \tau \|\beta\|_2^2.
\]
Here, $\ell(\cdot)$ is a convex loss and $\tau > 0$ is a hyperparameter controlling the degree of regularization. More detailed assumptions on $\ell(\cdot)$ are given in the following sections. We will set the score to be the absolute residuals from this regression. Under these definitions, full conformal inference outputs the prediction set,
\begin{equation}\label{eq:full_set}
\hat{C}_{\text{full}} := \left\{y : |y - X_{n+1}^\top\hat{\beta}^y|  \leq \text{Quantile}\left(1-\alpha, \frac{1}{n+1} \sum_{i=1}^n \delta_{|Y_i - X_i^\top \hat{\beta}^y|} + \frac{1}{n+1}\delta_{|y - X_{n+1}^\top\hat{\beta}^y|}  \right) \right\}.
\end{equation}

Throughout, we will work in a high-dimensional elliptical model in which $d/n \to \gamma \in (0,\infty)$ and $X_i = \lambda_i W_i$ for some bounded, positive i.i.d.~random variables $(\lambda_i)_{i=1}^{n+1}$ independent from $(W_{ij})_{i=1}^{n+1}$. We will assume that $(W_{ij})_{i \in [n+1], j \in [d]}$ are i.i.d.,~mean zero, unit variance, sub-Gaussian random variables and that the response is sampled from the linear model, $Y_i = X_i^\top \beta^* + \epsilon_i$ for some i.i.d., continuous random variables $(\epsilon_i)_{i=1}^{n+1}$ independent of $(\lambda_i,W_i)_{i=1}^{n+1}$, and coefficients, $\beta^*_j$ sampled independently from the data and such that $\{\sqrt{d}\beta^*_j\}_{j=1}^d$ are i.i.d.. For convenience, we will assume throughout that $\epsilon_i$ has a bounded density and a unique $1-\alpha$ quantile. Our additional assumptions on the distributions of $\sqrt{d}\beta^*_j$ and $\epsilon_i$ will vary throughout this article and we will specify them more explicitly in the relevant sections. One common set of assumptions under which all of our results hold is the case of Gaussian errors, $\epsilon_i \sim N(0,\sigma_{\epsilon}^2)$ and  $\sqrt{d}\beta^*_j$ having bounded support. It is perhaps useful to note that the scaling of the above quantities has been chosen so that $X_i^\top \beta^* = O_{\mmp}(1)$ and thus we have a constant signal-to-noise ratio. 

While these assumptions are strong, we expect our results to go through under other common settings (e.g non-isotropic $W_i$). Our primary concern is garnering a conceptual understanding of the functionality of full conformal inference and thus we aim only to have a reasonable set of assumptions that facilitates this.

Finally, it is crucial to note that under our assumptions the estimated regression coefficients do not converge. In particular, as we will discuss shortly, under appropriate assumptions on the loss and data distributions, $\|\hat{\beta}^{Y_{n+1}} - {\beta}^*\|_2$ converges to a non-zero constant and, as a result, $\hat{\beta}^{Y_{n+1}}$ is sensitive to the specific values in the dataset. There is an extensive literature on the behaviour of coefficient estimates in high-dimensional linear regression. For some seminal work on this topic, we refer the interested reader to \cite{EK2013a, EK2013, EK2018, Donoho2013, Thram2018}, and \cite{ThramThesis}. Our results will, in particular, be built on the leave-one-out arguments developed in \cite{EK2013} and \cite{EK2018}. \\

\textbf{Additional notation:} Before presenting our results it is useful to define a few additional pieces of notation. First, note that to determine if $Y_{n+1} \in \hat{C}_{\text{full}}$ we just need to consider the case $y = Y_{n+1}$. Thus, to ease notation we will let $\hat{\beta} := \hat{\beta}^{Y_{n+1}}$ denote the regression coefficients obtained when the model is fit using the full dataset, $\{(X_i,Y_i)\}_{i=1}^{n+1}$. We let $\hat{\beta}_{(j)}$ denote the same quantity when the $j$-th sample is left out of the fit, i.e.
\[
\hat{\beta}_{(j)} := \text{argmin}_{\beta \in \mmr^d} \frac{1}{n} \sum_{i \in [n+1]-\{j\}} \ell(Y_i - X_i^\top \beta) + \tau \|\beta\|_2^2.
\]
We take $X$ to be the matrix with rows $X_1,\dots,X_{n+1}$ and $I_d$ to denote the identity matrix in $d$ dimensions. Finally, abusing notation, we let $\text{Quantile}(1-\alpha, P)$ and $\text{Quantile}(1-\alpha, Z)$ denote the $1-\alpha$ quantiles of the distribution $P$ and the random variable $Z$, respectively.

\section{Results}

\subsection{Full conformal ridge regression}\label{sec:high_dim_ridge}

We will begin by considering the case of ridge regression, i.e., the case $\ell(r) = r^2$. Here, many of the quantities we are interested in have simple closed-form expressions. As a result, this will be a convenient setting for developing the main ideas of our work. Analogous results for a broader class of robust loss functions are presented in Section \ref{sec:general_losses}.

Our first result shows that high-dimensional ridge regression satisfies the stability condition (\ref{eq:stab_condition}) of \citet{Liang2023}.
\begin{lemma}\label{lem:ridge_stab}
    Assume that $\epsilon_i$ and $\sqrt{d}\beta^*_i$ have $8$ bounded moments. Then, under the setting of Section \ref{sec:high_dim_setting} with $\ell(r) = r^2$,
    \[
    \mme[|X_1^\top\hat{\beta} - X_1^\top \hat{\beta}_{(n+1)}|] = o(1).
    \]
\end{lemma}

Following Theorem 3.6 of \citet{Liang2023} (Theorem \ref{thm:stab_cov} above), this Lemma immediately implies that the conservative full conformal ridge prediction set, (\ref{eq:adjusted_prediction_set}) satisfies a conservative training-conditional coverage guarantee.

\begin{corollary}\label{corr:ridge_stab_cov}
Under the assumptions of Lemma \ref{lem:ridge_stab}, the conclusion of Theorem \ref{thm:stab_cov} holds with ridge regression as the base predictor.
\end{corollary}

To prove Lemma \ref{lem:ridge_stab}, we perform a direct calculation on the form of the leave-one-out coefficients, $\hat{\beta}_{(n+1)}$. We will give a sketch of this argument here, leaving formal details to the Appendix. 

\begin{proof}[Proof sketch of Lemma \ref{lem:ridge_stab}]
By the optimally of $\hat{\beta}$ and $\hat{\beta}_{(n+1)}$ we have the first-order conditions
\[
 \frac{1}{n+1}\sum_{i=1}^{n+1} (X_i^\top \hat{\beta} - Y_i) X_i + \tau \hat{\beta} = 0 \qquad { \text{and} } \qquad \frac{1}{n} \sum_{i=1}^n (X_i^\top \hat{\beta}_{(n+1)} - Y_i) X_i + \tau \hat{\beta}_{(n+1)} = 0.
\]
Letting $\hat{\beta}_{\Delta} = \hat{\beta} - \hat{\beta}_{(n+1)}$, and taking the difference between the two equations above it follows that, 
\begin{align*}
0 & = \sum_{i=1}^{n+1} (X_i^\top \hat{\beta} - Y_i) X_i + (n+1)\tau \hat{\beta} -  \sum_{i=1}^n (X_i^\top \hat{\beta}_{(n+1)} - Y_i) X_i - n\tau \hat{\beta}_{(n+1)}\\
& = \sum_{i=1}^{n+1} X_iX_i^\top \hat{\beta}_{\Delta} + (X_{n+1}^\top \hat{\beta}_{(n+1)} - Y_{n+1}) X_{n+1} + \tau  \hat{\beta}_{(n+1)} + (n+1) \tau \hat{\beta}_{\Delta},
\end{align*}
and rearranging we obtain the leave-one-out representation,
\[
\hat{\beta}_{(n+1)} = \hat{\beta} - \left(\frac{1}{n+1}X^\top X + \tau I_d\right)^{-1}\left( \frac{1}{n+1}(Y_{n+1} - \hat{X}_{n+1}^\top \hat{\beta}_{(n+1)})X_{n+1} - \frac{\tau}{n+1} \hat{\beta}_{(n+1)}\right).
\]
Thus, to prove Lemma \ref{lem:ridge_stab} it is sufficient to show that, 
\begin{enumerate}
\item 
$\mme\left[\left|(Y_{n+1} - X_{n+1}^\top \hat{\beta}_{(n+1)}) \frac{1}{n+1}X_1^\top \left(\frac{1}{n+1} X^\top X + \tau I_d \right)^{-1}X_{n+1} \right|\right] = o(1),$
\item $\mme\left[\left| \frac{\tau}{n+1} X_1^\top\left(\frac{1}{n+1} X^\top X + \tau I_d \right)^{-1}\hat{\beta}_{(n+1)} \right| \right]  = o(1).$
\end{enumerate}
We will discuss the bound on the first term leaving the second claim to the Appendix. Let $X_{2:(n+1)}$ denote the matrix with rows $X_2,\dots,X_{n+1}$. Using the Sherman-Morrison-Woodbury identity we have, 
\begin{align*}
& \left(\frac{1}{n+1} X^\top X + \tau I_d \right)^{-1}   = \left(\frac{1}{n+1} X_{2:(n+1)}^\top X_{2:(n+1)} + \tau I_d \right)^{-1}\\
& \hspace{3cm} - \frac{\left(\frac{1}{n+1} X_{2:(n+1)}^\top X_{2:(n+1)} + \tau I_d \right)^{-1} \frac{1}{n+1} X_1X_1^\top \left(\frac{1}{n+1} X_{2:(n+1)}^\top X_{2:(n+1)} + \tau I_d \right)^{-1} }{1+ \frac{1}{n+1}X_1^\top\left(\frac{1}{n+1} X_{2:(n+1)}^\top X_{2:(n+1)} + \tau I_d \right)^{-1} X_1}.
\end{align*}
Now, conditioning on $\{X_i\}_{i>1}$ we have that $\frac{1}{n+1}W_1^\top \left(\frac{1}{n+1} X_{2:(n+1)}^\top X_{2:(n+1)} + \tau I_d \right)^{-1}X_{n+1}$ is a weighted mean of $n+1$ mean-zero, sub-Gaussian random variables and thus is of size $O_{L_2}(1/\sqrt{n})$. Applying this fact to the above expressions and performing some minor additional calculation proves claim 1.

\end{proof}

The leave-one-out calculations presented above can be used to obtain much more than stability. As a starting point, let us consider the behaviour of the $(n+1)$-st residual, $Y_{n+1} - X_{n+1}^\top\hat{\beta}$. Ignoring asymptotically negligible terms, the calculations above give us the representation,
\[
Y_{n+1} - X_{n+1}^\top\hat{\beta} \approx \left(1 - \frac{1}{n+1}X_{n+1}^\top \left(\frac{1}{n+1}X^\top X + \tau I_d \right)^{-1} X_{n+1} \right) (Y_{n+1} - X_{n+1}^\top \hat{\beta}_{(n+1)}).
\]
The first term in this product is closely related to the Stieltjes transform of $\frac{1}{n+1}X^\top X$. This quantity has received extensive study in the literature across a variety of contexts dating back to the foundational work of \citet{Marcenko1967}. Here, we will build in particular on the results of \citet{Rubio2011} to demonstrate its convergence. In what follows, recall that our covariates have the representation $X_{i} = \lambda_{i} W_{i}$ for some $\lambda_{i} \in \mmr$ and $W_{i} \in \mmr^d$ with i.i.d.~entries. Then, by demonstrating convergence of the requisite Stieltjes transform we obtain the following representation of the fitted residuals in terms of their leave-one-out counterparts.

\begin{lemma} \label{lem:ridge_loo_lemma}
Assume that $\epsilon_i$ and $\sqrt{d}\beta^*_i$ have $8$ bounded moments. Then, under the setting of Section \ref{sec:high_dim_setting} there exists a constant $c_{\infty} > 0$ such that
\[
(Y_{n+1} - X_{n+1}^\top\hat{\beta}) - \frac{1}{1+2\lambda^2_{n+1}c_{\infty}}(Y_{n+1} - X_{n+1}^\top \hat{\beta}_{(n+1)}) \stackrel{\mmp}{\to} 0.
\]
\end{lemma}

To fully characterize the residuals, it remains to determine the behaviour of $\hat{\beta}_{(n+1)}$. This quantity has received extensive study in the literature and a number of authors have shown that in a variety of settings the estimation error $\|\beta^* - \hat{\beta}_{(n+1)}\|_2$ converges to a constant (\cite{EK2013a, EK2013, EK2018, Donoho2013, ThramThesis, Thram2018}). Our present setting does not directly meet the assumptions of any of these works and so here we give a self contained proof of this fact for our problem.  

\begin{lemma}\label{lem:ridge_norm_conv}
    Assume that $\epsilon_i$ and $\sqrt{d}\beta^*_i$ have $8$ bounded moments. Then, under the setting of Section \ref{sec:high_dim_setting}  with $\ell(r) = r^2$, there exists a constant $N_{\infty} > 0$ such that
    \[
    \|\beta^* - \hat{\beta}_{(n+1)}\|_2 \stackrel{\mmp}{\to} N_{\infty}.
    \]
\end{lemma}

Now, recalling that $X_{n+1} = \lambda_{n+1} W_{n+1}$, the leave-one-out residual can be written as 
\[
Y_{n+1} - X_{n+1}^\top\hat{\beta}_{(n+1)} = \epsilon_{n+1} + \lambda_{n+1} W_{n+1}^\top(\beta^*-\hat{\beta}_{(n+1)}).
\]
Moreover, by Lemma \ref{lem:ridge_norm_conv}, we have that conditional on $\{(X_i,Y_i)\}_{i=1}^n$, $W_{n+1}^\top (\beta^* - \hat{\beta}_{(n+1)})$ is simply a weighted sum of i.i.d. random variables with total variance $\|\beta^* - \hat{\beta}_{(n+1)}\|_2^2 \approx N_{\infty}^2$. Thus, it is reasonable to expect that  $W_{n+1}^\top (\beta^* - \hat{\beta}_{(n+1)}) \mid \{(X_i,Y_i)\}_{i=1}^n \stackrel{D}{\approx} \mathcal{N}(0,N_{\infty}^2)$. Our next lemma verifies this fact and completes our characterization of the residuals. 

\begin{lemma}\label{lem:ridge_resid_characterization}
    Assume that $\epsilon_i$ and $\sqrt{d}\beta^*_i$ have $8$ bounded moments. Let $(\lambda, \epsilon)$ denote a copy of $(\lambda_{i},\epsilon_{i})$ independent of $Z \sim N(0,1)$. Then, for any bounded, continuous function $\psi$,
    \[
    \mme\left[\psi\left(\frac{1}{1+2\lambda^2_{n+1}c_{\infty}}\left(\epsilon_{n+1} + X_{n+1}^\top(\beta^* - \hat{\beta}_{(n+1)})\right) \right) \mid \{(X_i,Y_i)\}_{i=1}^n \right] \stackrel{\mmp}{\to} \mme\left[\psi\left( \frac{1}{1+2\lambda^2 c_{\infty}}(\epsilon + \lambda N_{\infty} Z)\right) \right].
    \]
\end{lemma}

Our study of full conformal ridge regression is nearly complete. In particular, to fully determine the asymptotic training-conditional coverage it only remains to characterize the behaviour of the empirical quantile. Lemma \ref{lem:ridge_resid_characterization} shows that the asymptotic distribution of the test residual is invariant to the training data. Our final lemma of this section strengthens this result by demonstrating that in the limit the empirical residual distribution matches this marginal limit.

\begin{lemma}\label{lem:ridge_quant_conv}
     Assume that $\epsilon_i$ and $\sqrt{d}\beta^*_i$ have $8$ bounded moments. Let $(\lambda, \epsilon)$ denote a copy of $(\lambda_{i},\epsilon_{i})$ independent of $Z \sim N(0,1)$. Then, under the setting of Section \ref{sec:high_dim_setting} with $\ell(r) = r^2$, it holds that for any bounded, continous function $\psi$,
    \[
    \frac{1}{n+1} \sum_{i=1}^{n+1} \psi(Y_i - X_i^\top\hat{\beta}) \stackrel{\mmp}{\to} \mme\left[\psi\left(\frac{1 }{1+2\lambda^2 c_{\infty}}(\epsilon + \lambda N_{\infty} Z) \right) \right].
    \]
    Moreover,
    \[
    \textup{Quantile}\left(1-\alpha, \frac{1}{n+1} \sum_{i=1}^{n+1} \delta_{| Y_i - X_i^\top\hat{\beta}|} \right) \stackrel{\mmp}{\to}  \textup{Quantile}\left(1-\alpha,  \left| \frac{1 }{1+2\lambda^2 c_{\infty}}(\epsilon + \lambda N_{\infty} Z)\right| \right).
    \]
\end{lemma}

With these preliminary results in hand, we are now ready to state the main result of this section. In particular, we show that the training-conditional coverage of full conformal ridge regression converges to the target level, while the uncorrected method that does not fit with an imputed test point fails to account for the shrinkage induced by high-dimensional overfitting and thus systematically undercovers.

\begin{theorem}\label{thm:ridge_asymp_cov}
    Assume that $\epsilon_i$ and $\sqrt{d}\beta^*_i$ have $8$ bounded moments. Then, under the setting of Section \ref{sec:high_dim_setting} with $\ell(r) = r^2$,
    \[
    \mmp\left(Y_{n+1} \in \hat{C}_{\textup{full}} \mid \{(X_i,Y_i)\}_{i=1}^n \right) \stackrel{\mmp}{\to} 1-\alpha.
    \]
    Moreover, the uncorrected prediction set, 
\[    \hat{C}_{\textup{uncorr.}} := \left\{y : |y - X_{n+1}^\top\hat{\beta}_{(n+1)}| \leq \textup{Quantile}\left(1-\alpha, \frac{1}{n} \sum_{i=1}^{n} \delta_{|Y_i - X_{i}^\top\hat{\beta}_{(n+1)}|} \right) \right\},
\]
has asymptotic undercoverage,
    \[
     \mmp\left(Y_{n+1} \in \hat{C}_{\textup{uncorr.}} \mid \{(X_i,Y_i)\}_{i=1}^n \right)  \stackrel{\mmp}{\to} \mmp\left( \left|\epsilon + \lambda N_{\infty}Z \right| \leq \textup{Quantile}\left(1-\alpha,  \left|\frac{\epsilon + \lambda N_{\infty}Z}{1+2\lambda^2c_{\infty}}\right| \right) \right) < 1-\alpha,
    \]
    where $(\lambda, \epsilon)$ denotes a copy of $(\lambda_{i},\epsilon_{i})$ independent of $Z \sim N(0,1)$.
\end{theorem}

We conclude this section with a simulation demonstrating the guarantee of Theorem \ref{thm:ridge_asymp_cov}. Data for this experiment are sampled from the Gaussian linear model given by, $W_i \sim \mathcal{N}(0,I_d)$, $\lambda_i \sim \text{Uniform}([0,1])$, $\epsilon_i \sim \mathcal{N}(0,1)$, and $Y_i = \lambda_i W_i^\top\beta^* + \epsilon_i$. In each trial of the experiment we sample $\beta^*$ from $\mathcal{N}(0,I_d/d)$ and evaluate the training-conditional coverage empirically over 2000 test points taken from the same distribution. 

\begin{figure}[ht]
    \centering 
    \includegraphics[width=\textwidth]{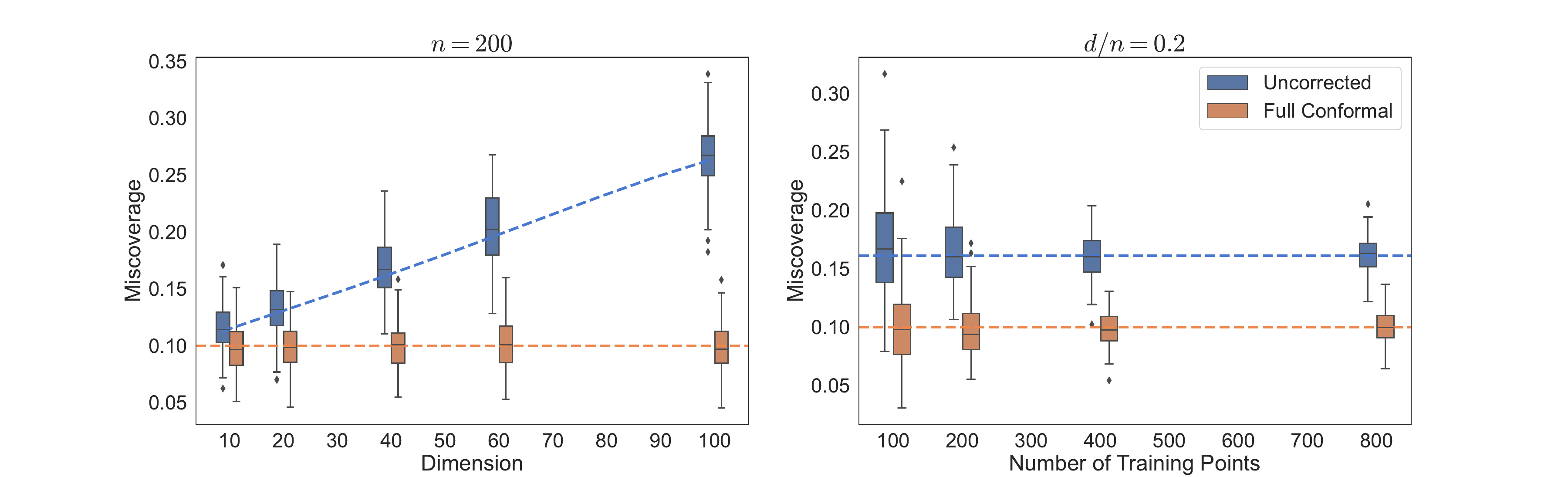}
    \caption{Empirical distribution of the training-conditional miscoverage of full conformal ridge regression (orange) and the uncorrected method (blue) that does not include the new test point in the fit. Boxplots in the figure show results from 100 trials, where in each trial the training-conditional miscoverage is evaluated empirically by averaging over 2000 test points. Dotted lines display the asymptotic miscoverages predicted by our theory. Note that since full conformal obtains the desired miscoverage asymptotically, the orange line also displays the target level of $\alpha = 0.1$.}
    \label{fig:ridge_cov}
\end{figure}

The boxplots in Figure \ref{fig:ridge_cov} display the observed distributions of the training-conditional miscoverage of full conformal ridge regression and of  the uncorrected method that does not include the new test point in the fit. Both methods are run at regularization level $\tau = 0.1$.\footnote{This value was chosen using data independent from the data used in the our evaluation of full conformal inference based on the fact that it gave small test error across all sample sizes and dimensions.} As the left panel shows, full conformal inference is extremely robust to the dimension, obtaining nearly exact training-conditional coverage for all values of $d$. This contrasts sharply with the uncorrected method, which displays undercoverage throughout. Notably, the improved coverage of full conformal inference is not limited to the high-dimensional examples and can be seen clearly at even what might typically be considered moderate levels of $d/n$, e.g., $n=200$, $d \in \{10,20\}$. Finally, we find that the convergence of the training-conditional coverage is quite rapid and the coverage of both methods appears tightly concentrated around their asymptotic values at a modest sample size of $n=800$ (right panel).

\subsection{Extension to generic loss functions}\label{sec:general_losses}

In this section, we show that our results on full conformal ridge regression can be readily extended to a more generic class of loss functions. More precisely, we will study a class of convex losses with three bounded derivatives. Interest in these functions comes from their application in high-dimensional robust regression and this class includes, for example, smooth approximations of the Huber function (see Section 2.3.5 of \cite{EK2018} for details). Our proofs will make extensive use of the work of \citet{EK2018}, which gives asymptotic characterizations of the residuals and leave-one-out residuals in this context that are analogous to our results for ridge regression from the previous section. 

While our assumptions on the loss may be strong, we expect our results to readily extend to other common contexts. For example, \citet{EK2013} gives a similar asymptotic characterization of the residuals under the assumption that the first two derivatives of $\ell(\cdot)$ grow polynomially, i.e. for $|x|$ large, $|\ell'(x)|, \ell''(x) \leq O(|x|^m)$. Thus, we expect that our results can be extended to this regime as well. Unfortunately, rigorously proving such an extension requires significant technical effort that we do not undertake here.

Before stating our results in detail, we first give a precise description of the assumptions of this section.\\

\noindent \textbf{Assumptions:}
\begin{enumerate}
    \item 
    The loss is convex with three bounded derivatives. Moreover, for all $x$, $\text{sign}(\ell'
    (x)) = \text{sign}(x)$, $\ell(x) \geq \ell(0) = 0$, and there exists $\delta > 0$ such that $\ell''(r) > 0$  for all $0 < |r| < \delta$.
    \item 
    The covariates are distributed as $X_i = \lambda_i W_i$, where $\lambda_i \stackrel{i.i.d}{\sim} P_{\lambda}$, $W_{ij} \stackrel{i.i.d.}{\sim} P_W$, and $\{\lambda_i\}_{i=1}^{n+1}$ is independent of $\{W_i\}_{i=1}^{n+1}$. Moreover, $P_{\lambda}$ has a bounded support and $\mmp(\lambda > 0) = 1$. Finally, $P_W$ is sub-Gaussian, mean zero, unit variance, and has the property that there exists constants $c_1, c_2, c_3 \geq 0$, such that for any $k \in \mmn$ with $k>1$ and any 1-Lipschitz, convex function $f : \mmr^k \to \mmr$,
    \begin{equation}\label{eq:lip_tail_condition}
    \mmp_{\tilde{w}_1,\dots,\tilde{w}_k \stackrel{i.i.d.}{\sim} P_W}(\left|f(\tilde{w}_1,\dots,\tilde{w}_k) - m_f \right| \geq t) \leq c_1 \exp\left(-\frac{c_2 t^2}{\log(k)^{c_3}}\right),\ \forall t >0,
    \end{equation}
    where $m_f$ denotes a median of $f(\tilde{w}_1,\dots,\tilde{w}_k)$.
    \item 
    The response is distributed as $Y_i = X_i^\top \beta^* + \epsilon_i$, where $\epsilon_{i} \stackrel{i.i.d.}{\sim} P_{\epsilon}$ and $\beta^*$ is a random vector sampled as $\sqrt{d}\beta^*_i \stackrel{i.i.d.}{\sim} P_{\beta}$. Moreover, $\beta^*$, $\{\epsilon_i\}_{i=1}^{n+1}$, $\{\lambda_i\}_{i=1}^{n+1}$, and $\{W_i\}_{i=1}^{n+1}$ are jointly independent, $P_{\beta}$ has bounded support, and $P_{\epsilon}$ is symmetric and unimodal with a bounded, positive, differentiable density $f_{\epsilon}$ satisfying $\lim_{|x| \to \infty} xf_{\epsilon}(x) = 0$.
\end{enumerate}

While our assumptions on the distribution of $P_W$ may appear somewhat unusual, we note that these conditions are met by many common distributions. For example, our assumption (\ref{eq:lip_tail_condition}) is met if $W_{ij}$ is normally distributed or has a bounded support. Further discussion of this condition can be found in \citet{EK2018} and the monograph \citet{Ledoux2001}. 

Now, under Assumptions 1-3, Theorem C.6 of \citet{EK2018} shows that $L_2$-regularized regression with loss $\ell(\cdot)$ meets the stability condition of \citet{Liang2023} (\ref{eq:stab_condition}). Thus, we may immediately conclude that the conservative prediction set (\ref{eq:adjusted_prediction_set}) achieves an asymptotic training conditional coverage of at least $1-\alpha$. 

Once again, by going beyond stability we can attain a much more precise characterization of the behaviour of full conformal inference. Recall that for any function $f$, the proximal map is the function $\text{prox}(f) : \mmr \to \mmr$ defined by
\[
\text{prox}(f)(x) := \text{argmin}_v f(v) + \frac{1}{2}(x-v)^2.
\]
Then, leveraging the results of \citet{EK2013}, we obtain the following asymptotic characterization of the residuals. 
\begin{lemma}\label{lem:gen_resid_char}
Under Assumptions 1-3 above, there exists a constant $c_{\infty} > 0$ such that
\[
(Y_{n+1} - X_{n+1}^\top \hat{\beta}) - \textup{prox}(\lambda_{n+1}^2 c_{\infty} \ell)\left(Y_{n+1} - X_{n+1}^\top \hat{\beta}_{(n+1)}\right) \stackrel{\mmp}{\to} 0.
\]
\end{lemma}

Now, identical to our results for ridge regression, it is also shown in \citet{EK2018} that in this setting the estimation error $\|\beta^* - \hat{\beta}_{(n+1)}\|_2$ converges asymptotically to a constant, $N_{\infty} > 0$. Combining this with the result of Lemma \ref{lem:gen_resid_char}, and arguing as the previous section, \citet{EK2018} shows that asymptotically the residuals are approximately distributed as
\[
(Y_{n+1} - X_{n+1}^\top \hat{\beta}) \stackrel{D}{\approx} \textup{prox}(\lambda^2 c_{\infty} \ell)\left(\epsilon + \lambda N_{\infty} Z\right),
\]
where $(\lambda, \epsilon)$ denotes a draw of $(\lambda_i,\epsilon_i)$ independent of $Z \sim \mathcal{N}(0,1)$. 

As in the previous section, to fully characterize the behaviour of full conformal inference it only remains to control the empirical quantile of the residuals. Our final lemma does exactly this.

\begin{lemma}\label{lem:gen_quant_conv}
    Under Assumptions 1-3 above,
    \[
    \textup{Quantile}\left(1-\alpha, \frac{1}{n+1} \sum_{i=1}^{n+1} \delta_{|Y_i - X_i^\top \hat{\beta}|} \right) \stackrel{\mmp}{\to} \textup{Quantile}\left(1-\alpha, |\textup{prox}(\lambda^2 c_{\infty} \ell)\left(\epsilon + \lambda N_{\infty} Z\right)|\right),
    \]
    where $(\lambda, \epsilon)$ denotes a draw of $(\lambda_i,\epsilon_i)$ independent of $Z \sim \mathcal{N}(0,1)$.
\end{lemma}

The above lemmas are already sufficient to show that the training-conditional coverage of full conformal inference converges to the target level. To understand the behaviour of the uncorrected method that does not fit with the imputed test point, note that the proximal function performs shrinkage. Namely, a short computation (Lemma \ref{lem:cont_diff_prox} in the Appendix) shows that 
\[
\frac{d}{dx} \text{prox}(\lambda_{n+1}^2 c_{\infty} \ell)(x) = \frac{1}{1+\lambda_{n+1}^2 c_{\infty} \ell''(\text{prox}(\lambda_{n+1}^2 c_{\infty} \ell)(x))} < 1.
\]
Combining this with Lemma \ref{lem:gen_resid_char} above, we conclude that in the limit, the leave-one-out residuals are larger (in absolute value) than the full test residuals and thus, the uncorrected method will undercover. These results are summarized in the following theorem.

\begin{theorem}\label{thm:conv_tau_fixed}
    Under Assumptions 1-3, the training-conditional coverage of full conformal inference converges in probability to the nominal level. Namely,
    \[
    \mmp\left(Y_{n+1} \in \hat{C}_{\textup{full}} \mid \{(X_i,Y_i)\}_{i=1}^n \right) \stackrel{\mmp}{\to} 1-\alpha.
    \]
    Moreover, the uncorrected prediction set,
    \begin{equation}\label{eq:uncorr_set}
    \hat{C}_{\textup{uncorr.}} := \left\{y : |y - X_{n+1}^\top\hat{\beta}_{(n+1)}| \leq \textup{Quantile}\left(1-\alpha, \frac{1}{n} \sum_{i=1}^{n} \delta_{|Y_i - X_{i}^\top\hat{\beta}_{(n+1)}|} \right) \right\},
    \end{equation}
    has asymptotic undercoverage,
    \begin{align*}
     \mmp\left(Y_{n+1} \in \hat{C}_{\textup{uncorr.}} \mid \{(X_i,Y_i)\}_{i=1}^n \right) & \stackrel{\mmp}{\to} \mmp\left( \left|\epsilon + \lambda N_{\infty}Z \right| \leq \textup{Quantile}\left(1-\alpha,  \left|\textup{prox}(\lambda^2c_{\infty}\ell)(\epsilon + \lambda N_{\infty}Z)\right| \right) \right) \\
     & < 1-\alpha,
    \end{align*}
    where $(\lambda, \epsilon)$ denotes a copy of $(\lambda_{i},\epsilon_{i})$ independent of $Z \sim N(0,1)$.
\end{theorem}

\subsection{Coverage with data-dependent regularization}\label{sec:tau_random}

The results from the previous sections require the regularization level to be fixed in advance of seeing the data. In practice, this is rarely what occurs and it is much more common for $\tau$ to be fit using cross-validation or another data-dependent procedure. The following result shows that full conformal inference seamlessly adapts to this setting and produces the desired coverage level regardless of the choice of regularization. Notably, this result does not require that the regularization level is stable or converging to any fixed optimum. Instead, we allow $\tau$ to be chosen arbitrarily with the small caveats that it cannot depend on the test point, $(X_{n+1},Y_{n+1})$ and it must stay bounded away from zero.

\begin{theorem}\label{thm:conv_with_cv}
    Suppose that the assumptions of either Theorem \ref{thm:ridge_asymp_cov} or Theorem \ref{thm:conv_tau_fixed} hold. Consider running full conformal prediction with random regularization level $ \tau_{\text{rand.}} = \tau_{\text{rand.}}(\{(X_i,Y_i)\}_{i=1}^n) \independent (X_{n+1},Y_{n+1})$ chosen dependent solely on the training data and with the property that there exists $\tau_0 > 0$ such that $\mmp(\tau_{\text{rand.}} \geq \tau_0) \to 1$. Then,
    \[
    \mmp\left(Y_{n+1} \in \hat{C}_{\textup{full}} \mid \{(X_i,Y_i)\}_{i=1}^n \right) \stackrel{\mmp}{\to} 1-\alpha.
    \] 
\end{theorem}

The proof of Theorem \ref{thm:conv_with_cv} is conceptually quite similar to that of Theorems \ref{thm:ridge_asymp_cov} and \ref{thm:conv_tau_fixed} above. In particular, we show that the asymptotic approximations of the residuals and quantile function that we obtained in the previous section for fixed levels of the regularization are in fact accurate uniformly over $\tau \geq \tau_0$. The details of this argument are somewhat delicate and we defer a complete discussion to the Appendix.

\begin{figure}[ht]
    \centering
\includegraphics[width=\textwidth]{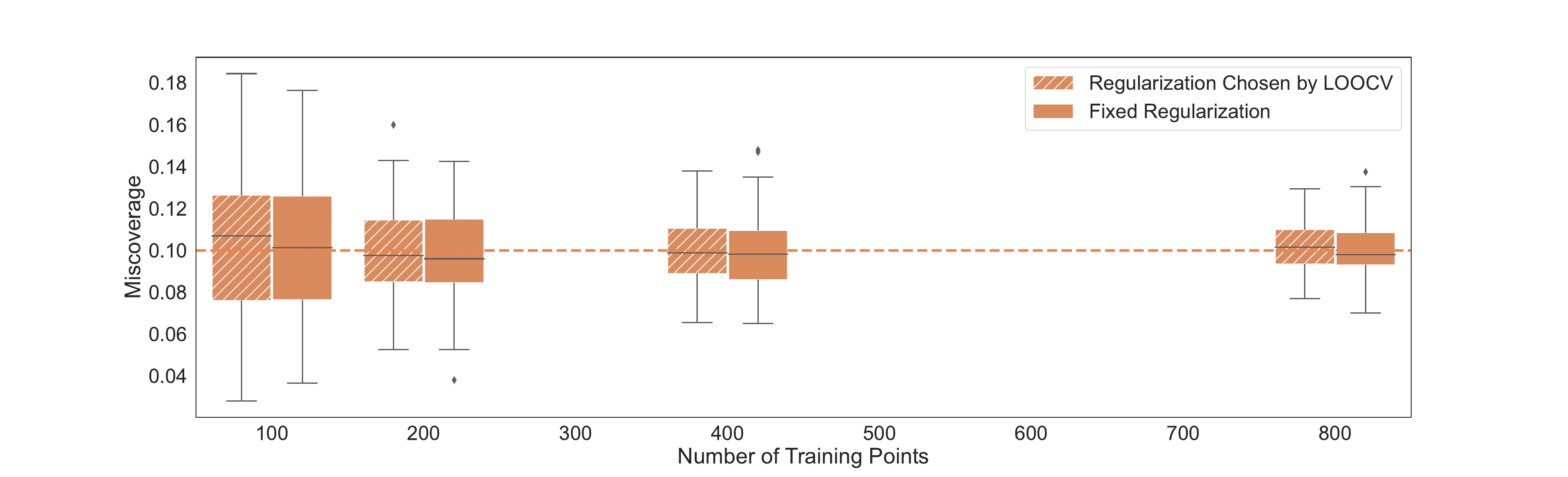}
    \caption{Empirical distribution of the training-conditional miscoverage of full conformal ridge regression with constant regularization level $\tau = 0.1$ (solid) and $\tau$ chosen using leave-one-out cross validation (striped) from the set $\{0,0.01,0.02,\dots,2\}$. Boxplots in the figure show results from 100 trials where in each trial the training-conditional coverage is evaluated empirically by averaging over 2000 test points. The orange dotted line shows the target level of $\alpha = 0.1$. For all samples sizes the dimension is set so that $d/n=0.25$ and data is generated as in Figure \ref{fig:ridge_cov}.}
    \label{fig:cv_validation}
\end{figure}

The utility of this result is shown in Figure \ref{fig:cv_validation}. Here we compare the coverage of full conformal ridge regression with $\tau$ chosen using leave-one-out cross-validation on the training data against the coverage of the same method with $\tau$ fixed at $0.1$. As predicted by our theory, we find that the training-conditional coverage of both methods converges to $1-\alpha$. Moreover, the rate of this convergence appears to be approximately the same for both methods. Thus, in this setting selection of the regularization level can essentially be done ``for free" using just the training data and without a full conformal correction. As discussed in the introduction, this contrasts with prior work on conformal inference in which a heavy emphasis is typically placed on treating the training and test data symmetrically across the entire model-fitting pipeline.

\section{Other applications of full conformal regression}

Beyond $L_2$-regularized regression, a variety of authors have considered full conformal prediction sets arising from alternative prediction algorithms (e.g.~\cite{VovkBook, Guha2023, Papadopoulos2023, Diptesh2024}). Two particularly popular prediction methods for which full conformal has been studied are the LASSO (\cite{Hebiri1010, Chen2016, Lei2019}) and quantile regression (\cite{CondConf}). The lack of differentiability and, in the case of quantile regression, strong convexity raises significant technical challenges to obtaining fully rigorous descriptions of the high-dimensional behaviour of these methods. Nevertheless, heuristic calculations suggest that the convergence results given in the previous sections can be extended to these settings. 

In this section, we give an informal treatment of full conformal prediction sets for the LASSO and quantile regression. Building on our previous theory, we derive formulas predicting the high dimensional behaviour of the fitted residuals, coefficients, and quantile function as well as the asymptotic training-conditional coverage of full conformal inference and its uncorrected alternative. Using simulated experiments, we then demonstrate that the predictions of our theory are extremely accurate at even moderate sample sizes and dimensions. Overall, these results demonstrate the generality of our approach and its potential applicability to a variety of other regression methods.

\subsection{The full conformal LASSO}

We begin by characterizing the behaviour of the full conformal LASSO. Our results will rely heavily on heuristic calculations performed in \citet{Bean2012} which suggest that the high-level conclusions of Sections \ref{sec:high_dim_ridge} and \ref{sec:general_losses} above also hold in this setting. More formally, let
\[
\hat{\beta} := \text{argmin}_{\beta \in \mmr^d} \frac{1}{n+1} \sum_{i=1}^{n+1} (Y_i - X_i^\top\beta)^2 + \frac{\tau}{\sqrt{n}} \|\beta\|_1
\]
denote the LASSO estimator fit on all $n+1$ datapoints, and 
\[
\hat{\beta}_{(n+1)} := \text{argmin}_{\beta \in \mmr^d} \frac{1}{n} \sum_{i=1}^{n} (Y_i - X_i^\top\beta)^2 + \frac{\tau}{\sqrt{n}}  \|\beta\|_1
\]
denote the same estimate fit on just the training data. Note that in our setting the true population regression coefficients obey $\|\beta^*\|_1 = O_{\mmp}(\sqrt{n})$ and thus we have normalized the penalty term to ensure that $\frac{\tau}{\sqrt{n}}  \|\hat{\beta}\|_1 = O_{\mmp}(1)$. 

Using non-rigorous calculations and simulated experiments \citet{Bean2012} provides evidence for the following conjecture.

\begin{conjecture}\label{conj:lasso_asymptotics}[Informal results of \citet{Bean2012}]
Under the assumptions on the data-generating process given in Section \ref{sec:general_losses}, there exist asymptotic constants $N_{\infty}$, $c_{\infty}$ such that
\[
 Y_i - X_i^\top\hat{\beta} - \frac{1}{1+2\lambda_{n+1}^2c_{\infty}}(Y_{n+1}-X_{n+1}^\top\hat{\beta}_{(n+1)}) \stackrel{\mmp}{\to} 0,
\]
and 
\[
\|\hat{\beta} - \hat{\beta}_{(n+1)}\|_2 \stackrel{\mmp}{\to} N_{\infty}.
\]
Moreover, $N_{\infty}$ and $c_{\infty}$ can be computed exactly as solutions to a system of equations that depends only on the distributions of $\epsilon_i$, $\lambda_i$, and $\beta^*_i$ as well as the regularization level, $\tau$ and dimension-to-sample size ratio, $\gamma = \lim_{n,d\to\infty} d/n$.\footnote{A detailed description of this system can be found in Section 3.2.3 of \citet{Bean2012}.}
\end{conjecture}

Following this conjecture, we expect both full conformal inference and the uncorrected prediction set to exhibit the same behaviour for the LASSO as we observed for ridge regression. Namely, letting $\hat{C}_{\textup{full}}$ and $\hat{C}_{\textup{uncorr.}}$ denote these sets (i.e., the analogs of (\ref{eq:full_set}) and (\ref{eq:uncorr_set}) for the LASSO) we conjecture that $\mmp(Y_{n+1} \in \hat{C}_{\textup{full}} \mid \{(X_i,Y_i)\}_{i=1}^n) \stackrel{\mmp}{\to} 1-\alpha$, while 
\begin{equation}\label{eq:lasso_asymp_cov}
\mmp(Y_{n+1} \in \hat{C}_{\textup{uncorr.}} \mid \{(X_i,Y_i)\}_{i=1}^n) \stackrel{\mmp}{\to} \mmp\left( |\epsilon + \lambda N_{\infty}Z| \leq \textup{Quantile}\left(1-\alpha, \left|\frac{\epsilon + \lambda N_{\infty}Z}{1+2\lambda^2c_{\infty}}\right| \right) \right),
\end{equation}
where $(\lambda, \epsilon)$ denotes a copy of $(\lambda_{i},\epsilon_{i})$ independent of $Z \sim N(0,1)$.

\begin{figure}[ht]
    \centering
\includegraphics[width=\textwidth]{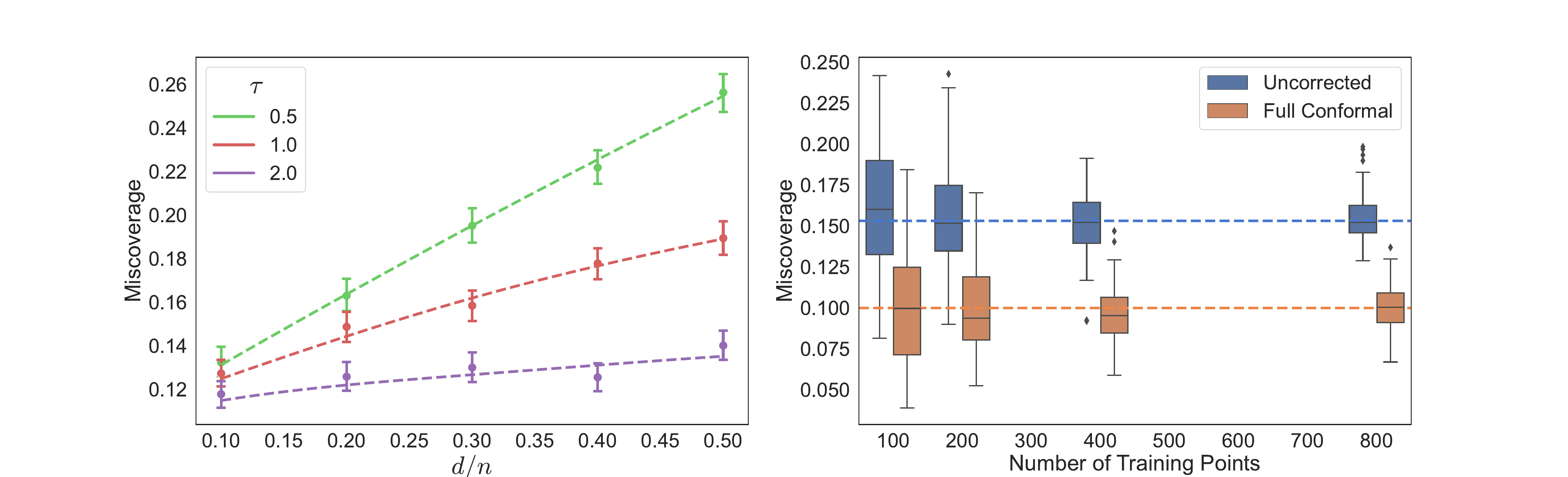}
    \caption{Empirical evaluation of our conjectures for the high-dimensional LASSO. Dots and error bars in the left panel display point estimates and 95\% confidence intervals for $\mmp(Y_{n+1} \notin \hat{C}_{\textup{uncorr.}})$ obtained by averaging over 10000 trials at $n=800$, while dotted lines show the values predicted by our theory (namely the value appearing on the right-hand side of (\ref{eq:lasso_asymp_cov})). Boxplots in the right panel show the empirical distribution of the training-conditional miscoverage of the full conformal LASSO (orange) and its uncorrected variant (blue) taken over $100$ trials where in each trial the training-conditional miscoverage is evaluated empirically by averaging over 2000 test points. Results in this panel are for $\tau = 1$ and $d/n = 0.25$. Once again dotted lines display the asymptotic miscoverages predicted by our theory. Note that since full conformal obtains the desired miscoverage asymptotically, the orange line also displays the target level of $\alpha = 0.1$ used throughout. Data for both panels were generated as in Figure \ref{fig:ridge_cov}.}
    \label{fig:lasso_validation}
\end{figure}

Empirical evidence supporting these conclusions is shown in Figure \ref{fig:lasso_validation}. The left-hand panel compares estimates of $\mmp(Y_{n+1} \notin \hat{C}_{\textup{uncorr.}})$ (dots and error bars) against the values predicted by our theory (equation (\ref{eq:lasso_asymp_cov}), dotted lines), while the right-hand panel displays boxplots demonstrating the convergence of the empirical distribution of the training-conditional miscoverage. We find that the predictions of our theory are highly accurate across all sample sizes, dimensions, and regularization levels.

\subsection{Full conformal quantile regression}

The final method we will consider is the full conformal quantile regression procedure of \citet{CondConf}. Analogous to the previous sections, prior work has shown that standard quantile regression suffers from a systematic undercoverage bias in high-dimensions (\cite{Bai2021}), which can be removed using a full conformal correction (\cite{CondConf}). Using the theory developed in the previous sections, here we provide heuristic calculations and simulated experiments demonstrating the mechanism through which full conformal inference succeeds in this problem. 

The construction of the full conformal prediction set for quantile regression is somewhat different from the methods considered in the previous section. To understand this procedure, let $\ell_{\alpha}(t) = \max\{t,0\} - \alpha t $ denote the pinball loss at level $\alpha$. As above, let 
\begin{equation}\label{eq:QR}
(\hat{\beta}_0^y,\hat{\beta}^y) \in \textup{argmin}_{(\beta_0,\beta) \in \mmr^{d+1}} \frac{1}{n+1} \sum_{i=1}^n \ell_{\alpha}(Y_i - \beta_0 -  X_i^\top \beta) + \frac{1}{n+1}\ell_{\alpha}(y - \beta_0 - X_{n+1}^\top\beta),
\end{equation}
denote the coefficients fit using quantile regression with the imputed test point $(X_{n+1},y)$. Then, we will begin by considering the set $\hat{C}^{\textup{QR}} := \{y : y \leq \hat{\beta}_0^y + X_{n+1}^\top \hat{\beta}^y\}$.

We make three remarks about this set. First, unlike the full conformal predictions sets considered in the previous sections, $\hat{C}^{\textup{QR}}$ is one-sided and may extend to $-\infty$. Two-sided analogs of this set can be obtained using a nearly identical procedure in which separate quantile regressions are used for the upper and lower bounds. To streamline our presentation, we omit these details. Second, $\hat{C}^{\textup{QR}}$ does not utilize the empirical quantile of the fitted residuals. Letting $R_i = Y_i - \hat{\beta}_0^y - X_i^\top \hat{\beta}^y$ and $R_{n+1}^y = y - \hat{\beta}_0^y - X_{n+1}^\top \hat{\beta}^y$ denote the fitted residuals, one can check that the first-order condition for the intercept in (\ref{eq:QR}) implies that the $1-\alpha$ quantile of $\frac{1}{n+1} \sum_{i=1}^{n+1} \delta_{R_i^y}$ is equal to zero.\footnote{Note that since $\hat{C}^{\textup{QR}}$ is one-sided we use the signed residuals here, not the absolute value.} Thus, in our definition of $\hat{C}^{\text{QR}}$ we simply plug-in this value. Third, our quantile regression includes an intercept term. Handling an intercept term in high-dimensions creates additional technical issues that we have chosen to avoid in the previous sections. Here, we include it because its presence is necessary to guarantee coverage (recall the aforementioned first-order condition or see Theorem 2 of \citet{CondConf} for further details).

Now, unlike the full conformal methods proposed in the previous sets, the marginal coverage of $\hat{C}^\textup{QR}$ is not unbiased. This is a result of the fact that similar to the $L_1$ penalty in the LASSO, the lack of curvature in the pinball loss at the origin pushes many of the residuals to be identically zero. In the high-dimensional setting, this creates a point-mass in the asymptotic residual distribution and overcoverage in the full conformal set. 

To overcome this shortcoming, we now define an alternative prediction set based on a dual representation of the quantile regression. First, note that by introducing the primal variables $v \in \mmr^{n+1}$ and constraints $v_i = Y_i - \beta_0 - X_i^\top\beta$ for $1 \leq i \leq n$ and $v_{n+1} = y - \beta_0 - X_{n+1}^\top\beta$, the quantile regression (\ref{eq:QR}) can be rewritten as  
\begin{align*}
\underset{(\beta_0,\beta) \in \mmr^{d+1}, v \in \mmr^{n+1}}{\text{minimize}} \quad &  
  \frac{1}{n+1}\sum_{i=1}^{n+1} \ell_{\alpha}(v_i)\\
  \text{subject to } \ \ \ \  \quad &  v_i = Y_i - \beta_0 - X_i^\top\beta,\  \forall 1 \leq i \leq n,\\
& v_{n+1} = y - \beta_0 - X_{n+1}^\top\beta.
\end{align*}
Then, letting $\eta \in \mmr^{n+1}$
denote the dual variables for these equality constraints, standard results in convex optimization tell us that primal-dual solutions of (\ref{eq:QR}) can be obtained as saddle points to the min-max problem
\begin{equation}\label{eq:qr_min_max}
\min_{(\beta_0,\beta) \in \mmr^{d+1}, v \in \mmr^{n+1}} \max_{\eta \in \mmr^{n+1}} \sum_{i=1}^{n+1} \ell_{\alpha}(v_i) + \sum_{i=1}^n \eta_i(Y_i - \beta_0 - X_i^\top\beta - v_i) + \eta_{n+1}(y - \beta_0 - X_{n+1}^\top\beta - v_{n+1}).
\end{equation}
Now, the critical observation is that the solution for $\eta$ in this program is tightly linked to the value of the residuals. More precisely, if $(\hat{\beta}^y_0,\hat{\beta}^y,\hat{v}^y,\hat{\eta}^y)$ is a saddle point of (\ref{eq:qr_min_max}, one can show that for all $1 \leq i \leq n+1$, $\hat{\eta}^y_i \in [-\alpha,1-\alpha]$ and that $\hat{\eta}^y_{n+1}$ can be directly related to the $(n+1)-$st residual through the conditions:
\begin{equation}\label{eq:eta_first_order}
\eta^y_{n+1} \in \begin{cases}
& \{1-\alpha\}, \text{ if } y > \hat{\beta}^y_0 + X_{n+1}^\top\hat{\beta}^y,\\
& [-\alpha,1-\alpha], \text{ if } y =  \hat{\beta}^y_0 + X_{n+1}^\top\hat{\beta}^y,\\
& \{-\alpha\}, \text{ if } y <  \hat{\beta}^y_0 + X_{n+1}^\top\hat{\beta}^y.
\end{cases}
\end{equation}
These observations, along with some additional calculations, lead us to the prediction set $\hat{C}^{\textup{QR}}_{\textup{dual}} = \{y : \hat{\eta}_{n+1}^y < U\}$ with $U \sim \text{Uniform}(-\alpha,1-\alpha)$ generated independently from $\{(X_i,Y_i)\}_{i=1}^{n+1}$. As anticipated, this dual-derived set is perfectly unbiased and realizes exact marginal coverage.
\begin{theorem}[Special case of Proposition 4 of \citet{CondConf}]
    Assume that $\{(X_i,Y_i)\}_{i=1}^{n+1}$ are independent and identically distributed. Then, $\mmp(Y_{n+1} \in \hat{C}^{\textup{QR}}_{\textup{dual}}) = 1-\alpha$.
\end{theorem}

To understand the behaviour of this set in the high-dimensional regime, we proceed as follows. Let $(\hat{\beta}_0,\hat{\beta},\hat{v},\hat{\eta})$ denote a saddle point of (\ref{eq:qr_min_max}) when $y = Y_{n+1}$. Recalling the constraint $\hat{v}_{n+1} = Y_{n+1} - \hat{\beta}_0 - X_{n+1}^\top\hat{\beta}$ and applying the first order condition for $\hat{v}_{n+1}$ gives us that $
\hat{\eta}_{n+1} \in \partial \ell_{\alpha}(Y_{n+1} - \hat{\beta}_0 - X_{n+1}^\top\hat{\beta})$ is a subgradient of the pinball loss (i.e.~(\ref{eq:eta_first_order}) above with $y = Y_{n+1}$). Let $(\hat{\beta}_{0,(n+1)}, \hat{\beta}_{(n+1)})$ denote the quantile regression coefficients fit on just the training data $\{(X_i,Y_i)\}_{i=1}^n$. For ease of notation, let $R_{n+1}^{(n+1)} = Y_{n+1} - \hat{\beta}_{0,(n+1)} - X_{n+1}^\top\hat{\beta}_{(n+1)} $ denote the leave-one-out test residual. Then, following the previous sections, we expect 
\[
\hat{\eta}_{n+1} \approx \partial \ell_{\alpha}(\textup{prox}(\lambda_{n+1}^2c_{\infty}\ell_{\alpha})(R_{n+1}^{(n+1)})),
\]
for some constant $c_{\infty} > 0$. Now, by definition the proximal function also satisfies the equation,
\[
\frac{x - \textup{prox}(c\ell_{\alpha})(x)}{c} \in \partial \ell_{\alpha}(\textup{prox}(c\ell_{\alpha})(x)),
\]
for all $c > 0$ and $x \in \mmr$. If $\ell_{\alpha}$ were differentiable this would give us the approximation,
\[
\hat{\eta}_{n+1} \approx \frac{R_{n+1}^{(n+1)} - \textup{prox}(\lambda_{n+1}^2c_{\infty}\ell_{\alpha})(R_{n+1}^{(n+1)})}{\lambda_{n+1}^2c_{\infty}}.
\]
Since $\ell_{\alpha}$ is not differentiable everywhere this argument does not directly go through at $R_{n+1}^{(n+1)} = 0$. Nevertheless, the final approximation for $\hat{\eta}_{n+1}$ does not involve any derivatives of $\ell_{\alpha}$ and we conjecture that by, e.g., taking a smooth approximation for $\ell_{\alpha}$ one can expect this final expression to hold. So, putting this all together, we conjecture that asymptotically, 
\[
Y_{n+1} \in \hat{C}^{\textup{QR}}_{\textup{dual}} \iff \frac{ R_{n+1}^{(n+1)} - \textup{prox}(\lambda_{n+1}^2c_{\infty}\ell_{\alpha})(R_{n+1}^{(n+1)})}{\lambda_{n+1}^2c_{\infty}} < U.
\]

To conclude our characterization, it only remains to determine the behaviour of the leave-one-out residuals. As in the previous sections, we expect that asymptotically, $\|\hat{\beta}_{(n+1)} - \beta^*\|_2 \stackrel{\mmp}{\to} N_{\infty}$ for some constant $N_{\infty} > 0$. Moreover, previous work on high-dimensional quantile regression shows that when the covariates are Gaussian the intercept is asymptotically deterministic, i.e., $\hat{\beta}_{0,(n+1)} \stackrel{\mmp}{\to} \beta_{0,\infty} \in \mmr$ (\cite{Bai2021}). We expect the same conclusion to hold for the more general data-generating distribution considered here. Thus, overall, we expect that $R_{n+1}^{(n+1)} \stackrel{D}{\approx} \epsilon - \beta_{0,\infty} + \lambda N_{\infty} Z$ with $(Z,\lambda,\epsilon) \sim N(0,1) \otimes P_{\lambda} \otimes P_{\epsilon}$.

In Section \ref{sec:app_qr_calcs} of the Appendix we provide heuristic calculations justifying a system of three equations for the values of $(c_{\infty}, N_{\infty}, \beta_{0,\infty})$. Most critically, these equations imply that
\[
\mmp\left( \frac{  \epsilon - \beta_{0,\infty} + \lambda N_{\infty} Z - \textup{prox}(\lambda_{n+1}^2c_{\infty}\ell_{\alpha})( \epsilon - \beta_{0,\infty} + \lambda N_{\infty} Z)}{\lambda_{n+1}^2c_{\infty}} < U\right) = 1-\alpha,
\]
and thus informally suggest that $\mmp(Y_{n+1} \in \hat{C}^{\text{QR}}_{\text{dual}} \mid \{(X_i,Y_i)\}_{i=1}^n) \stackrel{\mmp}{\to} 1-\alpha$. Finally, combining these equations with the previous approximations also gives us an exact formula for the high-dimensional coverage bias of standard quantile regression. These results are summarized in the following conjecture.

\begin{figure}[ht]
    \centering
\includegraphics[width=\textwidth]{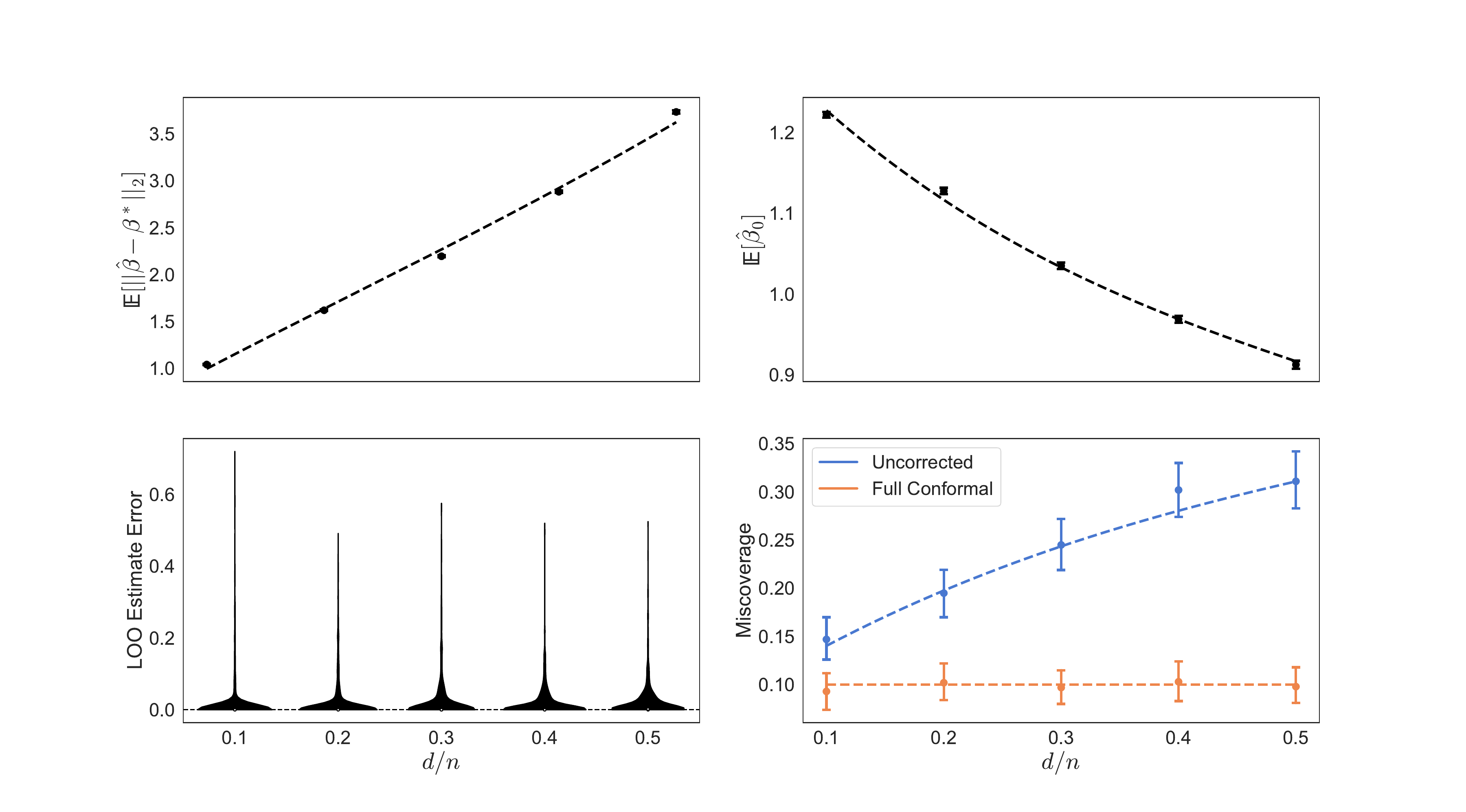}
    \caption{Empirical validation of our conjectures for high-dimensional quantile regression. Dots and error bars display point estimates and 95\% confidence intervals for the values of $\mme[\|\hat{\beta} - \beta^*\|_2]$ (top left panel), $\mme[\hat{\beta}_0]$ (top right panel), $\mmp(Y_{n+1} \leq \hat{\beta}_{0,(n+1)} + X_{n+1}^\top \hat{\beta}_{(n+1)})$ (bottom right panel, blue) and $\mmp(Y_{n+1} \in \hat{C}_{\text{dual}}^{\text{QR}})$ (bottom right panel, orange), while dotted lines show the values predicted by our theory. Finally, violin plots in the bottom left panel show the empirical distribution of the error in our leave-one-out formula for $\eta_{n+1}$ (namely the absolute value of the right-hand side of (\ref{eq:eta_loo_rep})). All figures contain results from 1000 trials at $n=400$ and $\alpha = 0.1$ with data generated as in Figure \ref{fig:ridge_cov}.}
    \label{fig:qr_validation}
\end{figure}

\begin{conjecture}\label{conj:qr}
    Let $(\lambda, \epsilon)$ denote a copy of $(\lambda_{i},\epsilon_{i})$ independent of $Z \sim N(0,1)$. Then, under the assumptions on the data-generating process given in Section \ref{sec:general_losses} there exist asymptotic constants $c_{\infty}$, $N_{\infty}$, $\beta_{0,\infty}$ such that
    \begin{equation}\label{eq:eta_loo_rep}
    \hat{\eta}_{n+1} - \frac{R_{n+1}^{(n+1)} - \textup{prox}(\lambda_{n+1}^2c_{\infty}\ell_{\alpha})(R_{n+1}^{(n+1)})}{\lambda_{n+1}^2c_{\infty}} \stackrel{\mmp}{\to} 0,
    \end{equation}
    and
    \[
     R_{n+1}^{(n+1)} \stackrel{D}{\to} \epsilon - \beta_{0,\infty} + \lambda N_{\infty} Z.
   \]
    Moreover, letting $W := \epsilon - \beta_{0,\infty} + \lambda N_{\infty} Z$,  the triple $(c_{\infty}, N_{\infty}, \beta_{0,\infty})$ satisfies the system of equations,
    \[
    \begin{cases}
    &   \mme\left[ \frac{(W - \textup{prox}(\lambda^2c_{\infty}\ell_{\alpha}(W))^2}{\lambda^2} \right] = \gamma N_{\infty}^2,\\
    & \mmp\left(W \in [-\lambda^2c_{\infty}\alpha,\lambda^2c_{\infty}(1-\alpha)] \right) = \gamma,\\
    & \mme\left[ \frac{W - \textup{prox}(\lambda^2c_{\infty}\ell_{\alpha})(W))}{\lambda^2c_{\infty}} \right] = 0.
    \end{cases}
    \]
    Following from this last equation it holds that,
    \[
    \mmp(Y_{n+1} \in \hat{C}_{\textup{dual}}^{\textup{QR}} \mid \{(X_i,Y_i)\}_{i=1}^{n}) = \mmp(\hat{\eta}_{n+1} < U \mid \{(X_i,Y_i)\}_{i=1}^{n}) = 1- \alpha -\mme[\hat{\eta}_{n+1} \mid \{(X_i,Y_i)\}_{i=1}^{n}] \stackrel{\mmp}{\to} 1-\alpha,
    \]
    while for standard quantile regression,
    \[
    \mmp(Y_{n+1} \leq \hat{\beta}_{0,(n+1)}  + X_{n+1}^\top \hat{\beta}_{(n+1)} \mid \{(X_i,Y_i)\}_{i=1}^{n}) \stackrel{\mmp}{\to} \mmp(W \leq 0).
    \]
   
\end{conjecture}

Empirical support for this conjecture is shown in Figure \ref{fig:qr_validation}. The top two panels compare empirical estimates of $\mme[\|\hat{\beta} - \beta^*\|_2]$ and $\mme[\hat{\beta}_0]$ against the values predicted by our theory. The bottom right panel forms the same comparison for the marginal miscoverage of both full conformal quantile regression (orange) and its vanilla uncorrected counterpart (blue). We see that in all cases the empirical estimates (dots) are extremely close to the values predicted by our theory (dotted lines). Finally, the bottom left panel displays an empirical estimate of the distribution of the error in our leave-one-out formula for $\eta_{n+1}$ (namely the absolute value of the right-hand side of (\ref{eq:eta_loo_rep})). As expected, we find that this distribution is tightly concentrated near zero. Overall, these results strongly support Conjecture \ref{conj:qr} and thus give us a tight characterization of the behaviour of full conformal quantile regression in high dimensions.  

\section{Acknowledgments}

E.J.C. was supported by the Office of Naval Research grant N00014-24-1-2305, the National Science Foundation grant DMS-2032014, and the Simons Foundation under award 814641. I.G. was also supported by the Office of Naval Research grant N00014-24-1-2305 and the Simons Foundation award 814641, as well as additionally by the Overdeck Fellowship Fund. The authors would like to thank Daniel LeJeune for helpful discussions on this work.

\newpage

\bibliographystyle{plainnat}
\bibliography{HighDimFullConformal}

\newpage

\appendix 

\section{Notation}

The proofs below make use of a number of additional pieces of notation. We will let $P_{\beta}$ denote the distribution of $\sqrt{d}\beta^*_i$ and $P_{\lambda}$ denote the distribution of $\lambda_i$. Since the distribution of $\sqrt{d}\beta^*_i$ does not depend on $d$ (as, in particular, $\beta^*_i$ is resampled as $d$ varies), we will write $\mme[(\sqrt{d}\beta^*_i)^k]$ to denote the $k_{\text{th}}$ moment of $P_{\beta}$. Recalling that we assumed that $P_{\lambda}$ has a bounded support, we let $\|P_{\lambda}\|_{\infty} := \inf \{t : \mmp(\lambda > t) = 0\}$ denote the minimum upper bound on the support. We let $X_{i:j}$ denote the matrix with rows $X_i,\dots,X_j$.

\section{Proofs for Section \ref{sec:high_dim_ridge}}

\subsection{Proof of Lemma \ref{lem:ridge_stab}}

The proof of Lemma \ref{lem:ridge_stab} will make use of the following bound on the norm of the fitted coefficients.
\begin{lemma}\label{lem:coef_bound}
Under no assumptions on the data, the fitted ridge regression coefficients satisfy the deterministic bound,
\[
\|\hat{\beta}\|_2 \leq \sqrt{\frac{1}{\tau(n+1)} \sum_{i=1}^{n+1} Y_i^2}.
\]    
\end{lemma}
\begin{proof}
    By the optimally of $\hat{\beta}$, we have
    \[
    \tau \|\hat{\beta}\|_2^2 \leq \frac{1}{n+1} \sum_{i=1}^{n+1} (Y_i - X_i^\top \hat{\beta})^2 + \tau \|\hat{\beta}\|_2^2 \leq \frac{1}{n+1} \sum_{i=1}^{n+1} (Y_i- 0)^2 + \tau \|0\|_2^2,
    \]
    as claimed.
\end{proof}

With this result in hand we now prove Lemma \ref{lem:ridge_stab}.

\begin{proof}[Proof of Lemma \ref{lem:ridge_stab}]
By the calculations from the main text it is sufficient to show that
    \begin{enumerate}
\item 
$\mme\left[\left|(Y_{n+1} - X_{n+1}^\top \hat{\beta}_{(n+1)}) \frac{1}{n+1}X_1^\top \left(\frac{1}{n+1} X^\top X + \tau I_d \right)^{-1}X_{n+1} \right|\right] = o(1),$
\item $\mme\left[\left| \frac{\tau}{n+1} X_1^\top\left(\frac{1}{n+1} X^\top X + \tau I_d \right)^{-1}\hat{\beta}_{(n+1)} \right| \right]  = o(1).$
\end{enumerate}
For the first claim, using our calculations from the main text it is easy to show that
\begin{align*}
    & \mme\left[\left|(Y_{n+1} - X_{n+1}^\top \hat{\beta}_{(n+1)}) \frac{1}{n+1}X_1^\top \left(\frac{1}{n+1} X^\top X + \tau I_d \right)^{-1}X_{n+1} \right|\right]\\
    & \leq 2\mme\left[ \left| (Y_{n+1} - X_{n+1}^\top \hat{\beta}_{(n+1)}) \max\left\{1, \frac{\|X_1\|_2^2}{\tau(n+1)} \right\} \frac{1}{n+1}X_1^\top\left(\frac{1}{n+1} X_{2:(n+1)}^\top X_{2:(n+1)} + \tau I_d \right)^{-1} X_{n+1}  \right| \right].
\end{align*}
Our goal will be to bound the above using the Cauchy-Schwartz in equality. To do this, note that since $P_W$ has bounded moments we have the bounds,
\begin{align*}
& \mme[|(Y_{n+1} - X_{n+1}^\top \hat{\beta}_{(n+1)})|^4]  \leq 64 \left(\mme[\epsilon_{n+1}^4] + \mme[|X_{n+1}^\top \beta^*|^4] + \mme[|X_{n+1}^\top \hat{\beta}_{(n+1)}|^4]\right)\\
 & \ \ \ \ \leq 64\left(\mme[\epsilon_{n+1}^4]  + 3\|P_{\lambda}\|_{\infty}^4\mme[W_{11}^4](\mme[\|\beta^*\|_2^4] + \mme[\|\hat{\beta}_{(n+1)}\|_2^4])\right)\\
 & \ \ \ \ \ \ \ \ \leq 64 \left(\mme[\epsilon_{n+1}^4] +  3\|P_{\lambda}\|_{\infty}^4\mme[W_{11}^4]\left(\mme[\|\beta^*\|_2^4] + \mme\left[\left(\frac{1}{\tau n} \sum_{i=1}^{n} Y_i^2\right)^2\right]\right)\right) = O(1),
\end{align*}
where the last inequality applies Lemma \ref{lem:coef_bound}. On the other hand, by assumption,
\[
\mme\left[ \max\left\{1, \frac{\|X_1\|_2^2}{\tau(n+1)} \right\}^4 \right] = O(1),
\]
and 
\begin{align*}
& \mme\left[\left( \frac{1}{n+1}X_1^\top\left(\frac{1}{n+1} X_{2:(n+1)}^\top X_{2:(n+1)} + \tau I_d \right)^{-1} X_{n+1}\right)^2  \right]\\
& \leq  \frac{\|P_{\lambda}\|_{\infty}^2}{(n+1)^2} \mme\left[\left\|\left(\frac{1}{n+1} X_{2:(n+1)}^\top X_{2:(n+1)} + \tau I_d \right)^{-1} X_{n+1}\right\|_2^2\right]\\
& \leq \frac{\|P_{\lambda}\|_{\infty}^2}{(n+1)^2}\frac{1}{\tau^2} \mme[\|X_{n+1}\|^2]\\
& = O\left(\frac{1}{n} \right).
\end{align*}
Combining these results proves claim 1. To prove claim 2 note that, by Lemma \ref{lem:coef_bound}
\begin{align*}
    & \mme\left[\left| \frac{\tau}{n+1} X_1^\top\left(\frac{1}{n+1} X^\top X + \tau I_d \right)^{-1}\hat{\beta}_{(n+1)} \right| \right]\\
    & \leq \mme\left[ \frac{1}{n+1} \|X_1\|_2 \|\hat{\beta}_{(n+1)}\|_2\right]\\
    & \leq \frac{1}{n+1} \mme[\|X_1\|_2^2]^{1/2} \mme\left[\frac{1}{\tau n}\sum_{i=1}^{n} Y_i^2\right]^{1/2}\\
    & = O\left(\frac{1}{\sqrt{n}} \right).
\end{align*}
\end{proof}

\subsection{Proof of Lemma \ref{lem:ridge_norm_conv}}

\begin{proof}[Proof of Lemma \ref{lem:ridge_norm_conv}] Let $Y_{1:n} = (Y_1,\dots,Y_{n})^T$ and $\epsilon_{1:n} = (\epsilon_1,\dots,\epsilon_n)^T$. By standard formula for ridge regression coefficients we have that 
\begin{align*}
\hat{\beta}_{(n+1)} & = \left(\frac{1}{n}X_{1:n}^\top X_{1:n} + \tau I_d\right)^{-1} \frac{1}{n} X_{1:n}^\top Y_{1:n}\\
& = \left(\frac{1}{n}X_{1:n}^\top X_{1:n} + \tau I_d\right)^{-1} \frac{1}{n} X_{1:n}^\top X_{1:n} \beta^* + \left(\frac{1}{n}X_{1:n}^\top X_{1:n} + \tau I_d\right)^{-1} \frac{1}{n} X_{1:n}^\top \epsilon_{1:n}.
\end{align*}
So,
\begin{align*}
    & \|\hat{\beta}_{(n+1)} - \beta^*\|_2^2 =\|\hat{\beta}_{(n+1)}\|_2^2  - 2 \langle \hat{\beta}_{(n+1)}, \beta^*\rangle +  \|\beta^*\|_2^2\\
    &  = ( \beta^*)^\top \frac{1}{n} X_{1:n}^\top X_{1:n}\left(\frac{1}{n}X_{1:n}^\top X_{1:n} + \tau I_d\right)^{-2} \frac{1}{n} X_{1:n}^\top X_{1:n} \beta^*\\
    & + (\epsilon_{1:n})^\top \frac{1}{n} X_{1:n} \left(\frac{1}{n}X_{1:n}^\top X_{1:n} + \tau I_d\right)^{-2} \frac{1}{n} X_{1:n}^\top \epsilon_{1:n}\\
    & \ \ \ \  + 2  ( \beta^*)^\top \frac{1}{n} X_{1:n}^\top X_{1:n} \left(\frac{1}{n}X_{1:n}^\top X_{1:n} + \tau I_d\right)^{-2} \frac{1}{n} X_{1:n}^\top \epsilon_{1:n}\\
    &  \ \ \ \ - 2(\beta^*)^\top\left(\frac{1}{n}X_{1:n}^\top X_{1:n} + \tau I_d\right)^{-1} \frac{1}{n} X_{1:n}^\top X_{1:n} \beta^* - 2(\beta^*)^\top\left(\frac{1}{n}X_{1:n}^\top X_{1:n} + \tau I_d\right)^{-1} \frac{1}{n} X_{1:n}^\top \epsilon_{1:n}\\
    & \ \ \ \ + \|\beta^*\|_2^2.
\end{align*}
We will show that each of the six terms above converges in probability. 

First, since the coordinates of $\beta^*$ are i.i.d., a standard application of the law of large numbers shows that $\|\beta^*\|_2^2 \stackrel{P}{\to} \mme[(\sqrt{d}\beta^*_i)^2]$. Now, to handle the cross terms involving both $\beta^*$ and $\epsilon_{1:n}$, recall that $\epsilon_{1:n} \independent (X_{1:n},\beta^*)$. Thus,
\begin{align*}
    & \mme\left[\left(( \beta^*)^\top \frac{1}{n} X_{1:n}^\top X_{1:n} \left(\frac{1}{n}X_{1:n}^\top X_{1:n} + \tau I_d\right)^{-2} \frac{1}{n} X_{1:n}^\top \epsilon_{1:n}\right)^2 \mid (X_{1:n}, \beta^*) \right]\\
    & = \mme[\epsilon_i^2]\left\|( \beta^*)^\top \frac{1}{n} X_{1:n}^\top X_{1:n} \left(\frac{1}{n}X_{1:n}^\top X_{1:n} + \tau I_d\right)^{-2} \frac{1}{n} X_{1:n}^\top\right\|_2^2\\
    & \leq  \mme[\epsilon_i^2]\| \beta^*\|^2 \left\|\frac{1}{n} X_{1:n}^\top X_{1:n} \left(\frac{1}{n}X_{1:n}^\top X_{1:n} + \tau I_d\right)^{-2} \frac{1}{n} X_{1:n}^\top\right\|_{op}^2\\
    & \leq \mme[\epsilon_i^2]\frac{1}{n} \| \beta^*\|^2 \frac{\|n^{-1/2}X_{1:n}\|_{op}^6}{\tau^4},
\end{align*}
and so, $(\beta^*)^\top \frac{1}{n} X_{1:n}^\top X_{1:n} (\frac{1}{n}X_{1:n}^\top X_{1:n} + \tau I_d)^{-2} \frac{1}{n} X_{1:n}^\top \epsilon_{1:n} = O_{\mmp}( \frac{\|\beta^*\|_2^2 \|n^{-1/2}X_{1:n}\|_{op}^6}{n\tau^4})$. Moreover, by standard results on the spectrum of the empirical covariance (see Theorem 3.1 of \cite{Yin1988}) we have that
\[
\|n^{-1/2}X_{1:n}\|_{op} = \left\|\frac{1}{n} \sum_{i=1}^{n} X_iX_i^\top\right\|^{1/2}_{op} \leq \|P_{\lambda}\|_{\infty} \left\|\frac{1}{n} \sum_{i=1}^{n} W_iW_i^\top\right\|^{1/2}_{op}  \stackrel{\mmp}{\to} \|P_{\lambda}\|_{\infty} (1 + \sqrt{\gamma}),
\]
and thus we conclude that $(\beta^*)^\top \frac{1}{n} X_{1:n}^\top X_{1:n} (\frac{1}{n}X_{1:n}^\top X_{1:n} + \tau I_d)^{-2} \frac{1}{n} X_{1:n}^\top \epsilon_{1:n} = o_{\mmp}(1)$. By a similar argument, one can also easily verify that $(\beta^*)^\top (\frac{1}{n}X_{1:n}^\top X_{1:n} + \tau I_d)^{-1} \frac{1}{n} X_{1:n}^\top \epsilon_{1:n} = o_{\mmp}(1)$. 

It remains to characterize the first, second, and fourth term from our original expression for $\|\hat{\beta}_{(n+1)} - \beta^*\|_2^2$. We claim that all three of these terms can be written as matrix traces. Namely, we claim that 
\begin{enumerate}
    \item  
    \begin{align*}
    \bigg| ( \beta^*)^\top \frac{1}{n} X_{1:n}^\top X_{1:n} & \left(\frac{1}{n}X_{1:n}^\top X_{1:n}  + \tau I_d\right)^{-2} \frac{1}{n} X_{1:n}^\top X_{1:n} \beta^*\\
    & - \frac{\mme[(\sqrt{d}\beta^*_i)^2]}{d} \text{tr}\left( \frac{1}{n} X_{1:n}^\top X_{1:n}\left(\frac{1}{n}X_{1:n}^\top X_{1:n} + \tau I_d\right)^{-2} \frac{1}{n} X_{1:n}^\top X_{1:n} \right) \bigg| \stackrel{\mmp}{\to} 0,
    \end{align*}
    \item 
    \begin{align*}
    \bigg| \epsilon_{1:n}^\top  \frac{1}{n} X_{1:n}\left(\frac{1}{n}X_{1:n}^\top X_{1:n} + \tau I_d\right)^{-2} & \frac{1}{n} X_{1:n}^\top \epsilon_{1:n}\\
    & - \frac{\mme[\epsilon_i^2]}{n} \text{tr}\left( X_{1:n}\left(\frac{1}{n}X_{1:n}^\top X_{1:n} + \tau I_d\right)^{-2} \frac{1}{n} X_{1:n}^\top  \right) \bigg| \stackrel{\mmp}{\to} 0,
    \end{align*}
    \item 
     \begin{align*}
    \bigg|  ( \beta^*)^\top \left(\frac{1}{n}X_{1:n}^\top X_{1:n} + \tau I_d\right)^{-1} & \frac{1}{n} X_{1:n}^\top X_{1:n} \beta^*\\
    & - \frac{\mme[(\sqrt{d}\beta^*_i)^2}{d} \text{tr}\left( \left(\frac{1}{n}X_{1:n}^\top X_{1:n} + \tau I_d\right)^{-1} \frac{1}{n} X_{1:n}^\top X_{1:n} \right) \bigg| \stackrel{\mmp}{\to} 0.
    \end{align*}
\end{enumerate}
We will prove the second claim. The proofs of the first and third claims are identical. For ease of notation, let $A =  X_{1:n}(\frac{1}{n}X_{1:n}^\top X_{1:n} + \tau I_d)^{-2} \frac{1}{n} X_{1:n}^\top$. Then, by a direct computation we have that 
\begin{align*}
    & \mme\left[ \left| \frac{1}{n}\epsilon_{1:n}^\top  A \epsilon_{1:n} - \frac{\mme[\epsilon_i^2]}{n} \text{tr}\left(A  \right) \right|^2 \mid A \right]  = \frac{1}{n^2} \sum_{i=1}^n \mme[(\epsilon^2_i - \mme[\epsilon_i^2])^2] A_{ii}^2 + \frac{1}{n^2} \sum_{i \neq j} \mme[\epsilon_i^2]^2 A_{ij}^2\\
    & \ \ \ \ \leq \frac{\mme[\epsilon_i^4]}{ n^2} \sum_{i,j} A_{ij}^2 =  \frac{\mme[\epsilon_i^4]}{ n^2}  \text{tr}(A^\top A) \leq \frac{\mme[\epsilon_i^4]}{ n} \|A^\top A\|_{op} \leq \frac{\mme[\epsilon_i^4]}{ n} \left(\sup_{x \geq 0} \frac{x^2}{(x^2+\tau)^2} \right)^2,
\end{align*}
where on the last line we have used the fact that $\|A\|_{op} \leq \sup_{x \geq 0} \frac{x^2}{(x^2+\tau)^2} $, which follows immediately from a direct calculation utilising the singular value decomposition of $X_{1:n}$. Since $\sup_{x \geq 0} \frac{x^2}{(x^2+\tau)^2} < \infty$, this proves claim 2.

It remain to study the three matrix traces appearing above. For this purpose, note that by the Woodbury matrix identity (and namely the special case $(I+A)^{-1} = I - (I+A)^{-1}A$),
\begin{align*}
 & \text{tr}\left( \frac{1}{n}X_{1:n}^\top X_{1:n}\left(\frac{1}{n}X_{1:n}^\top X_{1:n} + \tau I_d\right)^{-2} \frac{1}{n} X_{1:n}^\top X_{1:n} \right)  =  \text{tr}\left(\left(I_d - \tau\left(\frac{1}{n}X_{1:n}^\top X_{1:n} + \tau I_d\right)^{-1} \right)^2\right)\\
     & = \text{tr}\left(I_d\right) - 2\tau\text{tr}\left(\left(\frac{1}{n}X_{1:n}^\top X_{1:n} + \tau I_d\right)^{-1}\right) + \tau^2\text{tr}\left(\left(\frac{1}{n}X_{1:n}^\top X_{1:n} + \tau I_d\right)^{-2}\right),
\end{align*}
while by similar calculations,
\begin{align*}
& \text{tr}\left( X_{1:n}\left(\frac{1}{n}X_{1:n}^\top X_{1:n} + \tau I_d\right)^{-2} \frac{1}{n} X_{1:n}^\top \right)  = \text{tr}\left( \left(\frac{1}{n}X_{1:n}^\top X_{1:n} + \tau I_d\right)^{-2} \frac{1}{n} X_{1:n}^\top X_{1:n} \right)\\
& = \text{tr}\left( \left(\frac{1}{n}X_{1:n}^\top X_{1:n} + \tau I_d\right)^{-1} \left(I_d - \tau\left(\frac{1}{n}X_{1:n}^\top X_{1:n} + \tau I_d\right)^{-1}\right) \right)\\
& =  \text{tr}\left( \left(\frac{1}{n}X_{1:n}^\top X_{1:n} + \tau I_d\right)^{-1}\right)  -  \tau\text{tr}\left( \left(\frac{1}{n}X_{1:n}^\top X_{1:n} + \tau I_d\right)^{-2}\right), 
\end{align*}
and 
\begin{align*}
    \text{tr}\left( \left(\frac{1}{n}X_{1:n}^\top X_{1:n} + \tau I_d\right)^{-1} \frac{1}{n} X_{1:n}^\top X_{1:n} \right) = \text{tr}(I_d) - \tau \text{tr} \left( \left (\frac{1}{n}X_{1:n}^\top X_{1:n} + \tau I_d\right)^{-1}\right).
\end{align*}
Letting $\lambda \sim P_{\lambda}$ denote a generic draw of $\lambda_i$, we will show in Lemma \ref{lem:ridge_tr_final} below that there exists $e_{\infty}(\tau)>0$ such that
\begin{enumerate}
\item $\frac{1}{d}\text{tr}\left((\frac{1}{n}X_{1:n}^\top X_{1:n} + \tau I_d)^{-1}\right) \stackrel{\mmp}{\to} e_{\infty}(\tau)$,
\item $\frac{1}{d}\text{tr}\left((\frac{1}{n}X_{1:n}^\top X_{1:n} + \tau I_d)^{-2}\right) \stackrel{\mmp}{\to} \frac{e_{\infty}(\tau)}{\tau + \mme\left[\frac{\lambda^2}{(1+\gamma\lambda^2e_{\infty}(\tau))^2}\right]}$.
\end{enumerate}
So, plugging this into our original expression for $\|\hat{\beta}_{(n+1)} - \beta^*\|_2^2$ we find that 
\begin{align*}
\|\hat{\beta}_{(n+1)} - \beta^*\|_2^2 & \stackrel{\mmp}{\to} \mme[(\sqrt{d}\beta^*_i)^2] \frac{\tau^2 e_{\infty}(\tau)}{\tau + \mme\left[\frac{\lambda^2}{(1+\gamma\lambda^2e_{\infty}(\tau))^2}\right]}\\
& \ \ \ \ \ + \gamma\mme[\epsilon_i^2] \mme\left[\frac{\lambda^2}{(1+\gamma\lambda^2e_{\infty}(\tau))^2}\right]  \frac{e_{\infty}(\tau)}{\tau + \mme\left[\frac{\lambda^2}{(1+\gamma\lambda^2e_{\infty}(\tau))^2}\right]}\\
& > 0,
\end{align*}
as desired.

\end{proof}

We conclude this section by proving a sequence of lemmas that allow us to control $\frac{1}{d}\text{tr}\left((\frac{1}{n}X_{1:n}^\top X_{1:n} + \tau I_d)^{-1}\right)$ and $\frac{1}{d}\text{tr}\left((\frac{1}{n}X_{1:n}^\top X_{1:n} + \tau I_d)^{-2}\right)$. Our first result characterizes  the limit of $\frac{1}{d} \text{tr}\left( \left(\frac{1}{n}X_{1:n}^\top X_{1:n} + z I_d\right)^{-1}\right)$ over $z$ in the positive complex orthant.

\begin{lemma}\label{lem:ridge_tr_1}
    Let $\lambda \sim P_{\lambda}$ denote a generic draw of $\lambda_i$. Under the setting of Section \ref{sec:high_dim_setting}, we have that for all $\{\Theta_n\} \subseteq \mathbb{C}^{d \times d}$ with $\sup_n \textup{tr}((\Theta_n^H \Theta_n)^{1/2}) < \infty$ and all $z \in \{a + bi\in \mathbb{C} : a, b > 0\}$,
    \[
    \textup{tr}\left(\Theta_n\left(\left(\frac{1}{n}X_{1:n}^\top X_{1:n} + z I_d \right)^{-1} -  e_{\infty}(z) I_d \right) \right) \stackrel{a.s.}{\to} 0,
    \]
    where $e_{\infty}(z)$ is the unique solution in $\{a + bi \in \mathbb{C} : a > 0, b < 0\}$ to the equation
    \[
    \mme\left[\frac{\lambda^2 e_{\infty}(z)}{1+\gamma \lambda^2 e_{\infty}(z)} \right] - 1 + z e_{\infty}(z) = 0.
    \]
\end{lemma}
\begin{proof}
By Theorem 1 of \citet{Rubio2011} we have that for all $\{\Theta_n\} \in \mathbb{C}^{d \times d}$ with $\sup_n \text{tr}(\Theta_n^H \Theta_n) < \infty$ and all $z \in \{a + bi\in \mathbb{C} : a,b > 0\}$,
\begin{equation}\label{eq:rub_mestre_limit}
\text{tr}\left(\Theta_n\left(\left(\frac{1}{n}X_{1:n}^\top X_{1:n} + z I_d \right)^{-1} -  e_n(z) I_d \right) \right) \stackrel{a.s.}{\to} 0,
\end{equation}
where $e_n(z)$ is the unique solution in $\{a+bi \in \mathbb{C} : b < 0\}$ to the equation
\[
f_n(e) := \frac{1}{n} \sum_{i=1}^n \frac{\lambda^2_ie}{1 +  \frac{d}{n}\lambda^2_i e} - 1 + z e = 0.
\]
So, to prove the desired result we just need to show that $e_{n}(z) \stackrel{a.s.}{\to} e_{\infty}(z) \in \{a+bi \in \mathbb{C} : a > 0,b < 0\}$.

Let $\rho_1,\dots,\rho_d$ denote the eigenvalues of $\frac{1}{n} X_{1:n}^\top X_{1:n}$. Taking $\Theta_n = \frac{1}{d}I_d$ and writing $z = a+bi$ for some $a,b >0$, we see that (\ref{eq:rub_mestre_limit}) implies that 
\[
\frac{1}{d}\sum_{i=1}^d \frac{1}{\rho_i + z} - e_n(z) \stackrel{a.s.}{\to} 0 \iff \frac{1}{d}\sum_{i=1}^d \frac{\rho_i+ a - b i}{(\rho_i + a)^2 + b^2} - e_n(z) \stackrel{a.s.}{\to} 0.
\]
Now, by Theorem 1 of \cite{Yin1988}, $\|\frac{1}{n} X_{1:n}^\top X_{1:n} \|_{op} \leq \|P_{\lambda}\|_{\infty}^2 \|\frac{1}{n} \sum_{i=1}^{n} W_i W_i^\top \|_{op} \stackrel{a.s.}{\to} \|P_{\lambda}\|_{\infty}^2 (1 + \sqrt{\gamma})^2$. Applying this to the previous expression, we find that there exists constants $c_1, c_2', c_2, c_2' > 0$, such that with probability one, $e_n(z) \in \{x + yi \in \mathbb{C} : c_1 \leq x \leq  c'_1, - c_2 \leq  y \leq  -c'_2\}$ for all $n$ sufficiently large. For ease of notation, let $C := \{x + yi \in \mathbb{C} : c_1 \leq x \leq  c'_1, - c_2 \leq  y \leq  -c'_2\}$. By the law of large numbers, we know that for any $e \in C$, $f_n(e) \stackrel{a.s.}{\to} f_{\infty}(e) :=  \mme[\frac{\lambda^2e}{1 + \gamma \lambda^2e }] - 1 + z e$. Moreover, we have that for all $e \in C$,
\[
\left| \frac{d}{de} f_n(e)\right| = \left|  \frac{1}{n} \sum_{i=1}^n \frac{\lambda^2_i}{(1+\frac{d}{n}\lambda^2_i e)^2} + z \right| \leq \|P_{\lambda}\|_{\infty}^2 + |z|.
\]
Thus, $f_n(e)$ is Lipschitz on $C$ and since $C$ is compact, it follows from standard arguments that with probability one, $ f_n(e) \to f_{\infty}(e)$ for all $e \in C$ (see Lemma \ref{lem:lip_to_unif_as} for details). In particular, we find that with probability 1, if $e_{n_k}(z) \to \tilde{e}$ is a convergent subsequence of $\{e_n(z)\}_{n=1}^{\infty}$, then $f_{\infty}(\tilde{e}) = 0$. So, applying the compactness of $C$ again, we find that in order to prove the desired convergence of $e_n(z)$ it is sufficient to show that $f_{\infty}$ has a unique root on $\{ a+ bi \in \mathbb{C}: a>0, b < 0\}$.

To prove that $f_{\infty}$ has a unique root, we repeat the arguments of \citet{Rubio2011} for demonstrating the uniqueness of the root of $f_n$. Namely, suppose that $\tilde{e}_1, \tilde{e}_2 \in \{ a+ bi  \in \mathbb{C} : a>0, b < 0\}$ are both roots of $f_{\infty}$. For ease of notation, let $x_{\infty}(\tilde{e}_i) := 1/\tilde{e}_i - z = \mme\left[ \frac{\lambda^2}{1+\gamma \lambda^2 \tilde{e}_i} \right]$. Then, we have that 
\begin{align*}
|\tilde{e}_1 - \tilde{e}_2| & = |\tilde{e}_1|\cdot |\tilde{e}_2| \cdot |x_{\infty}(\tilde{e}_1) - x_{\infty}(\tilde{e}_2)| = |\tilde{e}_1|\cdot |\tilde{e}_2| \cdot \left| \mme\left[ \frac{\gamma \lambda^4}{(1+\gamma \lambda^2 \tilde{e}_1) (1+\gamma \lambda^2 \tilde{e}_2)}  \right] \right| \cdot |\tilde{e}_1 - \tilde{e}_2|\\
& \leq  |\tilde{e}_1|\cdot |\tilde{e}_2| \cdot  \mme\left[  \frac{\gamma \lambda^4}{|1+\gamma \lambda^2 \tilde{e}_1|^2} \right]^{1/2} \mme\left[  \frac{\gamma \lambda^4}{|1+\gamma \lambda^2 \tilde{e}_2|^2} \right]^{1/2} |\tilde{e}_1 - \tilde{e}_2|,
\end{align*}
where the last line follows from the Cauchy-Schwartz inequality. Now, note that by definition,
\[
\text{Im}(x_{\infty}(\tilde{e}_i)) = - \text{Im}(\tilde{e}_i)  \mme\left[  \frac{\gamma \lambda^4}{|1+\gamma \lambda^2 \tilde{e}_i|^2} \right],
\]
and thus,
\[
0 < |\tilde{e}_i|\mme\left[  \frac{\gamma \lambda^4}{|1+\gamma \lambda^2 \tilde{e}_i|^2} \right] = |\tilde{e}_i| \frac{\text{Im}(z) - \text{Im}\left(\frac{1}{\tilde{e}_i}\right) }{\text{Im}(\tilde{e}_i)} = |\tilde{e}_i| \frac{\text{Im}(z)}{\text{Im}(\tilde{e}_i)} + 1 < 1,
\]
where the last inequality follows from the fact that by assumption, $\text{Im}(z) > 0$ and $\text{Im}(\tilde{e}_i) <  0$. Plugging this into our previous inequalities on $|\tilde{e}_1 - \tilde{e}_2|$, we conclude that 
\[
|\tilde{e}_1 - \tilde{e}_2| < |\tilde{e}_1 - \tilde{e}_2|,
\]
and thus, $|\tilde{e}_1 - \tilde{e}_2| = 0$.

\end{proof}

Our next result characterizes the limit of $\frac{1}{d} \text{tr}\left( \left(\frac{1}{n}X_{1:n}^\top X_{1:n} + z I_d\right)^{-2}\right)$ over $z$ in the positive complex orthant.

\begin{lemma}\label{lem:ridge_tr_2}
    Let $\lambda \sim P_{\lambda}$ denote a generic draw of $\lambda_i$ and $e_{\infty}(z)$ denote the function appearing in the statement of Lemma \ref{lem:ridge_tr_1}. Then, under the setting of Section \ref{sec:high_dim_setting}, it holds that for all $\{\Theta_n\} \subseteq \mathbb{C}^{d \times d}$ with $\sup_n \textup{tr}((\Theta_n^H \Theta_n)^{1/2}) < \infty$ and all $z \in \{a + bi\in \mathbb{C} : a, b > 0\}$,
    \[
    \textup{tr}\left(\Theta_n\left(\left(\frac{1}{n}X_{1:n}^\top X_{1:n} + z I_d \right)^{-2} +  e_{\infty}'(z) I_d \right) \right) \stackrel{a.s.}{\to} 0.
    \]
    Moreover, we have that 
    \[
    e_{\infty}'(z) = - \frac{e_{\infty}(z)}{z + \mme\left[\frac{\lambda^2}{(1+\gamma\lambda^2e_{\infty}(z))^2}\right]}. 
    \]
\end{lemma}
\begin{proof}
    Note that 
    \[
    \frac{d}{dz}\left(\frac{1}{n}X_{1:n}^\top X_{1:n} + z I_d \right)^{-1} = - \left(\frac{1}{n}X_{1:n}^\top X_{1:n} + z I_d \right)^{-2}.
    \]
    Thus, by Lemma \ref{lem:ridge_tr_1}, to prove this result it is sufficient to apply the proof of Theorem 11 of \citet{Dobriban2020}. To do this, we must verify the conditions of the theorem. First, note that by following the arguments in the proof of Lemma \ref{lem:ridge_tr_1} (and in particular by arguing that the functions appearing in the theorem are locally Lipschitz in $z$), it is not difficult to show that for all $\{\Theta_n\} \subseteq \mathbb{C}^{d \times d}$ with $\sup_n \textup{tr}((\Theta_n^H \Theta_n)^{1/2}) < \infty$,
    \[\mmp\left(\frac{1}{d}\textup{tr}\left(\Theta_n\left(\left(\frac{1}{n}X_{1:n}^\top X_{1:n} + z I_d \right)^{-1} -  e_{\infty}(z) I_d \right) \right) \to 0,\ \forall z \in D \right) = 1,
    \]
    where $D$ is any open ball in $\{a+bi \in \mathbb{C} : a,b > 0\}$. 
    
    We now show that $e_{\infty}(z)$ is analytic. To begin, note that by taking $\Theta_n = \frac{1}{d} I_d$ in Lemma \ref{lem:ridge_tr_1}, we have that 
    \[
    \frac{1}{d} \text{tr}\left( (\frac{1}{n}X^\top X_{1:n} + zI_d)^{-1} \right) \stackrel{a.s.}{\to} e_{\infty}(z).
    \]
    Now, let $g_n(z) := \frac{1}{d} \text{tr}\left( (\frac{1}{n}X^\top X_{1:n} + zI_d)^{-1} \right)$. Then, $g_n(z)$ is analytic with derivative bounded as,
    \[
    \left|\frac{d}{dz} g_n(z) \right| = \left| \frac{1}{d} \text{tr}\left( (\frac{1}{n}X^\top X_{1:n} + zI_d)^{-2} \right)\right| \leq \left\|(\frac{1}{n}X^\top X_{1:n} + zI_d)^{-2}\right\|_{op} \leq \frac{1}{|z|}. 
    \]
    Thus, for any $z \in \{a + bi\in \mathbb{C} : a, b > 0\}$, there exists an open neighbourhood $C \subseteq \{a + bi\in \mathbb{C} : a, b > 0\}$ of $z$, such that $g_n$ is Lipschitz on $C$. Letting $D$ denote this open neighbourhood, standard arguments (c.f. Lemma \ref{lem:func_uniform_conv}), imply that $\sup_{z \in D}|g_n(z) - e_{\infty}(z)| \stackrel{\mmp}{\to} 0$ and by Morera's theorem, this is sufficient to imply that $e_{\infty}(z)$ is analytic.

    Finally, to complete our verification of the conditions of Theorem 11 of \citet{Dobriban2020}, note that for all $\Theta_n \in \mathbb{C}^{d \times d}$ with $\sup_n \text{tr}((\Theta_n^H \Theta_n)^{1/2}) \leq M$,
    \begin{align*}
    & \left| \text{tr}\left(\Theta_n\left(\left(\frac{1}{n}X_{1:n}^\top X_{1:n} + z I_d \right)^{-1} -  e_{\infty}(z) I_d \right) \right) \right|\\
    & \ \ \ \  \leq \text{tr}((\Theta_n^H \Theta_n)^{1/2}) \left\| \left(\frac{1}{n}X_{1:n}^\top X_{1:n} + z I_d \right)^{-1} - e_{\infty}(z) I_d \right\|_{op}  \leq M \left(\frac{1}{|z|} + |e_{\infty}(z)| \right).
    \end{align*}
    
    This verifies the conditions of Theorem 11 of \citet{Dobriban2020} and thus we conclude that
    \[
    \text{tr}\left(\Theta_n\left(\left(\frac{1}{n}X_{1:n}^\top X_{1:n} - z I_d \right)^{-2} +  e'(z) I_d \right) \right) \stackrel{a.s.}{\to} 0.
    \]
    To get the desired formula for $e_{\infty}'(z)$, one simply notes that 
    \[
     \mme\left[\frac{\lambda^2 e_{\infty}(z)}{1+\gamma \lambda^2 e_{\infty}(z)} \right] - 1 + z e_{\infty}(z) = 0 \implies  \mme\left[\frac{\lambda^2 }{(1+\gamma \lambda^2 e_{\infty}(z))^2} \right] e'_{\infty}(z) + e_{\infty}(z) + z e'_{\infty}(z) = 0,
    \]
    where here we have used the fact that $\frac{\lambda^2 }{(1+\gamma \lambda^2 e_{\infty}(z))^2}$ is bounded (recall that $\text{Re}(e_{\infty}(z))>0$) to swap the derivative in $z$ with the expectation. 
\end{proof}

With the previous two lemmas in hand, we are now ready to prove our desired characterization of the limits of the $\frac{1}{d} \textup{tr}\left(\left(\frac{1}{n}X_{1:n}^\top X_{1:n} + \tau I_d \right)^{-1} \right)$ and $\frac{1}{d} \textup{tr}\left(\left(\frac{1}{n}X_{1:n}^\top X_{1:n} + \tau I_d \right)^{-2}\right)$.

\begin{lemma}\label{lem:ridge_tr_final}
    Let $\lambda \sim P_{\lambda}$ denote a generic draw of $\lambda_i$ and $e_{\infty}(z)$ be the function defined in the statement of Lemma \ref{lem:ridge_tr_1}. Then, under the setting of Section \ref{sec:high_dim_setting},
    \[
    \frac{1}{d} \textup{tr}\left(\left(\frac{1}{n}X_{1:n}^\top X_{1:n} + \tau I_d \right)^{-1}  \right) \stackrel{\mmp}{\to} e_{\infty}(\tau),
    \]
    and 
    \[
    \frac{1}{d} \textup{tr}\left(\left(\frac{1}{n}X_{1:n}^\top X_{1:n} + \tau I_d \right)^{-2}  \right) \stackrel{\mmp}{\to} \frac{e_{\infty}(\tau)}{\tau + \mme\left[\frac{\lambda^2}{(1+\gamma\lambda^2e_{\infty}(\tau))^2}\right]}.
    \]
    Moreover, $e_{\infty}(\tau)$ is a positive real number.
\end{lemma}
\begin{proof}
    We will prove the first limit. The second limit follows from an identical argument. By taking $\Theta_n = \frac{1}{d}I_d$ in the statement of Lemma \ref{lem:ridge_tr_1}, we know that for any $z = \tau + ai$ such that $a >0$,
    \[
    \frac{1}{d} \textup{tr}\left(\left(\frac{1}{n}X_{1:n}^\top X_{1:n} + z I_d \right)^{-1}  \right) \stackrel{a.s.}{\to} e_{\infty}(z). 
    \]
    Now, define the map $g_n(a) =  \frac{1}{d} \textup{tr}\left(\left(\frac{1}{n}X_{1:n}^\top X_{1:n} + (\tau + ai) I_d \right)^{-1} \right)$. Note that,
    \[
    |g_n'(a)| = \left|  i \frac{1}{d} \textup{tr}\left(\left(\frac{1}{n}X_{1:n}^\top X_{1:n} + (\tau + ai) I_d \right)^{-2} \right) \right| \leq \frac{1}{\tau^2},
    \]
    and thus $g_n$ is $1/\tau^2$-Lipschitz. Since $g_n(a) \stackrel{a.s.}{\to} e_{\infty}(\tau + ai)$, this implies that $a \mapsto e_{\infty}(\tau + ai)$ is $1/\tau^2$-Lipschitz as well. Moreover, we also clearly have that $|g_n(a)| \leq 1/\tau$ is bounded and thus, $|e_{\infty}(\tau + ai)| \leq 1/\tau$ is bounded as well. Putting these two facts together, we conclude that there exists a continous extension $e_{\infty}(\tau) = \lim_{a \to 0} e_{\infty}(\tau + ai)$. Moreover, for all $a > 0$,
    \[
    \limsup_{n \to \infty} |g_n(0) - e_{\infty}(\tau)| \leq \limsup_{n \to \infty} \frac{2a}{\tau^2} + |g_n(a) - e_{\infty}(\tau + ia)| \stackrel{a.s.}{=}  \frac{2a}{\tau^2},
    \]
    and so sending $a \to 0$, we conclude that $g_n(0) \stackrel{\mmp}{\to} e_{\infty}(\tau)$. This proves the desired limit.

    Finally, to show that $e_{\infty}(\tau)$ is positive and real-valued, let $\rho_1,\dots,\rho_d$ denote the eigenvalues of $\frac{1}{n} X_{1:n}^\top X_{1:n}$. By Theorem 1 of \cite{Yin1988}, $\|\frac{1}{n} X_{1:n}^\top X_{1:n} \|_{op} \leq \|P_{\lambda}\|_{\infty}^2 \|\frac{1}{n} \sum_{i=1}^{n} W_i W_i^\top \|_{op} \stackrel{a.s.}{\to}  \|P_{\lambda}\|_{\infty}^2(1 + \sqrt{\gamma})^2$. Thus, 
    \begin{align*}
    e_{\infty}(\tau) & = \lim_{n \to \infty} g_n(0) = \lim_{n \to \infty} \frac{1}{d} \sum_{i=1}^d \frac{1}{\rho_i + \tau} \geq \limsup_{n \to \infty}  \frac{1}{\|\frac{1}{n} X_{1:n}^\top X_{1:n}\|_{op} + \tau}\\
    &  \stackrel{a.s.}{\geq} \frac{1}{\|P_{\lambda}\|_{\infty}^2(1 + \sqrt{\gamma})^2 + \tau} > 0,
    \end{align*}
    as desired. 
\end{proof}

\subsection{Proof of Lemma \ref{lem:ridge_loo_lemma}}

The proof of this lemma follows naturally from our previous leave-one-out calculations along with the asymptotic characterization of the Stieltjes transform of $\frac{1}{n}X_{1:n}^\top X_{1:n}$ given in the previous section. We give the details of this argument now.

\begin{proof}[Proof of Lemma \ref{lem:ridge_loo_lemma}]
The proof of this lemma follows from the same leave-one-out calculations that we applied in Lemma \ref{lem:ridge_stab}. Namely, by the calculations of Lemma \ref{lem:ridge_stab} we have that
\begin{align*}
Y_{n+1} - X_{n+1}^\top\hat{\beta} = & \left(1 - \frac{1}{n+1}X_{n+1}^\top \left(\frac{1}{n+1} X^\top X + \tau I_d \right)^{-1}X_{n+1} \right)(Y_{n+1} - X_{n+1}^\top\hat{\beta}_{(n+1)})\\
& + \frac{\tau}{n+1}X_{n+1}^\top \left(\frac{1}{n+1} X^\top X + \tau I_d \right)^{-1}\hat{\beta}_{(n+1)}.
\end{align*}
Repeating the calculations of Lemma 2, we have that 
\[
\mme\left[\left| \frac{\tau}{n+1}X_{n+1}^\top \left(\frac{1}{n+1} X^\top X + \tau I_d \right)^{-1}\hat{\beta}_{(n+1)} \right|\right] = O\left(\frac{1}{\sqrt{n}}\right),
\]
and thus $\frac{\tau}{n+1}X_{n+1}^\top \left(\frac{1}{n+1} X^\top X + \tau I_d \right)^{-1}\hat{\beta}_{(n+1)} = o_{\mmp}(1)$. Moreover, since $X_{n+1} \independent \hat{\beta}_{(n+1)}$ and since by Lemma \ref{lem:coef_bound}, $\|\hat{\beta}_{(n+1)}\|^2_{2} = O_{\mmp}(1)$, we know that $(Y_{n+1} - X_{n+1}^\top\hat{\beta}_{(n+1)}) = O_{\mmp}(1)$. Thus, to prove the desired result it is sufficient to show that there exists $c_{\infty}>0$, such that
\[
\left| \left( 1-\frac{1}{n+1}X_{n+1}^\top \left(\frac{1}{n+1} X^\top X + \tau I_d \right)^{-1}X_{n+1}\right) -   \frac{1}{1+2\lambda_{n+1}^2 c_{\infty}} \right| \stackrel{\mmp}{\to} 0
\]
By the Sherman-Morrison-Woodbury formula,
\[
 1-\frac{1}{n+1}X_{n+1}^\top \left(\frac{1}{n+1} X^\top X + \tau I_d \right)^{-1}X_{n+1} = \frac{1}{1+ \frac{1}{n+1}\lambda_{n+1}^2 W_{n+1}^\top \left(\frac{1}{n+1} X_{1:n}^\top X_{1:n} + \tau I_d \right)^{-1}W_{n+1}}.
\]
Now, a direct computation shows that 
\begin{align*}
& \mme\left[ \left(\frac{1}{n+1} W_{n+1}^\top \left(\frac{1}{n+1} X_{1:n}^\top X_{1:n} + \tau I_d \right)^{-1}W_{n+1} - \frac{1}{n+1} \text{tr}\left(\left(\frac{1}{n+1} X_{1:n}^\top X_{1:n} + \tau I_d \right)^{-1}\right)\right)^2 \right]\\
& = \frac{\mme[(W_1^2 - 1)^2]}{(n+1)^2}\sum_{i=1}^d \mme\left[ \left(\left(\frac{1}{n+1} X_{1:n}^\top X_{1:n} + \tau I_d \right)^{-1}_{ii}\right)^2 \right]\\
& \ \ \ \ \ \ \ \ \ \ \ \ \ \ \ \ \ \ \ \ \ \   + \frac{\mme[W_1^2]^2}{(n+1)^2}\sum_{i \neq j} \mme\left[\left(\left(\frac{1}{n+1} X_{1:n}^\top X_{1:n} + \tau I_d \right)^{-1}_{ij}\right)^2 \right]\\
& \leq \frac{\mme[W^4_{11}]}{(n+1)^2} \sum_{i=1}^{d} \sum_{j=1}^{d} \mme\left[ \left(\left(\frac{1}{n+1} X_{1:n}^\top X_{1:n} + \tau I_d \right)^{-1}_{ij}\right)^2\right]\\
& = \frac{\mme[W^4_{11}]}{(n+1)^2} \mme\left[ \text{tr}\left( \left(\frac{1}{n+1} X_{1:n}^\top X_{1:n} + \tau I_d \right)^{-1} \left(\frac{1}{n+1} X_{1:n}^\top X_{1:n} + \tau I_d \right)^{-1} \right) \right]\\
& \leq \frac{d\mme[W^4_{11}]}{\tau^2(n+1)^2}\\
& = O\left( \frac{1}{n}\right), 
\end{align*}
and thus,
\[
 \frac{1}{n+1}W_{n+1}^\top \left(\frac{1}{n+1} X_{1:n}^\top X_{1:n} + \tau I_d \right)^{-1}W_{n+1} - \frac{1}{n+1}\text{tr}\left(\left(\frac{1}{n+1} X_{1:n}^\top X_{1:n} + \tau I_d \right)^{-1}\right) \stackrel{\mmp}{\to} 0.
\]
Finally, by Lemma \ref{lem:ridge_tr_final}, we know that there exists $e_{\infty}(\tau) > 0$ such that 
\[
\frac{1}{n+1}\text{tr}\left(\left(\frac{1}{n+1} X_{1:n}^\top X_{1:n} + \tau I_d \right)^{-1}\right) \stackrel{\mmp}{\to} e_{\infty}(\tau).
\]
This proves the desired result.
    
\end{proof}

\subsection{Proof of Lemma \ref{lem:ridge_resid_characterization}}

To aid in future sections it will be helpful to prove a stronger statement than what is given in Lemma \ref{lem:ridge_resid_characterization}. In particular, we have the following result.

\begin{lemma}\label{lem:ridge_CLT}
    Assume that $\epsilon_i$ and $\sqrt{d}\beta^*_i$ have $8$ bounded moments. Let $(\lambda, \epsilon)$ denote a copy of $(\lambda_{n+1},\epsilon_{n+1})$ independent of $Z \sim N(0,1)$ and $\psi$ be any bounded, continuous function. Then, under the setting of Section \ref{sec:high_dim_setting} with $\ell(r) = r^2$,
    \[
    \mme\left[\psi\left( \frac{1}{1+2\lambda^2_{n+1}c_{\infty}}\left( Y_{n+1} - X_{n+1}^\top\hat{\beta}_{(n+1)} \right)\right) \mid \{(X_i,Y_i)\}_{i=1}^n\right] \stackrel{\mmp}{\to} \mme\left[\psi\left(\frac{1}{1+2\lambda^2c_{\infty}}(\epsilon + \lambda N_{\infty}Z)\right)\right].
    \]
    Moreover, for all $1 \leq i < j \leq n+1$,
    \[
    \mme\left[\psi\left(\frac{1}{1+2\lambda^2_{i}c_{\infty}}\left( Y_{i} - X_{i}^\top\hat{\beta}_{(i)} \right)\right) \psi\left(\frac{1}{1+2\lambda^2_{j}c_{\infty}}\left( Y_{j} - X_{j}^\top\hat{\beta}_{(j)} \right)\right) \right] \to \mme\left[\psi\left(\frac{1}{1+2\lambda^2c_{\infty}}(\epsilon + \lambda N_{\infty}Z)\right)\right]^2.
    \]
\end{lemma}

The main step necessary to prove Lemma \ref{lem:ridge_CLT} is a control on the size of the entries of $\hat{\beta}^* - \hat{\beta}_{(n+1)}$. This is given in the following lemma. 

\begin{lemma}\label{lem:ridge_entry_control}
     Assume that $\epsilon_i$ and $\sqrt{d}\beta^*_i$ have $8$ bounded moments. Then, under the setting of Section \ref{sec:high_dim_setting} with $\ell(r) = r^2$, we have that for all $1 \leq j \leq d$,
    \[
    \beta^*_j - \hat{\beta}_j = O_{L_4}\left(\frac{1}{\sqrt{d}}\right).
    \]
\end{lemma}
\begin{proof}
    By assumption we know that $\beta^*_j = O_{L_4}(1/\sqrt{d})$. So, it is sufficient to show that the result holds for $\hat{\beta}_j$. 
    
    Without loss of generality, assume that $j=d$. We will derive a formula for the fitted coefficients when a column of $X$ is left out. For this purpose, let $X_{i,-d}$ denote the vector $(X_{i1},\dots,X_{i(d-1)})$ and $X_{(d)}$ denote the matrix with rows $X_{1,-d},\dots,X_{n+1,-d}$. Similarly, let $\hat{\beta}_{-d}$ and $\beta^*_{-d}$ denote the first $d-1$ coordinates of $\hat{\beta}$ and $\beta^*$, respectively. Let $\hat{\beta}^{(d)} \in \mmr^{d-1}$ denote the fitted coefficients obtained on the dataset $\{( X_{i,-d},\epsilon_i + X_{i,-d}^\top\beta^*_{-d})\}_{i=1}^{n+1}$, i.e.,
    \[
    \hat{\beta}^{(d)}  := \text{argmin}_{\beta \in \mmr^{d-1}} \frac{1}{n+1} \sum_{i=1}^{n+1} (\epsilon_i +  X_{i,-d}^\top(\beta^*_{-d} - \beta))^2 + \tau \|\beta\|_2^2.
    \]
     By the optimality of $\hat{\beta}$ and $\hat{\beta}^{(d)}$, we have the first order conditions
    \[
    \frac{1}{n+1}\sum_{i=1}^{n+1} (X_i^\top(\hat{\beta} - \beta^*) - \epsilon_i)X_i + \tau \hat{\beta} = 0 \ \ \ \text{ and } \ \  \ \frac{1}{n+1}\sum_{i=1}^{n+1} (X_{i,-d}^\top(\hat{\beta}^{(d)} - \beta^*_{-d}) - \epsilon_i)X_{i,-d} + \tau \hat{\beta}^{(d)} = 0.
    \]
    In particular, taking the difference between the above two equations, we find that
    \begin{align*}
    & \left(\begin{matrix} \frac{1}{n+1} \sum_{i=1}^{n+1} (X_{i,-d}^\top(\hat{\beta}_{-d} - \hat{\beta}^{(d)}))X_{i,-d} + \tau(\hat{\beta}_{-d} - \hat{\beta}^{(d)}) \\ \frac{1}{n+1} \sum_{i=1}^{n+1}  (X_i^\top\hat{\beta} - Y_i)X_{id} + \tau \hat{\beta}_d \end{matrix} \right) =  \left(\begin{matrix} \frac{1}{n+1}\sum_{i=1}^{n+1}(X_{id}\beta^*_d - X_{id} \hat{\beta}_d) X_{i,-d} \\ 0 \end{matrix} \right)\\
     \implies & \left(\begin{matrix} \hat{\beta}_{-d} - \hat{\beta}^{(d)} \\  \hat{\beta}_d\end{matrix} \right) = \left(\begin{matrix}   (\beta^*_d - \hat{\beta}_d) \left(\frac{1}{n+1}X_{(d)}^\top X_{(d)} + \tau I_{d-1} \right)^{-1} \left(\frac{1}{n+1} \sum_{i=1}^{n+1} X_{id}X_{i,-d}\right) \\ \frac{\frac{1}{n+1} \sum_{i=1}^{n+1} (Y_i - X_{i,-d}^\top\hat{\beta}_{-d})X_{id} }{\tau + \frac{1}{n+1} \sum_{i=1}^{n+1} X_{id}^2} \end{matrix} \right)\\
     \implies &  \hat{\beta_d} = \frac{ \frac{1}{n+1} \sum_{i=1}^{n+1} (Y_i - X_{i-d}^\top \hat{\beta}^{(d)}) X_{id} }{\tau + \frac{1}{n+1} \sum_{i=1}^{n+1} X_{id}^2 - \frac{1}{n+1} \sum_{i=1}^{n+1} X_{id}X^\top_{i,-d}   \left(\frac{1}{n+1}X_{(d)}^\top X_{(d)} + \tau I_{d-1} \right)^{-1}  \frac{1}{n+1} \sum_{i=1}^{n+1} X_{id}X_{i,-d} }\\
     & \ \ \ \ \ - \frac{  \beta^*_d (\frac{1}{n+1}\sum_{i=1}^{n+1} X_{id} X_{i,-d})^\top \left(\frac{1}{n+1}X_{(d)}^\top X_{(d)} + \tau I_{d-1} \right)^{-1} \frac{1}{n+1}\sum_{i=1}^{n+1} X_{id} X_{i,-d}}{\tau + \frac{1}{n+1} \sum_{i=1}^{n+1} X_{id}^2 - \frac{1}{n+1} \sum_{i=1}^{n+1} X_{id}X^\top_{i,-d}   \left(\frac{1}{n+1}X_{(d)}^\top X_{(d)} + \tau I_{d-1} \right)^{-1}  \frac{1}{n+1} \sum_{i=1}^{n+1} X_{id}X_{i,-d} }. 
    \end{align*}
    Now, focusing on the denominator and letting $X_{\cdot d} := (X_{1d},\dots, X_{(n+1)d})$ be the $d_{th}$ column vector of the covariates, we have that 
    \begin{align*}
    & \frac{1}{n+1} \sum_{i=1}^{n+1} X_{id}^2 - \frac{1}{n+1} \sum_{i=1}^{n+1} X_{id}X^\top_{i,-d}   \left(\frac{1}{n+1}X_{(d)}^\top X_{(d)} + \tau I_{d-1} \right)^{-1}  \frac{1}{n+1} \sum_{i=1}^{n+1} X_{id}X_{i,-d}\\
    & = \frac{1}{n+1} X_{\cdot d}^\top \left(I_n - \frac{1}{n+1}X_{(d)} \left(\frac{1}{n+1}X_{(d)}^\top X_{(d)} + \tau I_{d-1} \right)^{-1} X^\top_{(d)} \right) X_{\cdot d}\\
    & \geq 0,
    \end{align*}
    where here we have used the fact that  $I_d - \frac{1}{n+1}X_{(d)} \left(\frac{1}{n+1}X_{(d)}^\top X_{(d)} + \tau I_{d-1} \right)^{-1} X^\top_{(d)}$ is positive semidefinite (to see this one can simply consider the singular value decomposition of $X_{(d)})$. So, we find that 
    \begin{align*}
    |\hat{\beta}_d| & \leq  \left| \frac{1}{\tau(n+1)} \sum_{i=1}^{n+1} (Y_i - X_{i,-d}^\top\hat{\beta}^{(d)}) X_{id} \right| + \frac{|\beta^*_d|}{\tau} \frac{1}{n+1}\sum_{i=1}^{n+1} X_{id}^2 \\
    & \leq  \left| \frac{1}{\tau(n+1)} \sum_{i=1}^{n+1} (\epsilon_i + X_{i,-d}^\top \beta^*_{-d} - X_{i,-d}^\top\hat{\beta}^{(d)}) X_{id} \right| + 2\frac{|\beta^*_d|}{\tau} \frac{1}{n+1}\sum_{i=1}^{n+1} X_{id}^2.
    \end{align*}
    The fact that the second term is $O_{L_4}(1/\sqrt{d})$ follows immediately by the fact that by assumption  $|\beta^*_d| = O_{L_8}(1/\sqrt{d})$ and $\frac{1}{n+1}\sum_{i=1}^{n+1} X_{id}^2 = O_{L_8}(1)$. For the first term, note that conditional on $\{\epsilon_i + X_{i,-d}^\top \beta^*_{-d} - X_{i,-d}^\top\hat{\beta}^{(d)}\}_{i=1}^{n+1}$, this becomes a sum of independent, mean zero, random variables. So, we find that in order to prove that the first term is $O_{L_4}(1/\sqrt{d})$, it is sufficient to show that 
    \[
    \epsilon_i + X_{i,-d}^\top \beta^*_{-d} - X_{i,-d}^\top\hat{\beta}^{(d)} = O_{L_4}(1).
    \]
    The fact that the first two terms are $O_{L_4}(1)$ is clear. To control $X_{i,-d}^\top\hat{\beta}^{(-d)}$, let $\hat{\beta}^{(d)}_{(i)}$ denote the fitted coefficients obtained when we additionally leave out point $i$. Then, to prove the desired result it is sufficient to show that $X_{i,-d}^\top\hat{\beta}^{(d)}_{(i)} = O_{L_4}(1)$ and $X_{i,-d}^\top\hat{\beta}^{(d)} - X_{i,-d}^\top\hat{\beta}^{(d)}_{(i)} = O_{L_4}(1)$. This is done in Lemmas \ref{lem:ridge_loo_predic_norm_bound} and \ref{lem:ridge_loo_norm_bound} below.
\end{proof} 

\begin{lemma}\label{lem:ridge_loo_predic_norm_bound}
Assume that $\epsilon_i$ and $\sqrt{d}\beta^*_i$ have $8$ bounded moments. Then, under the setting of Section \ref{sec:high_dim_setting} with $\ell(r) = r^2$,
\[
 X_1^\top\hat{\beta}_{(1)} = O_{L_8}(1)
\]
\end{lemma}
\begin{proof}
    Since $\lambda_1$ is bounded, we have that $|X_1^\top\hat{\beta}_{(1)}| = |\lambda_1 W_1^\top\hat{\beta}_{(1)}| \leq \|P_{\lambda}\|_{\infty}|W_1^\top\hat{\beta}_{(1)}|$. Thus, it is enough to prove the result for $|W_1^\top\hat{\beta}_{(1)}|$. For this, using the fact that $W_{1} \independent  \hat{\beta}_{(1)}$ we directly compute that
    \begin{align*}
    & \mme[(W_{1}^\top \hat{\beta}_{(1)})^8]\\
    & \leq  \mme[W_{11}^8] \sum_{j=1}^d \mme[( \hat{\beta}_{(1)})_j^8] + \frac{1}{2}\binom{8}{2}\mme[W_{11}^2]\mme[W_{11}^6]\sum_{j \neq k} \mme[( \hat{\beta}_{(1)})_j^2 (\hat{\beta}_{(1)})_k^6] + \frac{1}{2}\binom{8}{4}\mme[W_{11}^4]^2 \sum_{j \neq k} \mme[( \hat{\beta}_{(1)})_j^4 (\hat{\beta}_{(1)})_k^4]\\
    & \ \ \ \ \ \   + \frac{1}{6}\binom{8}{4} \binom{4}{2}\mme[W_{11}^4]\mme[W_{11}^2]^2 \sum_{j \neq k \neq i} \mme[( \hat{\beta}_{(1)})_j^4 (\hat{\beta}_{(1)})_k^2(\hat{\beta}_{(1)})_i^2]\\
    & \ \ \ \ \ \ + \frac{1}{24}\binom{8}{2} \binom{6}{2} \binom{4}{2} \mme[W_{11}^2]^4 \sum_{j \neq k \neq i \neq r} \mme[( \hat{\beta}_{(1)})_j^2 (\hat{\beta}_{(1)})_k^2(\hat{\beta}_{(1)})_i^2 (\hat{\beta}_{(1)})_r^2]  \\
    &  \leq 105\mme[W_{11}^8] \mme[\| \hat{\beta}_{(i)}\|_2^8],
    \end{align*}
    and applying Lemma \ref{lem:coef_bound} we arrive at the desired bound,
    \[
     \mme[W_{11}^8]\mme[\| \hat{\beta}_{(1)}\|_2^8]  \leq 105\mme[W_{11}^8]\mme\left[\left(\frac{1}{\tau n} \sum_{i=2}^{n+1}Y_i^2\right)^4 \right] = O(1).
    \]
\end{proof}

\begin{lemma}\label{lem:ridge_loo_norm_bound}
    Assume that $\epsilon_i$ and $\sqrt{d}\beta^*_i$ have $8$ bounded moments. Then, under the setting of Section \ref{sec:high_dim_setting} with $\ell(r) = r^2$,
    \[
    X_{1}^\top\hat{\beta} - X_1^\top\hat{\beta}_{(1)} = O_{L_4}(1).
    \]
\end{lemma}
\begin{proof}
    The proof of this result is nearly identical to the proof of Lemma \ref{lem:ridge_stab}. Namely, repeating the steps of Lemma \ref{lem:ridge_stab} we know that 
    \begin{align*}
    \left| X_{1}^\top\hat{\beta} - X_1^\top\hat{\beta}_{(1)} \right| & = \left|\frac{1}{n+1} X_1^\top \frac{1}{n+1} \left(\frac{1}{n+1} X^\top X + \tau I_d \right)^{-1}X_1 \left( Y_{1} - X_1^\top \hat{\beta}_{(1)} \right) \right|\\
    & \ \ \ \ +  \left|X_1^\top  \left(\frac{1}{n+1} X^\top X + \tau I_d \right)^{-1} \frac{\tau}{n+1} \hat{\beta}_{(1)}\right|\\
    & \leq \frac{\|X_1\|_2^2}{(n+1)\tau} \left( Y_{1} - X_1^\top \hat{\beta}_{(1)} \right) + \frac{\|X_1\|_2 \|\hat{\beta}_{(1)}\|_2}{n+1}
    \end{align*}
    By Lemma \ref{lem:ridge_loo_predic_norm_bound} we know that $ X_1^\top \hat{\beta}_{(1)}  = O_{L_8}(1)$. Moreover, our assumptions clearly imply that $\frac{\|X_1\|_2^2}{(n+1)\tau}, Y_{1} = O_{L_8}(1)$. This controls the first term. For the second term, by the Cauchy-Schwartz inequality,
    \[
    \mme\left[\left(\frac{\|X_1\|_2 \|\hat{\beta}_{(1)}\|_2}{n+1}\right)^4\right] \leq \mme\left[ \frac{\|X_1\|^8_2}{(n+1)^8}\right]^{1/2} \mme[\hat{\beta}_{(1)}\|_2^8]^{1/2} = O\left(\frac{1}{n^2} \right)\mme[\|\hat{\beta}_{(1)}\|_2^8]^{1/2},
    \]
    and finally by Lemma \ref{lem:coef_bound},
    \[
    \mme[\| \hat{\beta}_{(1)}\|_2^8]  \leq \mme\left[\left(\frac{1}{\tau n} \sum_{i=2}^{n+1}Y_i^2\right)^4 \right] = O(1).
    \]
\end{proof}

With the above results in hand we are ready to prove Lemma \ref{lem:ridge_CLT}
\begin{proof}[Proof of Lemma \ref{lem:ridge_CLT}]
    We will first show that the second part of the lemma follows from the first. Let $\hat{\beta}_{(i,j)}$ denote the fitted coefficients when both the data points $i$ and $j$ are removed from the dataset. Then, by Lemma \ref{lem:ridge_stab},
    \begin{align*}
    & \mme\left[\psi\left(\frac{1}{1+2\lambda^2_{i}c_{\infty}}\left( Y_{i} - X_{i}^\top\hat{\beta}_{(i)} \right)\right) \psi\left(\frac{1}{1+2\lambda^2_{j}c_{\infty}}\left( Y_{j} - X_{j}^\top\hat{\beta}_{(j)} \right)\right) \right]\\
    & = \mme\left[\psi\left(\frac{1}{1+2\lambda^2_{i}c_{\infty}}\left( Y_{i} - X_{i}^\top\hat{\beta}_{(i,j)} \right)\right) \psi\left(\frac{1}{1+2\lambda^2_{j}c_{\infty}}\left( Y_{j} - X_{j}^\top\hat{\beta}_{(i,j)} \right)\right) \right] + o(1)\\
    & = \mme\left[\mme\left[ \psi\left(\frac{1}{1+2\lambda^2_{i}c_{\infty}}\left( Y_{i} - X_{i}^\top\hat{\beta}_{(i,j)} \right)\right) \mid \{(X_k,Y_k)\}_{k \neq i,j} \right]^2\right] + o(1),
    \end{align*}
    and applying the first part of the lemma we find that 
    \[
    \mme\left[ \psi\left(\frac{1}{1+2\lambda^2_{i}c_{\infty}}\left( Y_{i} - X_{i}^\top\hat{\beta}_{(i,j)} \right)\right) \mid \{(X_k,Y_k)\}_{k \neq i,j} \right]^2 = \mme\left[\psi\left(\frac{1}{1+2\lambda^2c_{\infty}}(\epsilon + \lambda N_{\infty}Z)\right)\right]^2 + o_{L_1}(1).
    \]
    This proves the second part of the lemma.

    To prove the first part of the lemma, write 
    \[
    \psi\left( \frac{1}{1+2\lambda^2_{n+1}c_{\infty}}\left( Y_{n+1} - X_{n+1}^\top\hat{\beta}_{(n+1)}\right) \right) = \psi\left( \frac{1}{1+2\lambda^2_{n+1}c_{\infty}}\left( \epsilon_{n+1} + \lambda_{n+1} W_{n+1}^\top(\beta^* - \hat{\beta}_{(n+1)})\right) \right).
    \]
    Since $W_{n+1} \independent (\lambda_{n+1}, \epsilon_{n+1})$, by marginalizing over $(\lambda_{n+1},\epsilon_{n+1})$ and applying the bounded convergence theorem we find that it is sufficient to show that for any bounded, continuous functions $\phi$,
    \[
    \mme[\phi(W_{n+1}^\top (\beta^* - \hat{\beta}_{(n+1)})) \mid \{(X_i,Y_i)\}_{i=1}^{n}] \stackrel{\mmp}{\to} \mme[\phi(N_{\infty}Z)].
    \]
    
    Suppose $\phi$ is bounded with three bounded derivatives. Then, by the Lindenberg central limit theorem (e.g. Theorem 2.1.4 of \cite{Stroock_2010}), there exists a constant $C>0$, such that for all $\epsilon > 0$,
    \begin{align*}
    & \left| \mme[\phi(W_{n+1}^\top (\beta^* -  \hat{\beta}_{(n+1)})) \mid \{(X_i,Y_i)\}_{i=1}^{n}]  - \mme[\phi(N_{\infty}Z)]\right|\\
    & \leq C\left(\epsilon\|\phi'''\|_{\infty} + \frac{1}{\epsilon} \frac{\sum_{i=1}^d|\beta^*_i - (\hat{\beta}_{(n+1)})_i|^3}{(\sum_{i=1}^d(\beta^*_i - (\hat{\beta}_{(n+1)})_i)^2)^{3/2}} (\|\phi''\|_{\infty} + \|\phi'''\|_{\infty} ) \right). \end{align*}
    Now, by Lemma \ref{lem:ridge_norm_conv}, $(\sum_{i=1}^d(\beta^*_i - (\hat{\beta}_{(n+1)})_i)^2)^{3/2} = O_{\mmp}(1)$, while by Lemma \ref{lem:ridge_entry_control}, 
    \[
    \mme\left[\sum_{i=1}^d|\beta^*_i - (\hat{\beta}_{(n+1)})_i|^3 \right] = O\left( \frac{1}{\sqrt{d}}\right) \implies \sum_{i=1}^d|\beta^*_i - (\hat{\beta}_{(n+1)})_i|^3 = o_{\mmp}(1).
    \]
    Thus, we find that for all $\epsilon > 0$,
    \[
    \left| \mme[\phi(W_{n+1}^\top (\beta^* -   \hat{\beta}_{(n+1)})) \mid \{(X_i,Y_i)\}_{i=1}^{n}]  - \mme[\phi(N_{\infty}Z)]\right| \leq O\left(\epsilon + \frac{1}{\epsilon} o_{\mmp}(1)\right),
    \]
    and so choosing $\epsilon$ appropriately we conclude that 
    \[
    \mme[\phi(W_{n+1}^\top (\beta^* -   \hat{\beta}_{(n+1)})) \mid \{(X_i,Y_i)\}_{i=1}^{n}] \stackrel{\mmp}{\to} \mme[\phi(N_{\infty}Z)].
    \]
    This proves the desired result for bounded functions with three bounded derivatives. To prove the result for bounded, continuous functions one may use standard arguments to approximate the later class by the former. 
\end{proof}

\subsection{Proof of Lemma \ref{lem:ridge_quant_conv}}

\begin{proof}[Proof of Lemma \ref{lem:ridge_quant_conv}]
    To prove the first part of the lemma, one simply verifies that by the results of Lemma \ref{lem:ridge_CLT},
    \[
    \mme\left[\frac{1}{n+1} \sum_{i=1}^{n+1} \psi(Y_i - X_i^\top\hat{\beta}) \right] \to \mme\left[  \frac{1}{1+2\lambda^2c_{\infty}}\left( \epsilon + \lambda N_{\infty} Z)\right)  \right],
    \]
    while,
    \[
    \text{Var}\left(\frac{1}{n+1} \sum_{i=1}^{n+1} \psi(Y_i - X_i^\top\hat{\beta}) \right)  = o(1),
    \]
    (for detailed steps demonstrating this variance bound see the proof of Lemma \ref{lem:conv_in_dist} below).

    The second part of the lemma follows immediately from the fact that the limiting distribution, $\frac{1}{1+\lambda^2c_{\infty}}( \epsilon + \lambda N_{\infty} Z)$ has a strictly increasing cumulative distribution function (see Lemma \ref{lem:asymptotic_prox_dist_facts}) and standard results relating convergence in distribution to convergence of the associated quantiles (see Lemma \ref{lem:quantile_conv_fact}).
\end{proof}

\subsection{Proof of Theorem \ref{thm:ridge_asymp_cov}}

With the above preliminaries in hand, the proof of Theorem \ref{thm:ridge_asymp_cov} is identical to the proof of Theorem \ref{thm:conv_tau_fixed}, which considers losses with bounded derivatives. Thus, we omit the proof of this result and refer the interested reader to Theorem \ref{thm:conv_tau_fixed}.

\section{Proofs for Section \ref{sec:general_losses}}

\subsection{Additional notation}

For ease of notation we let ${R}_i := Y_i - X_i^\top \hat{\beta}$ and $\{{R}^{(i)}_{j}\}_{j=1}^{n+1} = \{Y_j - X_j^\top \hat{\beta}_{(i)} \}_{j=1}^{n+1}$ denote the residuals and leave-one-out residuals. Throughout, we will let $(Z,\lambda,\epsilon)$ denote a sample from $\mathcal{N}(0,1)\otimes P_{\lambda} \otimes P_{\epsilon}$. 

\subsection{Proof of Theorem \ref{thm:conv_tau_fixed}}

The proof of Theorem \ref{thm:conv_tau_fixed} will make extensive use of the following results of \citet{EK2018}.

\begin{theorem}\label{thm:EK_facts}
Under Assumptions 1-3 from Section \ref{sec:general_losses}, we have that for any fixed $k \in \mmn$ and $1 \leq i \leq n+1$,
\begin{enumerate}
    \item 
    \[
    \sup_{1 \leq i \leq n + 1} \sup_{j \neq i} |R^{(i)}_j - R_j| \leq O_{L_k}\left(\frac{\textup{polylog}(n)}{n^{1/2}} \right).
    \]
    \item
    \[
    \sup_{1 \leq i \leq n+1} \left|R_{i} - \textup{prox}(c_i\ell)(R^{(i)}_i)\right| \leq O_{L_k}\left(\frac{\textup{polylog}(n)}{n^{1/2}} \right).
    \]
    where
    \[
    c_i := \frac{1}{n+1} X_i^\top\left( \frac{1}{n+1}\sum_{j \neq i} \ell''(R^{(i)}_j)X_jX_j^\top + 2\tau I_d \right)X_i.
    \]
    \item
    There exists a constant $c_{\infty} > 0$ such that for all $i$,
    \[
    |c_i - \lambda_i^2 c_{\infty}| \stackrel{\mmp}{\to} 0.
    \]
    \item 
    There exists a constant $N_{\infty} > 0$ such that,
    \[
    \|\hat{\beta} - \beta^*\|_2 \stackrel{\mmp}{\to} N_{\infty}.
    \]
    \item 
    For any bounded continuous function $\psi$ and any $1 \leq i \leq n+1$,
    \[
    \mme[\psi(\textup{prox}(\lambda_i^2c_{\infty}\ell)(R^{(i)}_i)) \mid \{(X_j,Y_j)\}_{j \neq i}]  \stackrel{\mmp}{\to} \mme[\psi(\textup{prox}(\lambda^2c_{\infty}\ell)(\epsilon + \lambda N_{\infty}Z))],
    \]
    where $(Z,\lambda,\epsilon) \sim N(0,1) \otimes P_{\lambda} \otimes P_{\epsilon}$. 
\end{enumerate}
\end{theorem}
\begin{proof}
    The first two parts of this result follow from Theorem 3.9 of \citet{EK2018}, while the last part follows from the proof of Lemma 3.23 of the same article. Parts 3 and 4 follow from Theorem 2.6, Corollary 3.17, and Lemma 3.26 of \citet{EK2018} with exception that in that article $c_{\infty}$ and $N_{\infty}$ are only shown to be non-negative. To verify that these scalars must be positive, let $(Z,\lambda,\epsilon) \sim N(0,1) \otimes P_{\lambda} \otimes P_{\epsilon}$. Then, by Theorem 2.6 of \citet{EK2018}, $(c_{\infty}, N_{\infty})$ are the solutions to the following system of equations in variables $(c,N)$,
    \[
    \begin{cases}
    \mme\left[\left(1 + c\lambda^2\ell''(\text{prox}(\lambda^2c \ell)(\epsilon + \lambda N Z)) \right)^{-1} \right] = 1 - \gamma + 2\tau c,\\
    \gamma^2 N^2 = \gamma \mme[(c\lambda\ell'(\text{prox}(\lambda^2c \ell)(\epsilon + \lambda N Z))^2] + 4\tau^2 \mme[\|\beta^*\|^2] c^2.
    \end{cases}
    \]
    Since $\gamma > 0$ we clearly see that $c=0$ cannot be a solution to the first equation. Thus, we must have $c_{\infty} > 0$. Plugging in $c_{\infty}>0$ into the second equation, it becomes clear that $N_{\infty} > 0$.
\end{proof}

The above facts almost immediately imply the asymptotic characterization of the test residual given in Lemma \ref{lem:gen_resid_char} from the main text. Namely, part 2 of Theorem \ref{thm:EK_facts} gives $R_{n+1} - \textup{prox}(c_{n+1} \ell)(R_{n+1}^{(n+1)}) \stackrel{\mmp}{\to} 0$. Lemma \ref{lem:gen_resid_char} then immediately follows from the following result which shows that $c_{n+1}$ can be replaced by $\lambda_{n+1}^2 c_{\infty}$.

\begin{lemma}\label{lem:prox_resid_convergence}
Suppose that Assumptions 1-3 from Section \ref{sec:general_losses} hold. Then, as $n \to \infty$,
\[
\textup{prox}(c_{n+1} \ell)(R_{n+1}^{(n+1)}) - \textup{prox}(\lambda_{n+1}^2 c_{\infty} \ell)\left(\epsilon_{n+1} + X_{n+1}^\top(\beta^* - \hat{\beta}^{(n+1)})\right) \stackrel{\mmp}{\to} 0.
\]
\end{lemma}
\begin{proof}
    Since the proximal function is continuous (see Lemma \ref{lem:cont_diff_prox}) and $c_{n+1} \stackrel{\mmp}{\to} c_{\infty}$ it is sufficient to apply a variant of the continuous mapping theorem. Formally, we will show in Lemma \ref{lem:cont_map} that the desired conclusion holds so long as the random variables appearing in the above expression are of size $O_{\mmp}(1)$. The only subtlety here is to verify that $X_{n+1}^\top(\beta^* - \hat{\beta}^{(n+1)}) = O_{\mmp}(1)$. For this, we note that by part 4 of Theorem \ref{thm:EK_facts} we have that for any $M>N_{\infty} + 1 > 0$,
    \begin{align*}
    \mmp(X_{n+1}^\top(\beta^* - \hat{\beta}^{(n+1)}) \geq M) & \leq \mmp(\|\beta^* - \hat{\beta}^{(n+1)}\|_2 > N_{\infty} + 1)\\
    & \hspace{2cm} + \frac{\mme[(X_{n+1}^\top(\beta^* - \hat{\beta}^{(n+1)}))^2\bone\{\|\beta^* - \hat{\beta}^{(n+1)}\|_2 \leq N_{\infty} + 1\}]}{M^2} \\
    & \leq o(1) + \|P_{\lambda}\|^2_{\infty}\frac{(N_{\infty}+1)^2}{M^2}.
    \end{align*}
    This proves the desired result.
\end{proof}

Our next result shows that the empirical distribution of the residuals matches the limit identified in Part 5 of Theorem \ref{thm:EK_facts}.

\begin{lemma}\label{lem:conv_in_dist}
    Suppose that Assumptions 1-3 from Section \ref{sec:general_losses} hold. Let $\psi$ be any bounded, continous function. Then, as $n \to \infty$,
    \[
    \lim_{n \to \infty} \frac{1}{n+1} \sum_{i=1}^{n+1} \psi(R_i) \stackrel{\mmp}{\to} \mme[\psi(\textup{prox}(\lambda^2c_{\infty}\ell)(\epsilon + \lambda N_{\infty}Z))],
    \]
    where $(Z,\lambda,\epsilon) \sim N(0,1) \otimes P_{\lambda} \otimes P_{\epsilon}$.
\end{lemma}
\begin{proof}
    We compute the mean and variance of the left-hand side. Applying parts 2 and 5 of Theorem \ref{thm:EK_facts}, Lemma \ref{lem:prox_resid_convergence}, and the bounded convergence theorem we have that
    \begin{align*}
    \mme\left[\frac{1}{n+1} \sum_{i=1}^{n+1} \psi(R_i)\right] & = \mme\left[\psi(R_i)\right] = \mme[\psi(\text{prox}(\lambda_i^2c_i\ell)(R^{(i)}_i))] + o(1) =  \mme[\psi(\text{prox}(\lambda_i^2c_{\infty}\ell))(R^{(i)}_i)] + o(1)\\
    & = \mme[\psi(\text{prox}(\lambda_i^2c_{\infty}\ell)(\epsilon + \lambda_iN_{\infty}Z))] + o(1),
    \end{align*}
    as desired. To evaluate the variance note that by the exchangeability of the residuals
    \begin{align*}
        \text{Var}\left( \frac{1}{n+1} \sum_{i=1}^{n+1} \psi(R_i)\right) & = \frac{\mme[\psi(R_i)]^2}{n+1} + \frac{(n+1)n}{(n+1)^2} (\mme[\psi(R_1)\psi(R_2)] - \mme[\psi(R_1)]\mme[\psi(R_2)])\\
        & \leq O\left( \frac{1}{n}\right) + |\mme[\psi(R_1)\psi(R_2)] - \mme[\psi(R_1)]\mme[\psi(R_2)]|.
    \end{align*}
    Letting $R^{(1,2)}_i = Y_i - X_i^\top \hat{\beta}_{(1,2)}$ denote the residual for point $i$ obtained when the first two points are excluded from the fit and applying parts 1, 2 and 5 of Theorem \ref{thm:EK_facts}, Lemma \ref{lem:prox_resid_convergence}, and the bounded convergence theorem we have that 
    \begin{align*}
    & \mme[\psi(R_1)\psi(R_2)]  = \mme\left[\psi\left(\text{prox}(\lambda_1^2 c_{\infty}\ell)\left(R_1^{(1,2)}\right)\right) \psi\left(\text{prox}(\lambda_2^2 c_{\infty}\ell)\left(R_2^{(1,2)}\right)\right)\right] + o(1)\\
    & = \mme [\psi(\text{prox}(\lambda^2 c_{\infty}\ell)(\epsilon + N_{\infty} \lambda Z)]^2  + o(1).
    \end{align*}
     Plugging this and our approximation on the first moment into our expression for the variance yields that $\text{Var}\left( \frac{1}{n+1} \sum_{i=1}^{n+1} \psi(R_i)\right) = o(1)$, as desired.

\end{proof}

The final preliminary that we will need before proving Theorem \ref{thm:conv_tau_fixed} is an asymptotic characterization of the empirical quantiles (namely Lemma \ref{lem:gen_quant_conv} from the main text). This characterization will follow immediately from the convergence of the empirical distribution established above and the following standard lemma.

\begin{lemma}\label{lem:quantile_conv_fact}
    Let $\{P_n\}$ and $P$ be a sequence of random probability measures and a fixed probability measure on $\mmr$. Suppose that for all bounded, continuous functions $\psi$,
    \[
    \mme_{P_n}[\psi(X)] \stackrel{\mmp}{\to} \mme_P[\psi(X)],
    \]
    and let $\alpha \in (0,1)$ be such that the cumulative distribution function of $P$ is strictly increasing at $\textup{Quantile}(1-\alpha,P)$. Then,
    \[
    \textup{Quantile}(1-\alpha,P_n) \stackrel{\mmp}{\to} P.
    \]
\end{lemma}
\begin{proof}
    Fix any $\epsilon > 0$. Let 
    \[
    h^{+}_{\epsilon}(x) := \begin{cases}
        1, \text{ if } x \leq \textup{Quantile}(1-\alpha,P) + \epsilon,\\
        (\textup{Quantile}(1-\alpha,P) + 2\epsilon - x)/\epsilon, \text{ if } \textup{Quantile}(1-\alpha,P) + \epsilon < x \leq \textup{Quantile}(1-\alpha,P) + 2\epsilon\\
        0, \text{ if } x > \textup{Quantile}(1-\alpha,P) + 2\epsilon,
    \end{cases} 
    \]
    and
    \[
    h^-_{\epsilon} := \begin{cases}
        1, \text{ if } x \leq \textup{Quantile}(1-\alpha,P) - 2\epsilon,\\
        (\textup{Quantile}(1-\alpha,P) - \epsilon - x)/\epsilon, \text{ if } \textup{Quantile}(1-\alpha,P) - 2\epsilon < x \leq \textup{Quantile}(1-\alpha,P) - \epsilon\\
        0, \text{ if } x > \textup{Quantile}(1-\alpha,P) - \epsilon.
    \end{cases}
    \]
    Since the cumulative distribution function of $P$ is strictly increasing at $\text{Quantile}(1-\alpha,P)$ we must have that $\mme_P[h^-(X)] < 1-\alpha < \mme_P[h^+(X)]$. Moreover, since $h^-$ and $h^+$ are bounded and continuous we have that with probability approaching $1$, 
    \begin{align*}
    \mmp_{P_n}(X \leq \text{Quantile}(1-\alpha,P) - 2\epsilon) & \leq \mme_{P_n}[h^-(X)] <  1-\alpha\\
    & < \mme_{P_n}[h^+(X)] \leq \mmp_{P_n}(X \leq \text{Quantile}(1-\alpha,P) + 2\epsilon),
    \end{align*}
    and thus with probability approaching 1, 
    \[
    \text{Quantile}(1-\alpha,P_n) \in [\text{Quantile}(1-\alpha,P) - 2\epsilon, \text{Quantile}(1-\alpha,P) + 2\epsilon].
    \]
    This proves the desired result.
\end{proof}

To apply this result to the empirical measure over the absolute residuals we need to show that the limiting distribution has a strictly increasing cumulative distribution function. This is handled in the next lemma.

\begin{lemma}\label{lem:asymptotic_prox_dist_facts}
    Let $(Z,\lambda,\epsilon) \sim N(0,1) \otimes P_{\lambda} \otimes P_{\epsilon}$, where $P_{\epsilon}$ is arbitrary and $\mmp_{P_{\lambda}}(\lambda > 0) = 1$. Let $\ell$ denote any convex, twice differentiable loss function, with a bounded second derivative. Then for any constants $\tilde{c} \geq 0$, $\tilde{N} > 0$, the cumulative distribution function of $|\text{prox}(\lambda^2\tilde{c} \ell)(\epsilon + \lambda \tilde{N} Z)|$ is strictly increasing and continous on $\mmr_{\geq 0}$ .
\end{lemma}
\begin{proof}
    We will first show that the cumulative distribution function is strictly increasing. Fix any $x \geq 0$ and $\delta > 0$. Our goal will be to show that 
    \[
    \mmp(|\text{prox}(\lambda^2 \tilde{c}\ell)(\epsilon + \lambda \tilde{N} Z)| \leq x) < \mmp(|\text{prox}(\lambda^2\tilde{c} \ell)(\epsilon + \lambda \tilde{N} Z)| \leq x + \delta).
    \]
    To begin, first note that the proximal function is strictly increasing in its second argument. Namely, by Lemma \ref{lem:cont_diff_prox} we have that for any fixed $c \geq 0$,
    \[
    \frac{d}{dx} \text{prox}(c\ell)(x) = \frac{1}{1 + c\ell''(\text{prox}(c\ell)(x))} > 0.
    \]
    Since $\ell''(\cdot)$ is bounded, this clearly implies that $\text{prox}(c\ell)(\cdot)$ has image equal to $\mathbb{R}$ and thus, we find that $\text{prox}(c\ell)(\cdot)$ admits an increasing inverse function, $\text{prox}(c\ell)^{-1}:\mmr \to \mmr$. Using this fact, we compute that
    \begin{align*}
        & \mmp(|\text{prox}(\lambda^2 \tilde{c}\ell)(\epsilon + \lambda \tilde{N} Z)| \leq x+ \delta)
         = \mmp(-x - \delta \leq \text{prox}(\lambda^2  \tilde{c}\ell)(\epsilon + \lambda \tilde{N} Z) \leq x + \delta)\\
        & = \mme[\mmp(-x - \delta \leq \text{prox}(\lambda^2  \tilde{c}\ell)(\epsilon + \lambda \tilde{N} Z) \leq x + \delta \mid \lambda , \epsilon)]\\
        & = \mme\left[\mmp\left( \frac{\text{prox}(\lambda^2  \tilde{c})^{-1}(-x - \delta) - \epsilon}{ \tilde{N}} \leq \lambda Z \leq \frac{\text{prox}(\lambda^2  \tilde{c})^{-1}(x + \delta) - \epsilon}{ \tilde{N}} \mid \lambda , \epsilon\right)\right]\\
        & > \mme\left[ \mmp \left( \frac{\text{prox}(\lambda^2  \tilde{c})^{-1}(-x ) - \epsilon}{ \tilde{N}} \leq \lambda Z \leq \frac{\text{prox}(\lambda^2  \tilde{c})^{-1}(x ) - \epsilon}{ N_{\infty}}  \mid \lambda , \epsilon\right)\right]\\
        & =  \mmp(|\text{prox}(\lambda^2  \tilde{c}\ell)(\epsilon + \lambda \tilde{N} Z)| \leq x ),
    \end{align*}
    where the inequality uses the fact that $\text{prox}(\lambda^2  \tilde{c})^{-1}(\cdot)$ is strictly increasing and $\lambda Z$ is supported on $\mmr$ whenever $\lambda \neq 0$ (which by assumption occurs with probability one). Thus, the cumulative distribution function is strictly increasing.

    To show that the cumulative distribution function is continuous fix any $x,x' \geq 0$. We will consider the limit of the cumulative distribution function as $x' \to x$. To do this, we proceed as above and write 
    \begin{align*}
         \mmp(|\text{prox}(\lambda^2 \tilde{c}\ell)(\epsilon + \lambda \tilde{N} Z)| \leq x') & =  \mmp \left( \frac{\text{prox}(\lambda^2 \tilde{c})^{-1}(-x' ) - \epsilon}{ \tilde{N}} \leq  \lambda Z \leq \frac{\text{prox}(\lambda^2 \tilde{c})^{-1}(x' ) - \epsilon}{ \tilde{N}} \right)\\
         & = \mme\left[ \mmp \left( \frac{\text{prox}(\lambda^2 \tilde{c})^{-1}(-x' ) - \epsilon}{ \tilde{N}} \leq  \lambda Z \leq \frac{\text{prox}(\lambda^2 \tilde{c})^{-1}(x' ) - \epsilon}{ \tilde{N}} \mid \lambda, \epsilon \right) \right].
    \end{align*}
    Since $\mmp(\lambda > 0) = 1$, the bounded convergence theorem, the continuity of the inverse proximal function in $x$ (Lemma \ref{lem:cont_diff_prox}), and the continuity of the cumulative distribution function of $Z$ clearly imply that 
    \begin{align*}
    & \lim_{x' \to x}\mme\left[\mmp \left( \frac{\text{prox}(\lambda^2 c_{\infty})^{-1}(-x' ) - \epsilon}{ N_{\infty}} \leq \lambda Z \leq \frac{\text{prox}(\lambda^2 c_{\infty})^{-1}(x' ) - \epsilon}{ N_{\infty}}  \mid \lambda, \epsilon \right)  \right]\\
    & = \mme\left[ \lim_{x' \to x} \mmp \left( \frac{\text{prox}(\lambda^2 c_{\infty})^{-1}(-x' ) - \epsilon}{ N_{\infty}} \leq \lambda Z \leq \frac{\text{prox}(\lambda^2 c_{\infty})^{-1}(x' ) - \epsilon}{ N_{\infty}}  \mid \lambda, \epsilon \right)  \right]\\
    & = \mme\left[\mmp \left( \frac{\text{prox}(\lambda^2 c_{\infty})^{-1}(-x ) - \epsilon}{ N_{\infty}} \leq \lambda Z \leq \frac{\text{prox}(\lambda^2 c_{\infty})^{-1}(x ) - \epsilon}{ N_{\infty}}  \mid \lambda, \epsilon \right)   \right],
    \end{align*}
    as desired.
\end{proof}

Applying the previous three lemmas to the empirical measure over the absolute residuals gives the desired convergence of the empirical quantile.

\begin{corollary}[Lemma \ref{lem:gen_quant_conv} in the main text]\label{corr:full_quantile_conv}
    Under Assumptions 1-3 from Section \ref{sec:general_losses},
    \[
    \textup{Quantile}\left(1-\alpha, \frac{1}{n+1} \sum_{i=1}^{n+1} \delta_{|R_i|} \right) \stackrel{\mmp}{\to} \textup{Quantile}\left(1-\alpha, |\textup{prox}(\lambda^2 c_{\infty}\ell)(\epsilon + \lambda N_{\infty} Z)|\right),
    \]
    where $(Z,\lambda,\epsilon) \sim N(0,1) \otimes P_{\lambda} \otimes P_{\epsilon}$.
\end{corollary}
\begin{proof}
    This follows immediately from Lemmas \ref{lem:conv_in_dist}, \ref{lem:quantile_conv_fact}, and \ref{lem:asymptotic_prox_dist_facts}. 
\end{proof}

We are now ready to prove Theorem \ref{thm:conv_tau_fixed}. In the work that follows we use $q^*$ to denote the asymptotic quantile $\text{Quantile}(1-\alpha,|\text{prox}(\lambda^2c_{\infty}\ell)(\epsilon + \lambda N_{\infty}Z)|)$.

\begin{proof}[Proof of Theorem \ref{thm:conv_tau_fixed}]
Fix any $\xi > 0$. We have that 
\begin{align*}
\mmp(Y_{n+1} \in \hat{C}_{\text{full}} \mid \{(X_i,Y_i)\}_{i=1}^n ) & = \mmp\left(|R_{n+1}| \leq \text{Quantile}\left(1-\alpha,\frac{1}{n+1}\sum_{i=1}^{n+1} \delta_{|R_i|}\right) \mid \{(X_i,Y_i)\}_{i=1}^n \right)\\
& \leq \mmp\left( |\text{prox}(\lambda_{n+1}^2c_{\infty}\ell)(R^{(n+1)}_{n+1})| \leq q^* + 2\xi \mid \{(X_i,Y_i)\}_{i=1}^n \right)\\
& \ \ \ + \mmp( \left|\text{prox}(\lambda_{n+1}^2c_{\infty}\ell)(R^{(n+1)}_{n+1}) - R_{n+1} \right| > \xi \mid \{(X_i,Y_i)\}_{i=1}^n )\\
& \ \ \ + \mmp\left( \left|  \text{Quantile}\left(1-\alpha,\frac{1}{n+1}\sum_{i=1}^{n+1} \delta_{|R_i|}\right) - q^* \right| \geq \xi \mid \{(X_i,Y_i)\}_{i=1}^n  \right).
\end{align*}
Now, by Lemma \ref{lem:prox_resid_convergence} and Corollary \ref{corr:full_quantile_conv}, the later two expressions above converge to zero in expectation. Since these random variables are non-negative, this is sufficient to imply that they are $o_{\mmp}(1)$ (for a simple proof of this fact see Lemma \ref{lem:marg_to_cond_conv}). So, we find that
\begin{align*}
\mmp(Y_{n+1} \in \hat{C}_{\text{full}} \mid \{(X_i,Y_i)\}_{i=1}^n ) & \leq \mmp\left( | \text{prox}(\lambda_{n+1}^2c_{\infty}\ell)(R^{(n+1)}_{n+1})| \leq q^* + 2\xi \mid \{(X_i,Y_i)\}_{i=1}^n \right) + o_{\mmp}(1).
\end{align*}
To conclude the proof, let
\[
h_{\xi}(x) := \begin{cases}
        1, \text{ if } x \leq q^* + 3\xi,\\
        (q^* + 4\xi - x)/\xi, \text{ if } q^* + 3\xi < x \leq q^* + 4\xi\\
        0, \text{ if } x > q^* + 4\xi,
    \end{cases} 
\]
Then, by part 5 of Theorem \ref{thm:EK_facts} and the definition of $h_{\xi}(\cdot)$ we have that 
\begin{align*}
& \mmp\left( | \text{prox}(\lambda_{n+1}^2c_{\infty}\ell)(R^{(n+1)}_{n+1})| \leq q^* + 2\xi \mid \{(X_i,Y_i)\}_{i=1}^n \right) \\
& \leq \mme\left[ h_{\xi}\left( | \text{prox}(\lambda_{n+1}^2c_{\infty}\ell)(R^{(n+1)}_{n+1})| \right)  \mid \{(X_i,Y_i)\}_{i=1}^n \right]\\
& \stackrel{\mmp}{\to} \mme\left[ h_{\xi}(| \text{prox}(\lambda^2c_{\infty}\ell)(\epsilon + \lambda N_{\infty} Z)|) \right] \leq \mmp\left( | \text{prox}(\lambda^2c_{\infty}\ell)(\epsilon + \lambda N_{\infty} Z)| \leq q^* + 4\xi \right).
\end{align*}
Sending $\xi \to 0$ and applying the continuity of this limiting distribution (Lemma \ref{lem:asymptotic_prox_dist_facts}), we conclude that,
\[
\mmp(Y_{n+1} \in \hat{C}_{\text{full}} \mid \{(X_i,Y_i)\}_{i=1}^n) \leq 1-\alpha + o_{\mmp}(1).
\]
A matching lower bound follows from an identical argument and thus we conclude that $\mmp(Y_{n+1} \in \hat{C}_{\text{full}} \mid \{(X_i,Y_i)\}_{i=1}^n) \stackrel{\mmp}{\to} 1-\alpha $.

To control the coverage of the uncorrected method, we may simply repeat the arguments from above to conclude that 
\[
\mmp(Y_{n+1} \in \hat{C}_{\text{uncorr.}} \mid \{(X_i,Y_i)\}_{i=1}^n) \stackrel{\mmp}{\to} \mmp\left(|\epsilon + \lambda N_{\infty} Z | \leq  \text{Quantile}\left(1-\alpha,|\text{prox}(\lambda^2c_{\infty}\ell) (\epsilon + \lambda N_{\infty} Z )| \right) \right).
\]
Now, by Lemma \ref{lem:cont_diff_prox} we have that for any fixed $c \geq 0$,
\[
\frac{d}{dx}\text{prox}(c\ell)(x) = \frac{1}{1+c\ell''(\text{prox}(c\ell)(x))}.
\]
Since $\mmp(\lambda > 0) = 1$, $\mmp(\epsilon + \lambda N_{\infty} Z  = 0) = 0$, and $\text{prox}(c\ell)(0) = 0$, this immediately implies that 
\[
\mmp(|\text{prox }(\lambda^2c_{\infty}\ell) (\epsilon + \lambda N_{\infty} Z ) | < |\epsilon + \lambda N_{\infty} Z )|) = 1.
\]
 Moreover, since the cumulative distribution functions of $|\epsilon + \lambda N_{\infty} Z |$ and $ |\text{prox}(\lambda^2c_{\infty}\ell)(\epsilon + \lambda N_{\infty} Z )| $ are strictly increasing on $[0,\infty)$ (see the proof of Lemma \ref{lem:asymptotic_prox_dist_facts}) and $\text{prox}(\lambda^2c_{\infty}\ell)(\cdot)$ is continuous, we find that $\text{Quantile}\left(1-\alpha, |\text{prox}(\lambda^2c_{\infty}\ell)(\epsilon + \lambda N_{\infty} Z )| \right) < \text{Quantile}\left(1-\alpha, |\epsilon + \lambda N_{\infty} Z | \right)$ and thus,
\[
\mmp\left(|\epsilon + \lambda N_{\infty} Z | \leq  \text{Quantile}\left(1-\alpha, |\text{prox}(\lambda^2c_{\infty}\ell)(\epsilon + \lambda N_{\infty} Z )| \right) \right) < 1-\alpha.
\]
\end{proof}

\section{Proofs for Section \ref{sec:tau_random}}

In this section we prove Theorem \ref{thm:conv_with_cv}. The main idea is to demonstrate that the pointwise convergence results from the previous sections can be made uniform over $\tau$. 

While the dependence of many quantities on $\tau$ was left implicit in the previous section, here we will need to make this explicit. We will do this by adding the notation $(\tau)$ after any quantity that depends on the level of regularization. For instance, while in the previous section we used $\hat{\beta}$ and $R_i$ to denote the fitted regression coefficients and residuals, here we will write $\hat{\beta}(\tau)$ and $R_i(\tau)$. In some places in the proof it will be helpful to consider what happens in the limit as $\tau \to \infty$, i.e., in the case where the fitted regression coefficients are set to zero. To allow our notation to handle this situation seamlessly, we will use $\tau = \infty$ to denote this. This will give us the natural definitions $\hat{\beta}(\infty) = 0$, $R_i(\infty) = Y_i, R_i^{(i)}(\infty) = Y_i$, $c_i(\infty) = 0, c_{\infty}(\infty) = 0$, and $N_{\infty}(\infty) = \mme[(\sqrt{d}\beta^*_i)^2]^{1/2}$. It is straightforward to verify that all of the results stated in the previous section go through with these choices. 

As a final preliminary remark, recall that Theorem \ref{thm:conv_with_cv} considers two different sets of assumptions in which the loss function is taken to be either the squared loss or to satisfy Assumption 1 from Section \ref{sec:general_losses}. To handle these two sets of assumptions seamlessly in what follows we let $\ell$ denote a generic loss satisfy either set of assumptions. At times we will need to use specific properties of $\ell$ that hold under only one set of assumptions and we will mention these situations explicitly as they arise.

\subsection{Uniform approximations for the residuals}

In this section we will show that
\[
\sup_{\tau \geq \tau_0} |R_{n+1}(\tau_{\text{rand.}}) - \text{prox}(\lambda_{n+1}^2c_{\infty}(\tau_{\text{rand.}})\ell)(R^{(n+1)}_{n+1}(\tau_{\text{rand.}}))| \stackrel{\mmp}{\to} 0.
\]
We begin by giving a few useful preliminary lemmas. Our first result controls the rate of change of the estimated regression coefficients in $\tau$.

\begin{lemma}\label{lem:continuity_of_beta_in_tau}
    For all $\tau_1, \tau_2 > 0$,
    \[
    \| \hat{\beta}(\tau_1) - \hat{\beta}(\tau_2)\|_2 \leq \frac{|\tau_2 - \tau_1|}{\max\{\tau_2,\tau_1\}}  ( \| \hat{\beta}(\tau_1)\|^2_2 + \|\hat{\beta}(\tau_2)\|^2_2)
    \]
\end{lemma}
\begin{proof}
    Without loss of generality we may assume that $\tau_2 > \tau_1$. Then, by the strong convexity of the regression objective and the optimality of $\hat{\beta}(\tau_1)$ and $\hat{\beta}(\tau_2)$,
    \begin{align*}
    & \tau_2 \| \hat{\beta}(\tau_1) - \hat{\beta}(\tau_2) \|_2^2\\
    & \leq \frac{1}{n+1} \sum_{i=1}^{n+1} \ell(Y_i - X_i^\top \hat{\beta}(\tau_1)) + \tau_2\|\hat{\beta}(\tau_1)\|_2^2 - \frac{1}{n+1} \sum_{i=1}^{n+1} \ell(Y_i - X_i^\top \hat{\beta}(\tau_2)) - \tau_2\|\hat{\beta}(\tau_2)\|_2^2\\
    & = \left(\frac{1}{n+1} \sum_{i=1}^{n+1} \ell(Y_i - X_i^\top \hat{\beta}(\tau_1)) + \tau_1\|\hat{\beta}(\tau_1)\|_2^2 - \frac{1}{n+1} \sum_{i=1}^{n+1} \ell(Y_i - X_i^\top \hat{\beta}(\tau_2)) - \tau_1\|\hat{\beta}(\tau_2)\|_2^2 \right)\\
    & \hspace{10cm} + (\tau_2 - \tau_1)(\|\hat{\beta}(\tau_1)\|_2^2  - \|\hat{\beta}(\tau_2)\|_2^2)\\
    & \leq 0 + (\tau_2 - \tau_1)(\|\hat{\beta}(\tau_1)\|_2^2  - \|\hat{\beta}(\tau_2)\|_2^2)\\
    & \leq |\tau_2 - \tau_1| (\|\hat{\beta}(\tau_1)\|_2^2  + \|\hat{\beta}(\tau_2)\|_2^2),
    \end{align*}
    where to get the first inequality we have utilized the fact that $\hat{\beta}(\tau_2)$ minimizes $\beta \mapsto \frac{1}{n+1} \sum_{i=1}^{n+1} \ell(Y_i - X_i^\top \beta) + \tau_2 \|\beta\|^2$ and to get the second inequality we have used the fact that $\hat{\beta}(\tau_1)$ minimizes $\beta \mapsto \frac{1}{n+1} \sum_{i=1}^{n+1} \ell(Y_i - X_i^\top \beta) + \tau_1 \|\beta\|^2.$
\end{proof}

Our next lemma shows that the regression coefficients decrease in norm as the penalty, $\tau$, increases.

\begin{lemma}\label{lem:monotone_norm}
    The function $\tau \mapsto \|\hat{\beta}(\tau)\|_2$ is non-increasing.
\end{lemma}
\begin{proof}
    For ease of notation let 
    \[
    L(\beta) := \frac{1}{n+1} \sum_{i=1}^{n+1} \ell(Y_i - X_i^\top \beta).
    \]
    Fix any $\tau_2, \tau_1 \geq 0$. By definition of $\hat{\beta}(\cdot)$ we have the inequalities 
    \begin{align*}
        & L(\hat{\beta}(\tau_2)) + \tau_2 \|\hat{\beta}(\tau_2)\|_2^2  \leq L(\hat{\beta}(\tau_1)) + \tau_2 \|\hat{\beta}(\tau_1)\|_2^2, 
        \text{ and }  L(\hat{\beta}(\tau_1)) + \tau_1 \|\hat{\beta}(\tau_1)\|_2^2 \leq L(\hat{\beta}(\tau_2)) + \tau_1 \|\hat{\beta}(\tau_2)\|_2^2. \\
    \end{align*}
    Chaining these inequalities together gives us that,
    \[
    \tau_2 (\|\hat{\beta}(\tau_1)\|_2^2 - \|\hat{\beta}(\tau_2)\|_2^2) \geq L(\hat{\beta}(\tau_2)) - L(\hat{\beta}(\tau_1)) \geq \tau_1 (\|\hat{\beta}(\tau_1)\|_2^2 - \|\hat{\beta}(\tau_2)\|_2^2),
    \]
    and in particular,
    \[
    (\tau_2 - \tau_1)(\|\hat{\beta}(\tau_1)\|_2^2 - \|\hat{\beta}(\tau_2)\|_2^2) \geq 0.
    \]
    This proves the desired result.
\end{proof}

To state our next lemma, let
\[
S_{n+1}(\tau) := \frac{1}{n+1} \sum_{j \neq n+1} \ell''(R^{(n+1)}_j(\tau))X_jX_j^\top.
\]
Our next result shows that the function $\tau \mapsto (S_{n+1}(\tau) + 2\tau I_d)^{-1}$ is Lipschitz in the operator norm.

\begin{lemma}\label{lem:matrix_lip_in_op}
    Fix any $\tau_0 > 0$. Then, under the assumptions of either Theorem \ref{thm:ridge_asymp_cov} or Theorem \ref{thm:conv_tau_fixed}, there exists $L>0$ such that with probability converging to one, $\tau \mapsto (S_{n+1}(\tau) + 2\tau I_d)^{-1}$ is $\sqrt{n}L$-Lipschitz in the operator norm on $[\tau_0,\infty)$.
\end{lemma}
\begin{proof}
Fix any $\tau_1,\tau_2 \geq \tau_0$. Using the representation $A^{-1}-B^{-1} = A^{-1}(B-A)B^{-1}$ we compute that
    \begin{align*}
    & \|(S_{n+1}(\tau_1) + 2\tau_1 I_d)^{-1} -(S_{n+1}(\tau_2) + 2\tau_2 I_d)^{-1}\|_{op}\\
    & \leq \|(S_{n+1}(\tau_1) + 2\tau_1 I_d)^{-1}\|_{op} \|(S_{n+1}(\tau_2) + 2\tau_2 I_d)^{-1}\|_{op} \|(S_{n+1}(\tau_1) + 2\tau_1 I_d) -(S_{n+1}(\tau_2) + 2\tau_2 I_d) \|_{op}\\
    & \leq \frac{1}{4\tau_0^2} \left(2|\tau_1 - \tau_2| +  \|S_{n+1}(\tau_1) -S_{n+1}(\tau_2)\|_{op} \right) \\
    & \leq  \frac{|\tau_1 - \tau_2|}{2\tau_0^2} + \frac{1}{4\tau_0^2} \sup_{\|v\| = 1} \frac{1}{n+1} \sum_{j \neq n+1} |\ell''(R_j^{n+1}(\tau_1)) - \ell''(R_j^{n+1}(\tau_2))| (X_j^\top v)^2\\
    & \leq \frac{|\tau_1 - \tau_2|}{2\tau_0^2} + \frac{1}{4\tau_0^2} \sup_{\|v\| = 1} \frac{1}{n+1} \sum_{j \neq n+1} \| \ell'''\|_{\infty} | X_j^\top(\hat{\beta}_{(n+1)}(\tau_1) - \hat{\beta}_{(n+1)}(\tau_2))| (X_j^\top v)^2 \\
    & \leq \frac{|\tau_1 - \tau_2|}{2\tau_0^2} + \frac{ 
 \|\ell'''\|_{\infty} }{4\tau_0^2} \sup_{v : \|v\|_2 = 1}  \sqrt{\frac{1}{n+1} \sum_{j=1}^{n+1} (X_j^\top v )^4} \sqrt{\frac{1}{n+1} \sum_{j = 1}^{n+1} (X_j^\top(\hat{\beta}_{(n+1)}(\tau_1) - \hat{\beta}_{(n+1)}(\tau_2)) )^2}  \\
    & \leq \frac{|\tau_1 - \tau_2|}{2\tau_0^2}\\
    & + \frac{\|\ell'''\|_{\infty}}{4\tau_0^2}  \sup_{v : \|v\|_2 = 1}     \sqrt{(n+1)\left(\frac{1}{n+1} \sum_{j=1}^{n+1} (X_j^\top v )^2 \right)^2}  \left\| \frac{1}{n+1} \sum_{j=1}^{n+1} X_jX_j^\top \right\|_{op}  \|\hat{\beta}_{(n+1)}(\tau_1) - \hat{\beta}_{(n+1)}(\tau_2)\|_2\\
    & \leq \frac{|\tau_1 - \tau_2|}{2\tau_0^2} + \frac{\|\ell'''\|_{\infty}}{4\tau_0^2}\sqrt{n+1} \left\| \frac{1}{n+1} \sum_{j=1}^{n+1} X_jX_j^\top \right\|_{op}^3  \frac{2\|\hat{\beta}_{(n+1)}(\tau_0)\|^2_2}{\tau_0}|\tau_1 - \tau_2|\\
    & \leq \left(\frac{1}{2\tau_0^2} + \sqrt{n+1}\frac{\|\ell'''\|_{\infty}}{4\tau_0^2} \|P_{\lambda}\|_{\infty}^3 \left\| \frac{1}{n+1} \sum_{j=1}^{n+1} W_jW_j^\top \right\|_{op}^3  \frac{4\|\hat{\beta}_{(n+1)}(\tau_0) - \beta^*\|^2_2 + 4\|\beta^*\|^2_2}{\tau_0}\right)|\tau_1 - \tau_2|,
    \end{align*}
    where to get the second last inequality we have applied Lemmas \ref{lem:continuity_of_beta_in_tau} and \ref{lem:monotone_norm} to control $\|\hat{\beta}_{(n+1)}(\tau_1) - \hat{\beta}_{(n+1)}(\tau_2)\|_2$. Combining the preceding inequalities with a standard control on the operator norm of the empirical covariance (see Theorem 3.1 of \cite{Yin1988}) and the fact that $\|\hat{\beta}_{n+1}(\tau_0) -\beta^*\|^2_2$ and $\|\beta^*\|^2_2$ are both converging in probability to constants (this follows by Lemma \ref{lem:ridge_norm_conv} or part 4 of Theorem \ref{thm:EK_facts} and the law of large numbers) yields the desired result.
\end{proof}

With the above lemmas in hand, we are now ready to prove our main approximation for $R_{n+1}(\tau_{\text{rand.}})$. We will prove this result in two steps, first demonstrating that $$\sup_{\tau \geq \tau_0} |R_{n+1}(\tau) - \text{prox}( c_{n+1}(\tau) \ell)(R^{(n+1)}_{n+1}(\tau))| \stackrel{\mmp}{\to} 0,$$ and then that $$| \text{prox}( c_{n+1}(\tau_{\text{rand.}}) \ell)(R^{(n+1)}_{n+1}(\tau_{\text{rand.}})) - \text{prox}(\lambda_{n+1}^2 c_{\infty}(\tau_{\text{rand.}}) \ell)(R^{(n+1)}_{n+1}(\tau_{\text{rand.}}))| \stackrel{\mmp}{\to} 0.$$ Along the way we will prove a few useful facts about $c_{\infty}(\tau)$ that will be relevant to later results.

\begin{lemma}\label{lem:uniform_loo_resid_swap}
    Under the assumptions of Theorem \ref{thm:ridge_asymp_cov} or Theorem \ref{thm:conv_tau_fixed}, we have that for any $\tau_0 > 0$,
    \[
    \sup_{\tau \geq \tau_0} \left|R_{n+1}(\tau) - \textup{prox}( c_{n+1}(\tau) \ell)(R^{(n+1)}_{n+1}(\tau)) \right| \stackrel{\mmp}{\to} 0.
    \]
\end{lemma}
\begin{proof}
    First, suppose the assumptions of Theorem \ref{thm:ridge_asymp_cov} hold. Then, by the calculations of Lemma \ref{lem:ridge_loo_lemma},
    \begin{align*}
    \sup_{\tau \geq \tau_0} \left|R_{n+1}(\tau) -\text{prox}( c_{n+1}(\tau) \ell)(R^{(n+1)}_{n+1}(\tau)) \right| & = \sup_{\tau \geq \tau_0} \left| X_{n+1}^\top\left( \frac{1}{n+1}X^\top X + \tau I_d\right) \frac{\tau}{n+1} \hat{\beta}_{(n+1)}(\tau) \right|\\
    & \leq \frac{\|X_{n+1}\|_2}{n+1} \sup_{\tau \geq \tau_0} \|\hat{\beta}_{(n+1)}(\tau)\|_2.
    \end{align*}
    By the law of large numbers we have that $\|X_{n+1}\|_2 = O_{\mmp}(\sqrt{n})$, while by Lemmas \ref{lem:monotone_norm} and \ref{lem:ridge_norm_conv}, 
    \[
    \sup_{\tau \geq \tau_0} \|\hat{\beta}_{(n+1)}(\tau)\|_2 \leq \|\hat{\beta}_{(n+1)}(\tau_0)\|_2 \leq \|\hat{\beta}_{(n+1)}(\tau_0) - \beta^*\|_2 + \|\beta^*\|_2 = O_{\mmp}(1),
    \]
    which proves the desired result. 

    Now, suppose the assumptions of Theorem \ref{thm:conv_tau_fixed} hold. Then, by Proposition 3.4 and Lemmas 3.5 and 3.6 of \citet{EK2018} we have the deterministic bound,
    \begin{align*}
    & \left|R_{n+1}(\tau) - \text{prox}( c_{n+1}(\tau) \ell)(R^{(n+1)}_{n+1}(\tau)) \right|\\
    & \leq \frac{\|X_{n+1}\|^2_2}{\tau^2(n+1)} \left\|\frac{1}{n} \sum_{i=1}^n X_iX_i^\top\right\|^2_{op} \sup_{i < n+1} \frac{|X_i^\top(S_{n+1}(\tau) + 2\tau I_d)^{-1}X_{n+1}|}{n+1}.
    \end{align*}    
    Now, by well-known results on the operator norm of the empirical covariance, $\|\frac{1}{n+1}\sum_{i=1}^{n+1}X_iX_i^\top \|_{op} \leq \|P_{\lambda}\|^2_{\infty} \|\frac{1}{n+1}\sum_{i=1}^{n+1}W_iW_i^\top \|_{op} \stackrel{\mmp}{\to} \|P_{\lambda}\|^2_{\infty}(1+\sqrt{\gamma})^2$ (Theorem 3.1 of \cite{Yin1988}).  Applying this fact to the previous expression along with the law of large numbers gives us that
    \[
    \sup_{\tau \geq \tau_0} \left|R_{n+1}(\tau) - \textup{prox}( c_{n+1}(\tau) \ell)(R^{(n+1)}_{n+1}(\tau)) \right|  = O_{\mmp}\left( 1\right) \sup_{\tau \geq \tau_0}\sup_{i < n+1} \frac{|X_i^\top(S_{n+1}(\tau) + \tau I_d)^{-1}X_{n+1}|}{n+1}.
    \]
It remains to control $\sup_{\tau \geq \tau_0}\sup_{i < n+1} \frac{|X_i^\top(S_{n+1}(\tau) + 2\tau I_d)^{-1}X_{n+1}|}{n+1}$. We will begin by considering the case where $\tau$ is large. To do this, note that since the entries of $W_{ij}$ are sub-Gaussian, $W_{ij}^2$ is sub-exponential and thus by standard tail bounds (e.g.~Equation 2.18 in \cite{Wainwright2019}), $\sup_{i}\|X_i\|_2/\sqrt{n} \leq \|P_{\lambda}\|_{\infty}\sup_{i}\|W_i\|_2/\sqrt{n} \leq O_{\mmp}(\sqrt{\log(n)})$ and $\|X_{n+1}\|_2/\sqrt{n} = O_{\mmp}(1)$. Thus,
\[
\sup_{\tau \geq n}\sup_{i < n+1} \frac{|X_i^\top(S_{n+1}(\tau) + 2\tau I_d)^{-1}X_{n+1}|}{n+1} \leq \frac{1}{2n} \sup_{i < n+1} \frac{\|X_i\|_2 \|X_{n+1}\|_2}{n+1} = O_{\mmp}\left(\frac{\sqrt{\log(n)}}{n}\right).
\]
On the other hand, to obtain a control for $\tau \leq n$, let $\tau_j := \tau_0 + j/n$ for $1 \leq j \leq n^2$. By Lemma \ref{lem:matrix_lip_in_op}, there exists $L>0$ such that with probability converging to one,
\begin{align*}
& \sup_{\tau \geq n}\sup_{i < n+1} \frac{|X_i^\top(S_{n+1}(\tau) + 2\tau I_d)^{-1}X_{n+1}|}{n+1}\\
& \leq \frac{L}{\sqrt{n}}\sup_{i < n+1} \frac{\|X_i\|_2 \|X_{n+1}\|_2}{n+1} +\sup_{1 \leq j \leq n^2}\sup_{i < n+1}  \frac{|X_i^\top(S_{n+1}(\tau_j) + 2\tau_j I_d)^{-1}X_{n+1}|}{n+1}\\
& \leq O_{\mmp}\left(\sqrt{\frac{\log(n)}{n}}\right) + \sup_{1 \leq j \leq n^2}\sup_{i < n+1}  \frac{|X_i^\top(S_{n+1}(\tau_j) + 2\tau_j I_d)^{-1}X_{n+1}|}{n+1}.
\end{align*}
Finally, applying the fact that $W_{n+1}$ is mean zero, sub-Gaussian, and independent of $\{X_i^\top(S_{n+1}(\tau_j) + 2\tau_j I_d)^{-1}\}_{i \in [n], j \in [n^2]}$ we find that
\begin{align*}
\sup_{1 \leq j \leq n^2}\sup_{i < n+1}  \frac{|X_i^\top(S_{n+1}(\tau_j) + 2\tau_j I_d)^{-1}X_{n+1}|}{n+1} & \leq O_{\mmp}\left(  \sqrt{\log(n)} \sup_{1 \leq j \leq n^2}\sup_{i < n+1} \frac{\|X_i^\top(S_{n+1}(\tau_j) + 2\tau_j I_d)^{-1}\|_2 }{n} \right)\\
& = O_{\mmp}\left( \frac{\log(n)}{\sqrt{n}}\right).
\end{align*}
This proves the desired result.
\end{proof}

We now compare $\text{prox}(c_{n+1}(\tau) \ell)(R^{{(n+1})}_{n+1}(\tau))$ to $\text{prox}(\lambda_{n+1}^2 c_{\infty}(\tau) \ell)(R^{{(n+1})}_{n+1}(\tau))$.  

\begin{lemma}\label{lem:uniform_c_n_swap}
   Suppose that the assumptions of either Theorem \ref{thm:ridge_asymp_cov} or Theorem \ref{thm:conv_tau_fixed} hold. Fix any $\tau_0>0$ and let $\tau_{\textup{rand.}} \independent (X_{n+1},Y_{n+1})$ be independent of the test point and such that $\mmp(\tau_{\textup{rand.}} \geq \tau_0) = 1$. Then, 
    \[
    |\textup{prox}(c_{n+1}(\tau_{\textup{rand.}}) \ell)(R^{{(n+1})}_{n+1}(\tau_{\textup{rand.}})) - \textup{prox}(\lambda_{n+1}^2 c_{\infty}(\tau_{\textup{rand.}}) \ell)(R^{{(n+1})}_{n+1}(\tau_{\textup{rand.}}))| \stackrel{\mmp}{\to} 0. 
    \]
    Moreover, $c_{\infty}(\tau)$ admits the bound $c_{\infty}(\tau) \leq \gamma/\tau$ and there exists a constant $L>0$ such that $\tau \mapsto c_{\infty}(\tau)$ is $L$-Lipschitz on $[\tau_0,\infty)$.
\end{lemma}
\begin{proof}
    We claim it is sufficient to show that $\sup_{\tau \geq \tau_0} |c_{n+1}(\tau) - \lambda_{n+1}^2 c_{\infty}(\tau)|$ = $o_{\mmp}(1)$. To see this, first suppose that the assumptions of Theorem \ref{thm:conv_tau_fixed} hold. Then, by Lemma \ref{lem:cont_diff_prox} we know that for any fixed $x \in \mmr$,
    \[
     \frac{d}{dc} \textup{prox}(c\ell)(x) = - \frac{\ell'(\textup{prox}(c\ell)(x))}{1 + c\ell''(\textup{prox}(c\ell)(x))}.
    \]
    Combining this with our assumption that the first derivative of $\ell$ is bounded, we find that there exists a constant $C >0$ such that for all fixed $x \in \mmr$, $c \mapsto \text{prox}(c \ell)(x)$ is $C$-Lipschitz. In particular,
    \[
    \sup_{\tau \geq \tau_0} |\text{prox}(c_{n+1}(\tau) \ell)(R^{{(n+1})}_{n+1}(\tau)) - \text{prox}(\lambda_{n+1}^2 c_{\infty}(\tau) \ell)(R^{{(n+1})}_{n+1}(\tau))| \leq C \sup_{\tau \geq \tau_0} |c_{n+1}(\tau) - \lambda_{n+1}^2 c_{\infty}(\tau)|.
    \]
    On the other hand, suppose the assumptions of Theorem \ref{thm:ridge_asymp_cov} hold (namely the loss is the squared loss). Then, 
    \begin{align*}
     &  |\textup{prox}(c_{n+1}(\tau_{\textup{rand.}}) \ell)(R^{{(n+1})}_{n+1}(\tau_{\textup{rand.}})) - \textup{prox}(\lambda_{n+1}^2 c_{\infty}(\tau_{\textup{rand.}}) \ell)(R^{{(n+1})}_{n+1}(\tau_{\textup{rand.}}))| \\
     & =   \left|\frac{R^{{(n+1})}_{n+1}(\tau_{\textup{rand.}})}{1+2c_{n+1}(\tau_{\textup{rand.}})} -\frac{R^{{(n+1})}_{n+1}(\tau_{\textup{rand.}})}{1+2\lambda_{n+1}^2c_{\infty}(\tau_{\textup{rand.}})}  \right| \leq 2  |R^{{(n+1})}_{n+1}(\tau_{\textup{rand.}})| \sup_{\tau \geq \tau_0} | c_{n+1}(\tau) - \lambda_{n+1}^2 c_{\infty}(\tau)|.
    \end{align*}
    By Lemma \ref{lem:monotone_norm} we have that 
    \[
    \mme[|R^{{(n+1})}_{n+1}(\tau_{\textup{rand.}})|^2] \leq 2\mme[Y_{n+1}^2] + 2\mme[\lambda_{n+1}^2(W_{n+1}^\top\hat{\beta}_{n+1}(\tau_{\textup{rand.}}))^2] \leq 2\mme[Y_{n+1}^2] + 2\mme[\lambda_{n+1}^2] \mme[\|\hat{\beta}_{n+1}(\mathrm{\tau_0})\|^2_2].
    \]
    Moreover, by Lemma \ref{lem:large_tau_beta_norm_bound} below $\mme[\|\hat{\beta}_{n+1}(\mathrm{\tau_0})\|^2_2]$ is $O(1)$. Thus, $|R^{{(n+1})}_{n+1}(\tau_{\textup{rand.}})| = O_{\mmp}(1)$ and to get the desired result it is sufficient to bound $\sup_{\tau \geq \tau_0} | c_{n+1}(\tau) - \lambda_{n+1}^2 c_{\infty}(\tau)|$.

    Now suppose either set of assumptions holds. We will show that $\tau \mapsto c_{n+1}(\tau)$ is smooth enough to imply that the pointwise convergence of $c_{n+1}(\tau)$ to $\lambda_{n+1}^2 c_{\infty}(\tau)$ also guarantees uniform convergence. Unfortunately, $c_{n+1}(\tau)$ is somewhat delicate and thus it will be useful to go through an intermediate quantity that is easier to control. In particular, we will proceed in two steps, first bounding 
    \[
    \sup_{\tau \geq \tau_0} \left|c_{n+1}(\tau) -  \frac{ \lambda_{n+1}^2}{n+1}\text{tr}\left( \left(\frac{1}{n+1} \sum_{i=1}^{n} \ell''(R_i(\tau)) X_iX_i^\top + 2\tau I_d \right)^{-1} \right)\right|,
    \]
    and then controlling 
    \[
    \sup_{\tau \geq \tau_0} \left| \frac{ \lambda_{n+1}^2}{n+1}\text{tr}\left( \left(\frac{1}{n+1} \sum_{i=1}^{n} \ell''(R_i(\tau)) X_iX_i^\top + 2\tau I_d \right)^{-1} \right) -  \lambda_{n+1}^2c_{\infty}(\tau) \right|.
    \]
    
    To begin, first note that by the Hanson-Wright inequality (\cite{Wright1973}) we have that for any fixed $\tau \geq \tau_0$ and $t>0$,
    \begin{equation}\label{eq:pointwise_c_qf}
    \mmp\left(\left|\frac{c_{n+1}(\tau)}{\lambda_{n+1}^2} - \frac{1}{n+1} \text{tr}\left( \left(\frac{1}{n+1} \sum_{i=1}^{n} \ell''(R_i(\tau)) X_iX_i^\top + 2\tau I_d \right)^{-1} \right) \right| \geq t \right) \leq \exp\left(- \frac{C}{\tau_0^2} \min\{nt, nt^2\}\right),
    \end{equation}
    for some fixed constant $C>0$ not dependent on $\tau$ or $t$.
    
    We claim that both $c_{n+1}(\tau)$ and the trace above are smooth functions of $\tau$ and thus this pointwise convergence can be made uniform over $\tau$. To see this, fix any $M>0$ large. We clearly have that  
    \[
    \sup_{\tau > M} |c_{n+1}(\tau)| \leq \frac{\lambda_{n+1}^2\|W_{n+1}\|^2_2}{M(n+1)} = \frac{\lambda_{n+1}^2}{M} + o_{\mmp}(1),
    \]
    while
    \[
     \sup_{\tau > M} \left| \frac{\lambda_{n+1}^2}{n+1} \text{tr}\left( \left(\frac{1}{n+1} \sum_{i=1}^{n} \ell''(R_i(\tau)) X_iX_i^\top + 2\tau I_d \right)^{-1} \right) \right| \leq \frac{\lambda_{n+1}^2}{M}.
    \]
    Thus, it is sufficient to focus on $\tau_0 \leq \tau \leq M$. For $i \in \{1,\dots,(n+1) M\}$ define $\tau_i = \tau_0 + i/(n+1)$. Observe that for any $\tau \in [\tau_i, \tau_{i+1})$ and $j \neq n+1$,
    \begin{align*}
    |c_{n+1}(\tau) - c_{n+1}(\tau_i)| & =  \left|\frac{1}{n+1} X_{n+1}^\top(S_{n+1}(\tau) + 2\tau I_d)^{-1}X_{n+1} - \frac{1}{n+1} X_{n+1}^\top(S_{n+1}(\tau_i) + 2\tau_i I_d)^{-1}X_{n+1} \right|\\
    & \leq \frac{\|X_{n+1}\|^2_2}{n+1} \|(S_{n+1}(\tau) + 2\tau I_d)^{-1} -(S_{n+1}(\tau_i) + 2\tau_i I_d)^{-1}\|_{op}.
    \end{align*}
    So, applying the fact that  $\tau \mapsto (S_{n+1}(\tau) + 2\tau I_d)^{-1}$ is asymptotically $O(\sqrt{n})$-Lipschitz in the operator norm (Lemma \ref{lem:matrix_lip_in_op}), we find that
    \begin{align*}
         &  \sup_{\tau_0 \leq \tau \leq M} \inf_{i \in [(n+1)M]} |c_{n+1}(\tau) - c_{n+1}(\tau_i)| \leq O_{\mmp}\left(\frac{1}{\sqrt{n}} \right).
    \end{align*}
    A completely identical argument shows that the same result also holds for the matrix trace and thus in total we conclude that,
     \begin{align*}
     & \sup_{\tau \geq \tau_0} \left| c_{n+1}(\tau) - \frac{\lambda_{n+1}^2}{n+1} \text{trace}\left( \left(\frac{1}{n+1} \sum_{i=1}^{n} \ell''(R_i(\tau)) X_iX_i^\top + 2\tau I_d \right)^{-1} \right) \right|\\
     & \leq \sup_{\tau_i = \tau_0 + \frac{i}{n+1}, i \in [(n+1)M]} \left|c_{n+1}(\tau_i) - \frac{\lambda_{n+1}^2}{n+1} \text{trace}\left( \left(\frac{1}{n+1} \sum_{i=1}^{n} \ell''(R_i(\tau_i)) X_iX_i^\top + 2\tau_i I_d \right)^{-1} \right) \right|\\
     & \hspace{11cm} + O_{\mmp}\left(\frac{1}{\sqrt{n}} \right) + \frac{\|P_{\lambda}\|_{\infty}}{M}.
     \end{align*}
    Finally, combining the pointwise convergence result (\ref{eq:pointwise_c_qf}) with a union bound, one easily verifies that the first term above is of size $o_{\mmp}(1)$. Thus, sending $M \to \infty$ we conclude that 
    \[
     \sup_{\tau \geq \tau_0} \left| c_{n+1}(\tau) - \frac{\lambda_{n+1}^2}{n+1} \text{trace}\left( \left(\frac{1}{n+1} \sum_{i=1}^{n} \ell''(R_i(\tau)) X_iX_i^\top + 2\tau I_d \right)^{-1} \right) \right| \stackrel{\mmp}{\to} 0.
    \]
    
    It remains to compare the matrix trace to $c_{\infty}(\tau)$. The argument will be nearly identical to that given above, except that here we will make use of a tighter smoothness result. In particular, observe that for any $\tau_1, \tau_2 \geq \tau_0$,
    \begin{align*}
        & \left| \frac{1}{n+1} \text{trace}\left( \left(S_{n+1}(\tau_1) + 2\tau_1 I_d \right)^{-1} \right) - \frac{1}{n} \text{trace}\left( \left(S_{n+1}(\tau_2) + 2\tau_2 I_d \right)^{-1} \right) \right|\\
        & = \left| \frac{1}{n+1} \text{trace}\left( \left(S_{n+1}(\tau_1) + 2\tau_1 I_d \right)^{-1} \left(S_{n+1}(\tau_2) - S_{n+1}(\tau_1) + 2(\tau_2 - \tau_1)I_d \right) \left(S_{n+1}(\tau_2) + 2\tau_2 I_d \right)^{-1} \right)   \right|\\
        & \leq \frac{1}{n+1} \sum_{i=1}^{n} \left| (\ell''(R_i(\tau_2)) - \ell''(R_i(\tau_1))) \text{trace}\left(\left(S(\tau_1) + 2\tau_1 I_d \right)^{-1} \frac{X_iX_i^\top}{n+1}  \left(S(\tau_2) + 2\tau_2 I_d \right)^{-1}\right)\right|\\
        & \qquad \qquad + \|\left(S(\tau_1) + 2\tau_1 I_d \right)^{-1}\|_{op} \|\left(S(\tau_2) + 2\tau_2 I_d \right)^{-1}\|_{op} 2|\tau_1 - \tau_2|\\
        & \leq \frac{1}{4\tau_0^2} \left(\frac{1}{n+1} \sum_{i=1}^{n} |\ell''(R_i(\tau_2)) - \ell''(R_i(\tau_1))| \frac{\|X_i\|^2}{n+1} + 2|\tau_1 - \tau_2| \right)\\
        & \leq \frac{1}{4\tau_0^2} \left( \sqrt{\frac{1}{n+1} \sum_{i=1}^{n} \|\ell'''\|_{\infty}^2 |X_i^\top (\hat{\beta}(\tau_2) - \hat{\beta}(\tau_1))|^2 } \sqrt{\frac{1}{n+1} \sum_{i=1}^{n} \frac{\|X_i\|_2^4}{(n+1)^2}} + 2|\tau_1 - \tau_2| \right)\\
        & \leq \frac{|\tau_1 - \tau_2|}{2\tau_0^2}  \left(\|\ell'''\|_{\infty}  \left\| \frac{1}{n+1} \sum_{i=1}^{n} X_iX_i^\top \right\|_{op} \frac{2\|\hat{\beta}(\tau_0)\|^2}{\tau_0}   \sqrt{\frac{1}{n+1} \sum_{i=1}^{n} \frac{\|X_i\|_2^4}{(n+1)^2}} + 1 \right)\\
        & \leq \frac{|\tau_1 - \tau_2|}{2\tau_0^2}  \left(\|\ell'''\|_{\infty}  \left\| \frac{1}{n+1} \sum_{i=1}^{n} X_iX_i^\top \right\|_{op} \frac{4(\|\hat{\beta}(\tau_0) - \beta^*\|^2 + \|\beta^*\|_2^2)}{\tau_0}   \sqrt{\frac{1}{n+1} \sum_{i=1}^{n} \frac{\|X_i\|_2^4}{(n+1)^2}} + 1 \right)\\
       & \leq \frac{|\tau_1 - \tau_2|}{2\tau_0^2}  \left(\|\ell'''\|_{\infty} \left\| \frac{\|P_{\lambda}\|_{\infty}^2}{n+1} \sum_{i=1}^{n} W_iW_i^\top \right\|_{op} \frac{4(\|\hat{\beta}(\tau_0) - \beta^*\|^2 + \|\beta^*\|_2^2)}{\tau_0}   \sqrt{\frac{1}{n+1} \sum_{i=1}^{n} \frac{\|X_i\|_2^4}{(n+1)^2}} + 1 \right),
    \end{align*}
    where the second last inequality above uses Lemmas \ref{lem:continuity_of_beta_in_tau} and \ref{lem:monotone_norm} to control $\|\hat{\beta}(\tau_1) - \hat{\beta}(\tau_2)\|_2$. Now, all the stochastic elements above converge in probability to constants. Thus, we find that there exists a constant $L'>0$ such that with probability tending to $1$, $\tau \mapsto \frac{\lambda_{n+1}^2}{n+1} \text{trace}\left( \left(\frac{1}{n+1} \sum_{i=1}^{n} \ell''(R_i(\tau)) X_iX_i^\top + 2\tau I_d \right)^{-1} \right)$ is $L'$-Lipschitz. By standard arguments on the convergence of Lipschitz functions (see e.g.~Lemma \ref{lem:func_uniform_conv} or equivalently consider a grid over $\tau$ and argue as above), this is sufficient to conclude that for fixed any $M > 0$,
    \[
    \sup_{\tau_0 \leq \tau \leq M} \left| \frac{1}{n+1} \text{trace}\left( \left(\frac{1}{n+1} \sum_{i=1}^{n} \ell''(R_i(\tau)) X_iX_i^\top + 2\tau I_d \right)^{-1} \right) - c_{\infty}(\tau) \right| \stackrel{\mmp}{\to} 0.
    \]
    Moreover, 
    \[
     \frac{1}{n+1} \text{trace}\left( \left(\frac{1}{n+1} \sum_{i=1}^{n} \ell''(R_i(\tau)) X_iX_i^\top + 2\tau I_d \right)^{-1} \right)  \leq \frac{d}{\tau(n+1)} \to \frac{\gamma}{\tau},
    \]
    and using the fact that $c_{\infty}(\tau)$ is the pointwise limit of the trace, also $c_{\infty}(\tau) \leq \gamma/\tau$. Thus, for $\tau \geq M$,
    \[
    \sup_{\tau \geq M} \left| \frac{1}{n+1} \text{trace}\left( \left(\frac{1}{n+1} \sum_{i=1}^{n} \ell''(R_i(\tau)) X_iX_i^\top + 2\tau I_d \right)^{-1} \right) - c_{\infty}(\tau) \right|  \leq 2\frac{\gamma}{M}.
    \]
    The desired result then follows immediately by sending $M \to \infty$.

    Finally, to get the last part of the lemma, note that since $c_{\infty}(\tau)$ is the pointwise limit of the trace, the fact that the trace function is asymptotically Lipschitz immediately implies the Lipschitz continuity of $c_{\infty}(\tau)$.

\end{proof}

\subsection{Uniform convergence of the empirical quantille}

In this section we show that the empirical quantile, $\text{Quantile}(1-\alpha, \frac{1}{n+1} \sum_{i=1}^{n+1} \delta_{|R_i(\tau)|})$ converges uniformly to $q^*(\tau) := \text{Quantile}(1-\alpha,|\text{prox}(\lambda^2c_{\infty}(\tau)\ell)(\epsilon + \lambda N_{\infty}(\tau)Z)|)$ in $\tau$. As a preliminary step, we first bound the norm of the fitted regression coefficients. 

\begin{lemma}\label{lem:large_tau_beta_norm_bound}
    We have the deterministic bound
    \[
    \|\hat{\beta}(\tau)\|_2 \leq \sqrt{ \frac{1}{\tau(n+1)} \sum_{i=1}^{n+1} \ell(Y_i) }.
    \]
\end{lemma}
\begin{proof}
    Comparing the regression objective at $\beta = \hat{\beta}(\tau)$ and $\beta = 0$ and applying the optimality of $\hat{\beta}(\tau)$ gives the inequalities
    \[
    {\tau} \|\hat{\beta}(\tau)\|_2^2 \leq \frac{1}{n+1} \sum_{i=1}^{n+1} \ell(Y_i - X_i^\top \hat{\beta}(\tau)) + {\tau} \|\hat{\beta}(\tau)\|_2^2 \leq \frac{1}{n+1} \sum_{i=1}^{n+1}  \ell(Y_i),
    \]
    as claimed.
\end{proof}

\begin{lemma}\label{lem:uniform_quant_conv}
    Suppose that the assumptions of either Theorem \ref{thm:ridge_asymp_cov} or Theorem \ref{thm:conv_tau_fixed} hold. Then, for any $\tau_0 > 0$,
    \[
    \sup_{\tau \geq \tau_0} \left| \textup{Quantile}\left(1-\alpha, \frac{1}{n+1} \sum_{i=1}^{n+1} \delta_{|R_{n+1}(\tau)|} \right) - q^*(\tau) \right| \stackrel{\mmp}{\to} 0.
    \]
    Moreover, for any bounded, Lipschitz function $\psi(\cdot)$, the map
    \[
    \tau \mapsto \mme[\psi(\textup{prox}(\lambda^2 c_{\infty}(\tau)\ell)(\epsilon + \lambda N_{\infty}(\tau)Z)) ],
    \]
    is continous and satisfies
    \[
    \lim_{\tau \to \infty} \mme[\psi(\textup{prox}(\lambda^2 c_{\infty}(\tau)\ell)(\epsilon + \lambda N_{\infty}(\tau)Z)) ] = \mme[\psi(\epsilon + \lambda \mme[(\sqrt{d}\beta^*_i)^2]^{1/2}Z) ].
    \]
    Similarly, $q^*(\tau)$ is continuous and satisfies $\lim_{\tau \to \infty} q^*(\tau) = q^*(\infty)$.

\end{lemma}
\begin{proof}
    We begin by showing that $q^*(\tau)$ is continuous. Fix any $\tau_1, \tau_2 \geq \tau_0$ and $\xi > 0$. Define the functions 
    \[
    h^{+}_{\xi}(x; \tau_i) := \begin{cases}
        1, \text{ if } x \leq q^*(\tau_i) + \xi,\\
        (q^*(\tau_i) + 2\xi - x)/\xi, \text{ if } q^*(\tau_i) + \xi < x \leq q^*(\tau_i) + 2\xi,\\
        0, \text{ if } x > q^*(\tau_i) + 2\xi,
    \end{cases}
    \]
    and 
    \[
    h^{-}_{\xi}(x; \tau_i) := \begin{cases}
        1, \text{ if } x \leq q^*(\tau_i) - 2\xi,\\
        (q^*(\tau_i) - \xi - x)/\xi, \text{ if } q^*(\tau_i) - 2\xi < x < q^*(\tau_i) - \xi,\\
        0, \text{ if } x > q^*(\tau_i) - \xi.
        \end{cases}
    \]
    In what follows we use the notation $h^{\pm}_{\xi}$ to denote results that hold for both $h^+_{\xi}$ and $h^{-}_{\xi}$. 
    
    Now, we know from Lemmas \ref{lem:ridge_quant_conv} and \ref{lem:conv_in_dist} that for $i,j \in \{1,2\}$,
    \[
    \frac{1}{n+1} \sum_{k=1}^{n+1}h^{\pm}_{\xi}(R_k(\tau_j); \tau_i) \stackrel{\mmp}{\to} \mme[h^{\pm}_{\xi}(|\text{prox}(\lambda^2c_{\infty}(\tau_j)\ell )(\epsilon + \lambda N_{\infty}(\tau_j))| ; \tau_i)].
    \]
    Moreover, 
    \begin{align*}
     & \left|\frac{1}{n+1} \sum_{k=1}^{n+1}h^{\pm}_{\xi}(R_k(\tau_1); \tau_i) -  \frac{1}{n+1} \sum_{k=1}^{n+1}h^{\pm}_{\xi}(R_k(\tau_2); \tau_i) \right|  \leq \frac{1}{\xi} \frac{1}{n+1} \sum_{i=1}^{n+1} |R_i(\tau_1) - R_i(\tau_2)|\\
     & \ \ \ \ \leq \frac{1}{\xi} \frac{1}{n+1} \sum_{i=1}^{n+1} |X_i^\top(\hat{\beta}(\tau_2) - \hat{\beta}(\tau_1))| \leq \frac{1}{\xi} \sqrt{\frac{1}{n+1} \sum_{i=1}^{n+1} |X_i^\top(\hat{\beta}(\tau_2) - \hat{\beta}(\tau_1))|^2}\\
    & \ \ \ \ \leq \frac{1}{\xi} \left\|\frac{1}{n+1} \sum_{i=1}^{n+1} X_iX_i^\top\right\|_{op} \|\hat{\beta}(\tau_2) - \hat{\beta}(\tau_1)\|_2 \leq \frac{1}{\xi} \left\|\frac{1}{n+1} \sum_{i=1}^{n+1} X_iX_i^\top\right\|_{op} \frac{2\|\hat{\beta}(\tau_0)\|_2^2}{\tau_0} |\tau_2 - \tau_1|,
    \end{align*}
    where the last inequality follows from Lemmas \ref{lem:continuity_of_beta_in_tau} and \ref{lem:monotone_norm}. Now, by well-known controls on the empirical covariance matrix, $\left\|\frac{1}{n+1} \sum_{i=1}^{n+1} X_iX_i^\top\right\|_{op} \leq \|P_{\lambda}\|^2_{\infty} \left\|\frac{1}{n+1} \sum_{i=1}^{n+1} W_iW_i^\top\right\|_{op}  \stackrel{\mmp}{\to} \|P_{\lambda}\|^2_{\infty} (1 + \sqrt{\gamma})^2$ (see Theorem 3.1 of \cite{Yin1988}), while part 4 of Theorem \ref{thm:EK_facts} and our assumptions on $\beta^*$ imply that $\|\hat{\beta}(\tau_0)\|_2^2 \leq 2(\|\hat{\beta}(\tau_0) - \beta^*\|_2^2 + \|\beta^*\|_2^2) \stackrel{\mmp}{\to} 2N_{\infty}(\tau_0)^2 + 2\mme[(\sqrt{d}\beta^*_i)^2])$. So, in particular, we find that there exists $L>0$ such that with probability converging to 1,
    \[
    \left|\frac{1}{n+1} \sum_{k=1}^{n+1}h^{\pm}_{\xi}(R_k(\tau_1); \tau_i) -  \frac{1}{n+1} \sum_{k=1}^{n+1}h^{\pm}_{\xi}(R_k(\tau_2); \tau_i) \right|  \leq L |\tau_1 - \tau_2|.
    \]
    Finally, since by Lemma \ref{lem:asymptotic_prox_dist_facts}, the cumulative distribution function of $|\text{prox}(\lambda^2c_{\infty}(\tau_2)\ell )(\epsilon + \lambda N_{\infty}(\tau_2))|$ is strictly increasing, we find that for all $\tau_1$ sufficiently close to $\tau_2$,
    \[
    \lim_{n \to \infty} \frac{1}{n+1} \sum_{k=1}^{n+1}h^{+}_{\xi}(R_k(\tau_1); \tau_2) \stackrel{\mmp}{\geq}  \lim_{n \to \infty}    \frac{1}{n+1} \sum_{k=1}^{n+1}h^{+}_{\xi}(R_k(\tau_2); \tau_2) - L |\tau_1 - \tau_2| \stackrel{\mmp}{>} 1-\alpha,
    \]
    and 
    \[
    \lim_{n \to \infty} \frac{1}{n+1} \sum_{k=1}^{n+1}h^{-}_{\xi}(R_k(\tau_1); \tau_2) \stackrel{\mmp}{\leq}   \lim_{n \to \infty}  \frac{1}{n+1} \sum_{k=1}^{n+1}h^{-}_{\xi}(R_k(\tau_2); \tau_2) + L |\tau_1 - \tau_2| \stackrel{\mmp}{<} 1-\alpha.
    \]
    By definition of $h^{\pm}_{\xi}$ and the fact that the empirical quantiles converge pointwise (Lemmas \ref{lem:ridge_quant_conv} and \ref{lem:gen_quant_conv}), this implies that for all $\tau_1$ sufficiently close to $\tau_2$,  $|q^*(\tau_1) - q^*(\tau_2)| \leq 4\xi$. Thus, $q^*(\cdot)$ is continuous.

    With the continuity of $q^*(\tau)$ in hand, we will now prove the uniformity of the quantile convergence over any bounded interval $\tau_0 \leq \tau \leq M$. Namely, fix any $M>0$ large and $\xi > 0$ small. Fix any $\rho > 0$ small and let $\tau_i = \tau_0 + i\rho$ for all $0 \leq i \leq \lceil M/\rho \rceil$. By Lemmas \ref{lem:ridge_quant_conv} and \ref{lem:gen_quant_conv},
    \[
    \sup_{i} \left|\textup{Quantile}\left(1-\alpha, \frac{1}{n+1} \sum_{i=1}^{n+1} \delta_{|R_{n+1}(\tau_i)|} \right) -q^*(\tau_i)\right| \stackrel{\mmp}{\to} 0.
    \]
    Moreover, by Lemma \ref{lem:asymptotic_prox_dist_facts} and the continuity of $h^{\pm}_{\xi}$ we know that 
    \[
    \inf_{\tau \in [\tau_0,M]} \left|\mme[h^{\pm}_{\xi}(|\text{prox}(\lambda^2c_{\infty}(\tau)\ell )(\epsilon + \lambda N_{\infty}(\tau))| ; \tau)] - (1-\alpha) \right| > 0,
    \]
     and by repeating our previous argument, there exists $L>0$ such that with probability converging to one, it holds that for all $\tilde{\tau} \geq \tau_0$, $\tau \mapsto \frac{1}{n+1} \sum_{i=1}^{n+1} h^{\pm}_{\xi}(R_i(\tau); \tilde{\tau})$ is $L$-Lipschitz. Using these facts, it is straightforward to show for all $\rho$ sufficiently small, we have that with probability converging to one,
    \[
    \sup_{\tau_0 \leq \tau \leq M} \inf_{0 \leq i \leq \lceil M/\rho \rceil} \left|\textup{Quantile}\left(1-\alpha, \frac{1}{n+1} \sum_{i=1}^{n+1} \delta_{|R_{n+1}(\tau)|} \right) - \textup{Quantile}\left(1-\alpha, \frac{1}{n+1} \sum_{i=1}^{n+1} \delta_{|R_{n+1}(\tau_i)|} \right) \right| \leq 4\xi.
    \]
    Applying the continuity of $q^*(\tau)$ we find that once again by taking $\rho$ sufficiently small,
    \begin{align*}
    & \sup_{\tau_0 \leq \tau \leq M} \left|\textup{Quantile}\left(1-\alpha, \frac{1}{n+1} \sum_{i=1}^{n+1} \delta_{|R_{n+1}(\tau)|} \right) - q^*(\tau) \right|\\
    & \ \ \ \ \stackrel{\mmp}{\leq} 5\xi + \sup_{i}\left|\textup{Quantile}\left(1-\alpha, \frac{1}{n+1} \sum_{i=1}^{n+1} \delta_{|R_{n+1}(\tau_i)|} \right) - q^*(\tau_i) \right|  = 5\xi + o_{\mmp(1)}.
    \end{align*}
    Since $\xi>0$ was arbitrary, we conclude that 
    \[
    \sup_{\tau_0 \leq \tau \leq M} \left|\textup{Quantile}\left(1-\alpha, \frac{1}{n+1} \sum_{i=1}^{n+1} \delta_{|R_{n+1}(\tau)|} \right) - q^*(\tau) \right| \stackrel{\mmp}{\to} 0.
    \]

    To handle the case where $\tau \geq M$, we simply remark that proceeding exactly as above
    \begin{align*}
    & \left|\frac{1}{n+1} \sum_{i=1}^{n+1} h^{\pm}_{\xi}(R_i(\tau);\infty) - \frac{1}{n+1} \sum_{i=1}^{n+1} h^{\pm}_{\xi}(R_i(\infty);\infty) \right|  \leq \frac{1}{\xi} \left\|\frac{1}{n+1} \sum_{i=1}^{n+1} X_iX_i^\top\right\|_{op} \|\hat{\beta}(\tau)\|_2\\ 
    &  \leq \frac{1}{\xi} \|P_{\lambda}\|^2_{\infty} \left\|\frac{1}{n+1} \sum_{i=1}^{n+1} W_iW_i^\top\right\|_{op} \sqrt{ \frac{1}{\tau(n+1)} \sum_{i=1}^{n+1} \ell(Y_i)} \stackrel{\mmp}{\to} \frac{1}{\xi}\|P_{\lambda}\|^2_{\infty}  (1+\sqrt{\gamma})^2 \sqrt{\frac{1}{\tau}\mme[\ell(Y_i)]},
    \end{align*}
    where to get the second inequality we have applied Lemma \ref{lem:large_tau_beta_norm_bound}. Since $\frac{1}{n+1} \sum_{i=1}^{n+1} h^{\pm}_{\xi}(R_i(\infty);\infty) \stackrel{\mmp}{\to} \mme[ h^{\pm}_{\xi}(|\epsilon + \lambda \mme[(\sqrt{d}\beta^*_i)^2]^{1/2}Z|;\infty)]$ and by the fact that $|\epsilon + \lambda \mme[(\sqrt{d}\beta^*_i)^2]^{1/2}Z|$ has a unique $1-\alpha$ quantile, we find that for $M$ sufficiently large,
    \[
    \mmp\left(\sup_{\tau \geq M} \left|\textup{Quantile}\left(1-\alpha, \frac{1}{n+1} \sum_{i=1}^{n+1} \delta_{|R_{n+1}(\tau)|} \right) - q^*(\infty) \right| \leq 4\xi \right) \stackrel{\mmp}{\to} 1. 
    \]
    Again, arguing as above, it is straightforward to show that $\lim_{\tau \to \infty} q^*(\tau) = q^*(\infty)$ and thus choosing $M$ sufficiently large we find that
    \[
    \mmp\left(\sup_{\tau \geq M} \left|\textup{Quantile}\left(1-\alpha, \frac{1}{n+1} \sum_{i=1}^{n+1} \delta_{|R_{n+1}(\tau)|} \right) - q^*(\tau) \right| > 5\xi\right) \to 0.
    \]
    This proves the desired result. 
    
\end{proof}

\subsection{Final preliminaries}

Before proving Theorem \ref{thm:conv_with_cv}, we will give one final preliminary result which provides a uniform bound on the amount that $\text{prox}(\lambda^2c_{\infty}(\tau)\ell)(\epsilon + \lambda N_{\infty}(\tau) Z)$ can concentrate in any small interval.

\begin{lemma}\label{lemma:uniformly_bounded_cdf}
    Suppose that the assumptions of either Theorem \ref{thm:ridge_asymp_cov} or Theorem \ref{thm:conv_tau_fixed} hold. Then, for all $\tau_0 >0$ there exists a constant $C> 0 $ such that for all $\tau \in [\tau_0,\infty]$ and $a < b$,
    \[
    \mmp(\textup{prox}(\lambda^2c_{\infty}(\tau)\ell)(\epsilon + \lambda N_{\infty}(\tau) Z) \in [a,b]) \leq C|a-b|.
    \]
\end{lemma}
\begin{proof}
    Recall that by assumption $\epsilon$ has a bounded density, $f_{\epsilon}$. Using standard formula for convolutions of random variables, we find that for any fixed values of $\lambda$ and $\tau$, the density of $\epsilon + \lambda N_{\infty}(\tau) Z$ exists and is bounded by $\|f_{\epsilon}\|_{\infty}$. 
    Marginalizing over $\lambda$, this immediately proves the lemma for $\tau = \infty$ with $C = \|f_{\epsilon}\|_{\infty}$ (recall that in this case $c_{\infty}(\tau) = 0$ and thus $\textup{prox}(\lambda^2c_{\infty}(\tau)\ell)(\epsilon + \lambda N_{\infty}(\tau) Z) = \epsilon + \lambda N_{\infty}(\tau) Z$). For $\tau < \infty$, note that by Lemma \ref{lem:cont_diff_prox} we have that for any fixed $c>0$,
    \[
    \frac{d}{dx} \text{prox}(c\ell)(x) = \frac{1}{1+c\ell''(\text{prox}(c\ell)(x))}.
    \]
    So, a simple change of variables calculation shows that for $(\lambda$, $\tau)$ fixed the density of $\text{prox}(\lambda^2c_{\infty}(\tau)\ell)(\epsilon + \lambda N_{\infty}(\tau) Z)$ exists and is bounded by $\|f_{\epsilon}\|_{\infty}(1+\lambda^2c_{\infty}(\tau)\|\ell''\|_{\infty})$. In particular, this implies that
    \begin{align*}
    \mmp(\text{prox}(\lambda^2c_{\infty}(\tau)\ell)(\epsilon + \lambda N_{\infty}(\tau) Z) \in [a,b]) & = \mme[\mmp(\text{prox}(\lambda^2c_{\infty}(\tau)\ell)(\epsilon + \lambda N_{\infty}(\tau) Z) \in [a,b] \mid \lambda)]\\
    & \leq \mme[|a-b| \cdot \|f_{\epsilon}\|_{\infty}(1+\mme[\lambda^2]c_{\infty}(\tau)\|\ell''\|_{\infty})]\\
    & = |a-b| \cdot \|f_{\epsilon}\|_{\infty}(1+\mme[\lambda^2]c_{\infty}(\tau)\|\ell''\|_{\infty}).
    \end{align*}
    Finally, recalling that $c_{\infty}(\tau) \leq \gamma/\tau$ (Lemma \ref{lem:uniform_c_n_swap}), we conclude that for $\tau \geq \tau_0$,
    \[
     \mmp(\text{prox}(\lambda^2c_{\infty}(\tau)\ell)(\epsilon + \lambda N_{\infty}(\tau) Z) \in [a,b]) \leq |a-b|\cdot \|f_{\epsilon}\|_{\infty}\left(1+\mme[\lambda^2]\frac{\gamma }{\tau_0}\|\ell''\|_{\infty}\right).
    \]
    Taking $C = \max\{\|f_{\epsilon}\|_{\infty},\|f_{\epsilon}\|_{\infty}(1+\mme[\lambda^2]\frac{\gamma }{\tau_0}\|\ell''\|_{\infty})\}$ gives the desired result.
\end{proof}

\subsection{Final proof of Theorem \ref{thm:conv_with_cv}}

With the above preliminaries out of the way, we are now ready to prove Theorem \ref{thm:conv_with_cv}.

\begin{proof}[Proof of Theorem \ref{thm:conv_with_cv}]
    The main structure of the proof is similar to the proof for $\tau$ fixed (Theorem \ref{thm:conv_tau_fixed}). Namely, using the uniform convergence results of Lemmas \ref{lem:uniform_loo_resid_swap}, \ref{lem:uniform_c_n_swap}, and \ref{lem:uniform_quant_conv} along with Lemma \ref{lem:marg_to_cond_conv}, which shows that results which hold marginally with high probability also hold conditionally with high probability, we have that for any fixed $\xi >0 $,
\begin{align*}
& \mmp(Y_{n+1} \in \hat{C}_{\text{full}} \mid \{(X_i,Y_i)\}_{i=1}^n )\\
& = \mmp\left(|R_{n+1}(\tau_{\textup{rand.}})| \leq \text{Quantile}\left(1-\alpha,\frac{1}{n+1}\sum_{i=1}^{n+1} \delta_{|R_i(\tau_{\textup{rand.}})|}\right) \mid \{(X_i,Y_i)\}_{i=1}^n \right)\\
& \leq \mmp\left( |\text{prox}(\lambda_{n+1}^2 c_{\infty}(\tau_{\textup{rand.}})\ell)(R^{(n+1)}_{n+1}(\tau_{\textup{rand.}}))| \leq q^*(\tau_{\textup{rand.}}) + \xi \mid \{(X_i,Y_i)\}_{i=1}^n \right) + o_{\mmp}(1).
\end{align*}
To conclude the proof, we need to compare the distribution of $|\text{prox}(\lambda_{n+1}^2 c_{\infty}(\tau_{\textup{rand.}})\ell)(R^{(n+1)}_{n+1}(\tau_{\textup{rand.}}))|$ to that of its asymptotic counterpart, $|\text{prox}(\lambda^2c_{\infty}(\tau_{\textup{rand.}})\ell)(\epsilon + \lambda N_{\infty}(\tau_{\textup{rand.}}) Z)|$. Since our previous lemmas only allow us to make this comparison pointwise, we will begin by rounding $\tau_{\textup{rand.}}$ to a finite grid. 

Fix any $M>0$ large and $\rho > 0$ small. Let $\{\tau_1,\dots,\tau_k\}$ be a $\rho$-cover of $[\tau_0,M]$ of size at most $(M/\rho) + 1$. Let 
\[
g(\tau) := \begin{cases} \text{argmin}_{\tau_\in \{\tau_1,\dots,\tau_k\}} |\tau - \tau_i|, \text{ if } \tau \leq M,\\
\infty, \text{ otherwise},
\end{cases}
\]
denote the point in the cover closest to $\tau$ with the exception that large values are ``rounded" to infinity. 

Now, suppose the assumptions of Theorem \ref{thm:conv_tau_fixed} hold. Then, using  the fact that the proximal function is Lipschitz in both its arguments (Lemma \ref{lem:cont_diff_prox}), that $c_{\infty}(\tau)$ is Lipschitz and satisfies $|c_{\infty}(\tau) - c_{\infty}(\infty)|  \leq \gamma/\tau$ (Lemma \ref{lem:uniform_c_n_swap}), and that $q^*(\tau)$ is continuous with $\lim_{\tau \to \infty} q^*(\tau) = q^*(\infty)$ (Lemma \ref{lem:uniform_quant_conv}) we may choose $M$ sufficiently large and $\rho$ sufficiently small such that 
\begin{align*}
& \mmp\left( |\text{prox}(\lambda_{n+1}^2 c_{\infty}(\tau_{\textup{rand.}})\ell)(R^{(n+1)}_{n+1}(\tau_{\textup{rand.}}))| \leq q^*(\tau_{\textup{rand.}}) + \xi \mid \{(X_i,Y_i)\}_{i=1}^n \right)\\
& \leq \mmp\left( |\text{prox}(\lambda_{n+1}^2 c_{\infty}(g(\tau_{\textup{rand.}}))\ell)(R^{(n+1)}_{n+1}(g(\tau_{\textup{rand.}})))| \leq q^*(g(\tau_{\textup{rand.}})) + 3\xi \mid \{(X_i,Y_i)\}_{i=1}^n \right)\\
& \ \ \ \ \ + \mmp\left(|X_{n+1}^\top(\hat{\beta}_{(n+1)}(g(\tau_{\textup{rand.}})) - \hat{\beta}_{(n+1)}(\tau_{\textup{rand.}}))| \geq \xi \mid \{(X_i,Y_i)\}_{i=1}^n \right).
\end{align*}
Under the assumptions of Theorem \ref{thm:ridge_asymp_cov} (namely, when loss is the squared loss) this same result holds with the extra caveat that the proximal function is no longer Lipschitz in its first argument and thus we require an additional bound on $R_{n+1}^{(n+1)}(\tau_{\textup{rand.}})$ (see the arguments at the beginning of Lemma \ref{lem:uniform_c_n_swap}).

Now, since $g(\tau_{\textup{rand.}})$ takes on only finitely many values, we may apply the results of Lemma \ref{lem:ridge_CLT} and part 5 of Theorem \ref{thm:EK_facts} regarding the convergence in distribution of $(R^{(n+1)}_{n+1}(g(\tau)),\lambda)$ for $\tau$ fixed along with standard calculations relating convergence of the cumulative distribution function to convergence of the expectation of continuous functions (see e.g., the arguments at the end of the proof of Theorem \ref{thm:conv_tau_fixed}), to conclude that
\begin{align*}
&\mmp\left( |\text{prox}(\lambda_{n+1}^2 c_{\infty}(g(\tau_{\textup{rand.}}))\ell)(R^{(n+1)}_{n+1}(g(\tau_{\textup{rand.}})))| \leq q^*(g(\tau_{\textup{rand.}})) + 3\xi \mid \{(X_i,Y_i)\}_{i=1}^n \right)\\
& \leq \mmp\left( |\text{prox}(\lambda^2 c_{\infty}(g(\tau_{\textup{rand.}}))\ell)(\epsilon + \lambda N_{\infty}(g(\tau_{\textup{rand.}})) Z) | \leq q^*(g(\tau_{\textup{rand.}})) + 4\xi \mid \tau_{\textup{rand.}} \right) + o_{\mmp}(1).
\end{align*}
Finally, note that by our uniform bounds on the cumulative distribution function of this asymptotic distribution (Lemma \ref{lemma:uniformly_bounded_cdf}), there exists a constant $C>0$ such that,
\begin{align*}
& \mmp\left( |\text{prox}(\lambda^2 c_{\infty}(g(\tau_{\textup{rand.}}))\ell)(\epsilon + \lambda N_{\infty}(g(\tau_{\textup{rand.}})) Z) | \leq q^*(g(\tau_{\textup{rand.}})) + 4\xi  \mid \tau_{\textup{rand.}}   \right)\\
& \leq \mmp\left( |\text{prox}(\lambda^2 c_{\infty}(g(\tau_{\textup{rand.}}))\ell)(\epsilon + \lambda N_{\infty}(g(\tau_{\textup{rand.}})) Z) | \leq q^*(g(\tau_{\textup{rand.}})) \mid \tau_{\textup{rand.}} \right)  + 4C\xi = 1-\alpha + 4C\xi.
\end{align*}
It remains to bound the error term $ \mmp\left(|X_{n+1}^\top(\hat{\beta}_{(n+1)}(g(\tau_{\textup{rand.}})) - \hat{\beta}_{(n+1)}(\tau_{\textup{rand.}}))| \geq \xi \mid \{(X_i,Y_i)\}_{i=1}^n \right)$. By the results of Lemmas \ref{lem:continuity_of_beta_in_tau} and \ref{lem:monotone_norm} we have that for $\tau \leq M$,
\[
\|\hat{\beta}_{(n+1)}(g(\tau)) - \hat{\beta}_{(n+1)}(\tau)\|_2 \leq \frac{2\rho}{\tau_0} \|\hat{\beta}_{(n+1)}(\tau_0)\|_2^2 \leq  \frac{4\rho}{\tau_0} (\|\hat{\beta}_{(n+1)}(\tau_0) - \beta^*\|_2^2 + \|\beta^*\|_2^2) = \rho O_{\mmp}(1).
\]
On the other hand, for $\tau > M$, Lemma \ref{lem:large_tau_beta_norm_bound} implies that
\[
\|\hat{\beta}_{(n+1)}(g(\tau)) - \hat{\beta}_{(n+1)}(\tau)\|_2 \leq \sqrt{\frac{1}{\tau(n+1)} \sum_{i=1}^{n+1} \ell(Y_i)} \stackrel{\mmp}{\to} \sqrt{\frac{1}{\tau} \mme[\ell(Y_i)]} \leq  \sqrt{\frac{1}{M} \mme[\ell(Y_i)]} .
\]
So, by taking $\rho$ sufficiently small and $M$ sufficiently large and using the fact that $X_{n+1} \independent \hat{\beta}_{(n+1)}(g(\tau_{\textup{rand.}})) - \hat{\beta}_{(n+1)}(\tau_{\textup{rand.}})$, we may guarantee that
\[
\mmp\left(|X_{n+1}^\top(\hat{\beta}_{(n+1)}(g(\tau_{\textup{rand.}})) - \hat{\beta}_{(n+1)}(\tau_{\textup{rand.}}))| \geq \xi \mid \{(X_i,Y_i)\}_{i=1}^n \right) \leq \xi +  o_{\mmp}(1).
\]
Putting all of our previous work together, we conclude that 
\[
\mmp(Y_{n+1} \in \hat{C}_{\text{full}} \mid \{(X_i,Y_i)\}_{i=1}^n ) \leq 1-\alpha + (4C+1)\xi + o_{\mmp}(1),
\]
and sending $\xi \to 0$ at an appropriate rate we arrive at our final bound,
\[
\mmp(Y_{n+1} \in \hat{C}_{\text{full}} \mid \{(X_i,Y_i)\}_{i=1}^n ) \leq 1-\alpha + o_{\mmp}(1).
\]
A matching lower bound follows from an identical argument.

\end{proof}

\section{System of equations for full conformal quantile regression}\label{sec:app_qr_calcs}

In this section we give a short derivation of the system of equations for the constants $N_{\infty}$, $c_{\infty}$, $\beta_{0,\infty}$ appearing in our asymptotic characterization of quantile regression from Conjecture \ref{conj:qr}. As outlined in the main text, this derivation is heuristic and while we find it to be highly accurate empirically, rigorously proving these statements is left as an open problem for future work.

To begin, consider the behaviour of quantile regression when the intercept term is held fixed. Namely, consider the regression
\begin{equation}\label{eq:fixed_b_qr}
\hat{\beta}_{(n+1)}(b) \in \text{argmin}_{\beta \in \mmr^d} \frac{1}{n} \sum_{i=1}^n \ell_{\alpha}(Y_i - b - X_i^\top\beta).
\end{equation}
Letting $\tilde{\epsilon}_i = \epsilon_i - b$ we may interpret this as a regression of $\tilde{Y}_i = X_i^\top\beta^* + \tilde{\epsilon}_i$ on $X_i$. 

Now, due to the fact that $\ell_{\alpha}$ is non-differentiable and (\ref{eq:fixed_b_qr}) is unregularized we cannot directly apply the results of \citet{EK2018} to characterize the behaviour of $\hat{\beta}_{(n+1)}(b)$. Nevertheless, we conjecture that by taking a smooth approximate to $\ell_{\alpha}$ and considering vanishing regularization it is reasonable to expect the results of \citet{EK2018} to extend to this setting (for further discussion of this conjecture see also Section 2.3.5 of \citet{EK2018} as well as Theorem 6.1 of \citet{EK2013}). Thus, extending Theorem 2.1 of \citet{EK2018} to this setting, we conjecture that for a fixed intercept, the asymptotic behaviour of quantile regression is characterized by constants $c_{\infty}(b)$, $N_{\infty}(b)$ satisfying the system of equations
\[
\begin{cases}
   &   \mme\left[ \frac{(W(b) - \text{prox}(\lambda^2c_{\infty}\ell_{\alpha}(W(b)))^2}{\lambda^2} \right] = \gamma N_{\infty}^2,\\
    & \mmp\left(W(b) \in [-\lambda^2c_{\infty}\alpha,\lambda^2c_{\infty}(1-\alpha)] \right) = \gamma,
\end{cases}
\]
where $W(b) := \epsilon - b + \lambda N_{\infty} Z$ and $(Z,\lambda,\epsilon) \sim N(0,1) \otimes P_{\lambda} \otimes P_{\epsilon}$.

It remains to derive an equation for the intercept. To do this, recall that in the main text we conjectured the approximation
\[
\hat{\eta}_i = \frac{R_i^{(i)} - \textup{prox}(\lambda_{n+1}^2c_{\infty}\ell_{\alpha})(R_i^{(i)})}{\lambda_{n+1}^2c_{\infty}}.
\]
Now, the first-order condition for the intercept in the min-max formulation of quantile regression (see (\ref{eq:qr_min_max})) necessitates that $\sum_{i=1}^{n+1} \hat{\eta}_i = 0$. By the exchangeability of $(\hat{\eta}_i)_{i=1}^{n+1}$ this immediately implies that $\mme[\hat{\eta}_i] = \frac{1}{n+1} \sum_{j=1}^{n+1}\mme[\hat{\eta}_j] = 0$. Finally, applying the approximation $R_{i}^{(i)} \stackrel{D}{\approx} W(\beta_{0,\infty})$ we arrive at the desired third equation
\[
\mme\left[ \frac{W(\beta_{0,\infty}) - \textup{prox}(\lambda_{n+1}^2c_{\infty}\ell_{\alpha})(W(\beta_{0,\infty}))}{\lambda_{n+1}^2c_{\infty}}\right] = 0.
\]

\subsection{Comparison to the results of \citet{Bai2021}}

A system of equations characterizing the asymptotic values of $\|\hat{\beta}_{(n+1}) - \beta^*\|_2$ and $\hat{\beta}_{0,(n+1)}$ was also given in \citet{Bai2021}. They consider a more restricted setting in which there is no elliptical parameter and $X_i \sim N(0,I_d)$ is Gaussian. To describe their results formally, let
\[
e(\ell)(x) := \min_v \ell(v) + \frac{1}{2}(x-v)^2, 
\]
denote the Moreau envelope of loss $\ell$. Then, \citet{Bai2021} derive the following asymptotic characterization.
\begin{theorem}[Informal Theorem C.1 of \citet{Bai2021}] Suppose $Y_i = X_i^\top\beta^* + \epsilon_i$ with $X_i \sim N(0,I_d)$. Assume $P_{\epsilon}$ satisfies (Assumption A of \citet{Bai2021}). Then, there exists constants $(N_{\infty}, c_{\infty}, \beta_{0,\infty})$ such that any empirical quantile regression minimizer $(\hat{\beta}_{0,(n+1)}, \hat{\beta}_{(n+1)})$ satisfies
\[
\hat{\beta}_{0,(n+1)} \stackrel{\mmp}{\to} \beta_{0,\infty} \ \ \ \text{ and } \ \ \ \|\hat{\beta}_{(n+1}) - \beta^*\|_2 \stackrel{\mmp}{\to} N_{\infty}.
\]
Moreover, letting $W = \epsilon - \beta_{0,\infty} + N_{\infty}Z$ for $(Z,\epsilon) \sim N(0,1) \otimes P_{\epsilon}$, $(N_{\infty}, c_{\infty}, \beta_{0,\infty})$ are solutions to the system of equations
\begin{equation}\label{eq:bai_system}
\begin{cases}
    & \mme[(W - \text{prox}(c_{\infty}\ell_{\alpha})(W))^2] = \gamma N_{\infty}^2, \\
    & \mme[(W - \text{prox}(c_{\infty}\ell_{\alpha})(W))Z] = \gamma N_{\infty}, \\
    &  \mme[W - \text{prox}(c_{\infty}\ell_{\alpha})(W)] = 0.
\end{cases}
\end{equation}
    
\end{theorem}

The first and third equations in (\ref{eq:bai_system}) exactly match those in our Conjecture \ref{conj:qr} in the case where $\mmp(\lambda = 1) = 1$. For the second equation, note that 
\[
W-\text{prox}(c_{\infty}\ell_{\alpha})(W) = W\bone\{W \in [-c_{\infty}\alpha,c_{\infty}(1-\alpha)]\} -c_{\infty}\alpha \bone\{W < -c_{\infty}\alpha\} + c_{\infty}(1-\alpha)\bone\{W > c_{\infty}\alpha\}.
\]
So, by Stein's Lemma
\[
\mme[(W-\text{prox}(c_{\infty}\ell_{\alpha})(W))Z] = N_{\infty}\mmp(W \in [-c_{\infty}\alpha,c_{\infty}(1-\alpha)]),
\]
which exactly matches the second equation appearing in Conjecture \ref{conj:qr}.

Overall, we find that the system of equations appearing in our Conjecture \ref{conj:qr} exactly matches Theorem C.1 of \citet{Bai2021} for the special case of Gaussian covariates. The proof of Theorem C.1 given in \citet{Bai2021} uses machinery derived from the convex Gaussian minimax theorem (\citet{Thram2018, ThramThesis}). Unfortunately, this technique provides no insight into the leave-one-out approximations needed to complete a formal proof of Conjecture \ref{conj:qr}. Thus, even in this special case, the proof of Conjecture \ref{conj:qr} remains an open problem.

\section{Technical Lemmas}

In this section we state and prove a number of technical lemmas that aid in the proofs of our main results. Our first two results demonstrate that pointwise convergence of Lipschitz functions on a compact domain implies uniform convergence. Both these results are well-known and we include proofs simply for completeness.

\begin{lemma}\label{lem:lip_to_unif_as}
    Let $A$ and $B$ be normed vector spaces with norms $\|\cdot\|_A$ and $\|\cdot\|_B$, respectively. Let $f_n : A \to B$ be a sequence of random $L$-Lipschitz functions for some $0 \leq L < \infty$. Assume that $A$ is compact and that for all $x \in A$, $\lim_{n \to \infty} f_n(x) \stackrel{a.s.}{=} f(x)$. Then, $\mmp(\forall x \in A,\ \lim_{n \to \infty} f_n(x) = f(x)) = 1$.
\end{lemma}
\begin{proof}
    Since $A$ is compact we may find a countable dense subset $\{x_n\}_{n \in \mmn} \subseteq A$ of $A$. Let $E$ denote the event that $\forall k \in \mmn$,  $\lim_{n \to \infty} f_n(x_k) = f(x_k)$. By a simple union bounded we have that $\mmp(E) = 1$. Fix any $x \in A$ and $\epsilon >0$. Let $k \in \mmn$ be such that $\|x_k - x\|_A \leq \epsilon/(2L)$ and observe that since $f$ is the pointwise limit of $f_n$ we clearly have that $f$ is $L$-Lipschitz. Thus, we find that on the event $E$,
    \begin{align*}
    \limsup_{n \to \infty} \|f_n(x) - f(x)\|_B & \leq \limsup_{n \to \infty} \|f_n(x_k) - f_n(x)\|_B +  \|f(x_k) - f(x)\|_B +  \|f_n(x_k) - f(x_k)\|_B\\
    & \leq 2L \|x_k - x\|_A + \limsup_{n \to \infty}  \|f_n(x_k) - f(x_k)\|_B \leq \epsilon,
    \end{align*}
    and sending $\epsilon \to 0$ we find that on the event $E$, $\lim_{n \to \infty} f_n(x) = f(x)$. Since the choice of $x$ was arbitrary, this proves the desired result.
\end{proof}

\begin{lemma}\label{lem:func_uniform_conv}
    Let $A$ and $B$ be normed vector spaces with norms $\|\cdot\|_A$ and $\|\cdot\|_B$, respectively. Let $f_n : A \to B$ be a sequence of random functions with the property that there exists $L>0$ such that $\mmp(f_n \text{ is L-Lispchitz on } A) \to 1$. Assume that $A$ is compact and that for all $x \in A$, $\lim_{n \to \infty} f_n(x) \stackrel{\mmp}{=} f(x)$. Then,
    \[
    \sup_{x \in A} \|f_n(x) - f(x)\|_B \stackrel{\mmp}{\to} 0.
    \]
\end{lemma}
\begin{proof}
    Note that the given conditions clearly imply that that $f$ is $L$-Lipschitz on $A$. Fix any $\epsilon > 0$. Let $\{x_k\}_{k=1}^m$ be a finite $\epsilon/(4L)$-covering of $A$. Then, on the event that $f_n$ is $L$-Lipschitz we have that
    \begin{align*}
    \sup_{x \in A} \|f_n(x) - f(x)\|_B & \leq \sup_{x \in A} \inf_{1 \leq i \leq m} (\|f_n(x) - f_n(x_i)\|_B +   \|f(x) - f(x_i)\|_B)  +  \sup_{1 \leq i \leq m}  \|f_n(x_i) - f(x_i)\|_B\\
    & \leq \frac{\epsilon}{2} + \sup_{1 \leq i \leq m}  \|f_n(x_i) - f(x_i)\|_B. 
    \end{align*}
    Thus,
    \begin{align*}
    \mmp\left(\sup_{x \in A} \|f_n(x) - f(x)\|_B > \epsilon\right) & \leq \mmp(f_n \text{ is not L-Lispchitz on } A)\\
    & \ \ \ \ \ + \mmp\left(\sup_{1 \leq i \leq m}  \|f_n(x_i) - f(x_i)\|_B > \epsilon/2\right) \to 0,
    \end{align*}
    as desired.
\end{proof}

Our next lemma is a variant of the continuous mapping theorem. This result is well-known and we include a proof simply for completeness.

\begin{lemma}\label{lem:cont_map}
    Let $f : \mmr^k \times \mmr^m \to \mmr$ be continuous. Let $\{A_n,B_n,C_n\} \subseteq \mmr^k \times \mmr^m \times \mmr^k$ be a sequence of random variables such that $\|A_n - C_n\|_2 \stackrel{\mmp}{\to} 0$ and $\|B_n\|_2, \|C_n\|_2 = O_{\mmp}(1)$. Then,
    \[
    f(A_n,B_n) - f(C_n,B_n) \stackrel{\mmp}{\to} 0.
    \]
\end{lemma}
\begin{proof}
    Fix any $\delta > 0$ small and $M > 0$ large. Since $S : = \{(A,B) : \|A\|_2, \|B\|_2 \leq M\}$ is compact and $f$ is continuous there exists $\rho > 0$ such that if $(A,B), (A',B) \in S$ and $\|A - A'\|_2 \leq \rho$, then $|f(A,B) - f(A',B)| \leq \delta$. Thus,
    \begin{align*}
    \limsup_{n \to \infty} \mmp(|f(A_n,B_n) - f(C_n,B_n)| > \delta) & \leq \limsup_{n \to \infty} \mmp( \|A_n - C_n\|_2 > \rho) + \mmp((A_n,B_n) \notin S)\\
    & = \limsup_{n \to \infty} \mmp((A_n,B_n) \notin S).
    \end{align*}
    Moreover, 
    \[
    \limsup_{n \to \infty} \mmp((A_n,B_n) \notin S) \leq  \limsup_{n \to \infty} \mmp(\|B_n\| > M) + \mmp(\|A_n - C_n\| > M/2)  + \mmp(\|C_n\| > M/2).
    \]
    By our assumptions, this last expression goes to 0 as $M \to \infty$. 
\end{proof}

Our next result shows that the proximal function is differentiable in each of its arguments and continuous in their combination.

\begin{lemma}\label{lem:cont_diff_prox}
    Let $\ell$ be any twice differentiable convex loss. Then, the function $(c,x) \mapsto \text{prox}(c\ell)(x)$ is continuous on $\mmr_{\geq 0} \times \mmr$. Moreover, for any fixed $x \in \mmr$,
    \[
    \frac{d}{dc} \textup{prox}(c\ell)(x) = - \frac{\ell'(\textup{prox}(c\ell)(x))}{1 + c\ell''(\textup{prox}(c\ell)(x))},
    \]
    and for any fixed $c \geq 0$,
    \[
    \frac{d}{dx}\textup{prox}(c\ell)(x) = \frac{1}{1+c\ell''(\textup{prox}(c\ell)(x))} \leq 1.
    \]
\end{lemma}
\begin{proof}
    The second and third parts of the lemma follow from Lemmas 3.32 and 3.33 of \citet{EK2018} (see also Lemmas A-2 and A-3 of \cite{EK2013} and \cite{Moreau1965ProximitED}). To get the first part of the lemma, let
    \[
    f(v;x,c) := c \ell(x) + \frac{1}{2}(x-v)^2,
    \]
    denote the objective of the proximal function's optimization. Observe that for any $(c,x), (c',x') \in \mmr_{\geq 0} \times \mmr$, the strong convexity of $f$ implies that
    \begin{align*}
        & \frac{1}{2} (\text{prox}(c\ell)(x) - \text{prox}(c'\ell)(x'))^2  \leq f(\text{prox}(c'\ell)(x');c,x) - f(\text{prox}(c\ell)(x);c,x)\\
        & \leq (f(\text{prox}(c'\ell)(x');c',x') - f(\text{prox}(c\ell)(x);c',x')) + (f(\text{prox}(c'\ell)(x');c,x) - f(\text{prox}(c'\ell)(x');c',x'))\\
        & \ \ \ + (f(\text{prox}(c\ell)(x);c',x') - f(\text{prox}(c\ell)(x);c,x))\\
        & \leq 0 + (c-c')\ell(\text{prox}(c'\ell)(x')) + \frac{1}{2}(\text{prox}(c'\ell)(x') - x)^2 - \frac{1}{2}(\text{prox}(c'\ell)(x') - x')^2\\
        & \ \ \ + (c'-c)\ell(\text{prox}(c\ell)(x)) + \frac{1}{2}(\text{prox}(c\ell)(x) - x')^2 - \frac{1}{2}(\text{prox}(c\ell)(x) - x)^2
    \end{align*}
    To get our desired result, we need to verify that this last expression goes to zero as $(c',x') \to (c,x)$. For this, it is clearly sufficient to show that $\ell(\text{prox}(c'\ell)(x'))$ and $\text{prox}(c'\ell)(x')$ remain bounded. Now, note that by the optimality of $\text{prox}(c'\ell)(x')$ we have 
\[
\frac{1}{2}(\text{prox}(c'\ell)(x') - x')^2 \leq  f(\text{prox}(c'\ell)(x'); c',x') \leq f(0; c',x') = c'\ell(0) + (x')^2.
\]
Thus, $\text{prox}(c'\ell)(x')$ is clearly bounded as $(c',x') \to (c,x)$. Since $\ell(\cdot)$ is continuous, this also shows that $\ell(\text{prox}(c'\ell)(x'))$ is bounded.

\end{proof}

Our final lemma allows us to translate convergence in probability to an analogous conditional statement. Once again, this result is well-known and we include a proof simply for completeness.

\begin{lemma}\label{lem:marg_to_cond_conv}
    Let $\{A_n,B_n,Z_n\}_{n=1}^{\infty} \subseteq \mmr \times \mmr \times \mathcal{Z}$ be a sequence of random variables such that $|A_n - B_n| \stackrel{\mmp}{\to} 0$. Then, for all $\delta > 0$,
    \[
    \mmp(|A_n - B_n| > \delta \mid Z_n) \stackrel{\mmp}{\to} 0.
    \]
\end{lemma}
\begin{proof}
By Markov's inequality, we have that for any $\rho > 0$,
\[
\mmp\left(\mmp\left(|A_n - B_n| > \delta  \mid Z_n \right) \geq \rho\right) \leq \frac{1}{\rho} \mmp(|A_n - B_n| > \delta ) \to 0,
\]
as desired.
\end{proof}

\end{document}